\documentclass[10pt]{article}

\usepackage{url}
\usepackage{mathtools}
\usepackage{amssymb}
\usepackage{amsthm}
\usepackage{empheq}
\usepackage{latexsym}
\usepackage{enumitem}
\usepackage{eurosym}
\usepackage{dsfont}
\usepackage{appendix}
\usepackage{color} 
\usepackage[unicode]{hyperref}
\usepackage{frcursive}
\usepackage[utf8]{inputenc}
\usepackage[T1]{fontenc}
\usepackage{geometry}
\usepackage{multirow}
\usepackage{todonotes}
\usepackage{lmodern}
\usepackage{anyfontsize}
\usepackage{stmaryrd}
\usepackage{natbib}
\usepackage{cleveref}

\setcitestyle{numbers,open={[},close={]}}

\definecolor{red}{rgb}{0.7,0.15,0.15}
\definecolor{green}{rgb}{0,0.5,0}
\definecolor{blue}{rgb}{0,0,0.7}
\hypersetup{colorlinks, linkcolor={red},citecolor={green}, urlcolor={blue}}
			
\makeatletter \@addtoreset{equation}{section}

\newtheorem{theorem}{Theorem}[section]
\newtheorem{assumption}[theorem]{Assumption}
\newtheorem{corollary}[theorem]{Corollary}

\newtheorem{lemma}[theorem]{Lemma}
\newtheorem{proposition}[theorem]{Proposition}

\newtheorem{definition}[theorem]{Definition}
\newtheorem{remark}[theorem]{Remark}

\DeclareUnicodeCharacter{014D}{\=o}
\def \E{\mathbb{E}}
\def \F{\mathbb{F}}
\def \G{\mathbb{G}}
\def \H{\mathbb{H}}

\def \M{\mathbb{M}}
\def \N{\mathbb{N}}

\def \P{\mathbb{P}}
\def \Q{\mathbb{Q}}
\def \R{\mathbb{R}}
\def \S{\mathbb{S}}

\def \Z{\mathbb{Z}}

\def \Pr{\mathrm{P}}

\def\Ac{{\cal A}}
\def\Bc{{\cal B}}
\def\Cc{{\cal C}}

\def\Fc{{\cal F}}
\def\Gc{{\cal G}}
\def\Hc{{\cal H}}

\def\Lc{{\cal L}}
\def\Mc{{\cal M}}

\def\Pc{{\cal P}}

\def\Rc{{\cal R}}
\def\Sc{{\cal S}}

\def\Uc{{\cal U}}

\def\Wc{{\cal W}}

\def\Fb{{\bar F}}
\def\Gb{{\overline \G}}
\def\Gcb{\overline \Gc}

\def\x{\times}

\def\Om{\Omega}
\def\Omt{\widetilde{\Omega}}

\def\om{\omega}

\def\Omb{\overline{\Om}}
\def\omb{\bar \om}
\def\Fb{\overline{\F}}
\def\Gcb{\overline{\Gc}}
\def\Fcb{\overline{\Fc}}
\def\Bt{\widetilde{B}}
\def\Gct{\widetilde{\Gc}}
\def\Fct{\widetilde{\Fc}}
\def\Ft{\widetilde{\F}}
\def\Gt{\widetilde{\G}}
\def\Wt{\widetilde{W}}
\def\Xt{\widetilde{X}}

\def\0{\mathbf{0}}

\def \Ert{\widetilde{\mathrm{E}}}

\def \bb{\mathbf{b}}
\def \mb{\mathbf{m}}
\def \alephb{\overline{\aleph}}
\def \alephh{\widehat{\aleph}}
\def \mbb{\overline{\mathbf{m}}}

\def \nb{\mathbf{n}}

\def \qb{\mathbf{q}}

\def \mub{\overline{\mu}}
\def \Lambdap{\Lambda^{\prime}}
\def \zetab{\overline{\zeta}}

\def \muh{\widehat{\mu}}
\def \mup{\mu^{\prime}}
\def \Pih{\widehat{\Pi}}

\def \zetah{\widehat{\zeta}}

\def \Qrt{\widetilde{\mathrm{Q}}}
\def \Art{\widetilde{\mathrm{A}}}
\def \Rrt{\widetilde{\mathrm{R}}}
\def \alphat{\widetilde{\alpha}}
\def \varphit{\widetilde{\varphi}}
\def \Qr{\mathrm{Q}}
\def \Pt{\widetilde{\P}}
\def \Prt{\widetilde{\Pr}}
\def \Kt{\widetilde{K}}

\def \Mct{\widetilde{\Mc}}
\def \St{\widetilde{S}}

\def \Zt{\widetilde{Z}}
\def \sigmat{\widetilde{\sigma}}
\def \gammat{\widetilde{\gamma}}

\def \Deltat{\widetilde{\Delta}}
\def \betat{\widetilde{\beta}}

\def \nub{\bar{\nu}}
\def \Xh{\widehat{X}}
\def \Zh{\widehat{Z}}
\def \pib{\overline{\pi}}

\def\normeL2#1{\left\|{#1}\right\|_{L^2}}

\def\Pcb{\overline \Pc}

\def \mubb {\boldsymbol{\mu}}

\def \Lim{\displaystyle\lim}

\def \Limsup{\displaystyle\limsup}

\def \Xbb{\mathbf{X}}

 \title{Mean Field Games of Controls: on the convergence of Nash equilibria \footnote{The author is grateful to Dylan Possama\"{i} and Xiaolu Tan for helpful comments and suggestions.}}

\author{
    Mao Fabrice {\sc Djete}\footnote{Universit\'e Paris--Dauphine, PSL Research University, CNRS, CEREMADE, 75016 Paris,France, djete@ce-remade.dauphine.fr. This work benefited from support of the r\'egion \^Ile--de--France.}
    }
             \date{\today}

\begin{document}

\maketitle
 
\begin{abstract}

{\color{black}In this paper, we provide convergence and existence results for mean field games of controls. Mean field games of controls are a class of mean field games where the mean field interactions are achieved through the joint (conditional) distribution of the controlled state and the control process. The framework we are considering allows to control the volatility coefficient $\sigma$, and the controls/strategies are supposed to be of $open\;loop$ type.
Using (controlled) Fokker--Planck equations, we introduce a notion of measure--valued solution of mean field games of controls and prove a relation between these solutions on the one hand, and the approximate Nash equilibria on the other hand.}
First of all, in the $N$--player game associated to the mean field games of controls, given a sequence of approximate Nash equilibria, it is shown that, this sequence admits limits as $N$ tends to infinity, and each limit is a measure--valued solution of the {\color{black}corresponding} mean field games of controls. Conversely, any measure--valued solution can be obtained as the limit of a sequence of approximate Nash equilibria of the $N$--player game. In other words, the measure--valued solutions are the {\color{black}accumulation} points of the approximate Nash equilibria. Then, by considering an approximate strong solution of mean field games of controls which is the classical strong solution where the optimality is obtained by admitting a small error $\varepsilon,$ we prove that the measure--valued solutions are the {\color{black}accumulation points} of this type of solutions when $\varepsilon$ goes to zero. Finally, the existence of a measure--valued solution of mean field games of controls is proved in the case without common noise. 


\end{abstract}
\section{Introduction}
Since the pioneering work of \citeauthor*{lasry2006jeux} \cite{lasry2007mean} and \citeauthor*{huang2003individual} \cite{huang2006large}, mean field games (MFG) have been the subject of intensive research in recent years. Due to the diversity of applications, particularly in models of oil production, volatility formation, population dynamics and economic growth (see \citeauthor*{carmona2018probabilisticI} \cite{carmona2018probabilisticI} for an overview), the study of MFG has attracted increasing interest in the field of applied mathematics.

\medskip
The MFG can be seen as symmetric stochastic differential games with {\color{black}infinitely} many players. Indeed, under the appropriate assumptions, a MFG solution can be used to construct approximate Nash equilibria for a game involving a large number of players. Also, it can be shown that each sequence of Nash equilibria converges towards a solution of MFG when the number of players tends to infinity. 

\medskip
So far, this study has been conducted considering that the interactions between the players are achieved only through the empirical distribution of the state processes. We refer to \citeauthor*{lacker2016general} \cite{lacker2016general} for a general analysis of this case (see also \citeauthor*{M-Fisher} \cite{M-Fisher}). 
The goal of this paper is to give a general analysis of the case where {\color{black} the mean field interaction occurs through the empirical distribution of both the states and the controls.}

\medskip
To briefly summarize the finite--player games, specified
in full details in Section \ref{sec:finite-player-game}, let us suppose {\color{black} that the} $N$--players have private state processes $\Xbb:=(\Xbb^1,\dots,\Xbb^N)$ given by the stochastic differential equations (SDEs) system
\begin{align*}
        \mathrm{d}\Xbb^{i}_t
        &= b\big(t,\Xbb^{i}_{t},\varphi^{N,\Xbb}_{t \wedge \cdot},\varphi^{N}_{t} ,\alpha^i_t \big) \mathrm{d}t 
        +
        \sigma \big(t,\Xbb^{i}_{t},\varphi^{N,\Xbb}_{t \wedge \cdot},\varphi^{N}_{t} ,\alpha^i_t \big) \mathrm{d}W^i_t
        +
        \sigma_0 \mathrm{d}B_t,~~t \in [0,T],
        \\
        \varphi^{N}_{t}&:= \frac{1}{N}\sum_{i=1}^N \delta_{(\Xbb^{i}_{t},\alpha^i_t )}~\mbox{and}~\varphi^{N,\Xbb}_{t}:= \frac{1}{N}\sum_{i=1}^N \delta_{\Xbb^{i}_{t}},
    \end{align*}
where $T>0$ is a fixed time horizon, $(B,W^1,\dots,W^N)$ are independent Brownian motions where $B$ is called the common noise, and $\alpha^i$ is the control of player $i$. Here, the specificity is that the dynamics $X^i$ of player $i$ depends on the empirical distribution $\varphi^{N}$ of states and controls of all players.
Given a strategy $(\alpha^1,\dots,\alpha^N),$ the reward to the player $i$ is 
\begin{align*}
    J^N_i(\alpha^1,\dots,\alpha^N)
    :=
    \E \bigg[
        \int_0^T L\big(t,\Xbb^{i}_{t},\varphi^{N,\Xbb}_{t \wedge \cdot},\varphi^{N}_{t} ,\alpha^i_t \big) \mathrm{d}t 
        + 
        g \big( \Xbb^{i}_{T}, \varphi^{N,\Xbb}_{T \wedge \cdot} \big)
        \bigg].
\end{align*}
The aim of each agent is to maximize this reward through a Nash equilibrium criterion. The strategy $(\alpha^1,\dots,\alpha^N)$ is an $\varepsilon_N$--Nash equilibrium if for any admissible control $\beta,$
\begin{align} \label{eq:opti-player_i}
    J^N_i(\alpha^1,\dots,\alpha^N)
    \ge
    J^N_i(\alpha^1,...,\alpha^{i-1},\beta,\alpha^{i+1},\dots,\alpha^N) - \varepsilon_N.
\end{align}

Following the intuition, if $N$ is large, because of the symmetry of the model, the contribution of player $i$ and its control over $\varphi^{N}$ is negligible, and everything happens as if $\varphi^{N}$ was fixed in the optimization \eqref{eq:opti-player_i}. This line of argument leads to the derivation of the problem called in the literature the mean field games of controls or $extended$ mean field games, which has, loosely speaking, the following structure (the precise definition is given in Section \ref{section:strong_formulation}): a $(\sigma\{B_s,s \le t\})_{t \in [0,T]}$--adapted measure--valued process $(\mub^\star_t)_{t \in [0,T]}$ is an $\varepsilon$--strong MFG solution (or approximate strong MFG) if for all $t \in [0,T],$ $\mub^\star_t=\Lc(X^\star_t,\alpha^\star_t|B),$ where the state process $X^\star$ is governed by
\begin{align*}
		\mathrm{d}X^\star_t
		&= 
		b \big(t, X^\star_t,(\mu^\star_s)_{s \le t}, \mub^\star_t, \alpha^\star_t \big) \mathrm{d}t
		+
	    \sigma\big(t, X^\star_t,(\mu^\star_s)_{s \le t}, \mub^\star_t, \alpha^\star_t \big) \mathrm{d} W_t
		+ 
		\sigma_0 \mathrm{d}B_t,,~~t \in [0,T]
		\\
		\mu^\star_t&:=\Lc(X^\star_t|B),
	\end{align*}
	and one has
	\begin{align} \label{def-epsilon-MFG}
		    \begin{cases} \displaystyle \E \bigg[
				\int_0^T L(t, X^{\star}_t,\mu^\star_{t \wedge \cdot}, \mub^{\star}_t,\alpha^\star_t) \mathrm{d}t 
				+ 
				g(X^{\star}_T, \mu^{\star}) 
			\bigg]
			\ge
			\sup_{\alpha}\;\E \bigg[
				\int_0^T L(t, X_t,\mu^\star_{t \wedge \cdot}, \mub^{\star}_t,\alpha_t) \mathrm{d}t 
				+ 
				g(X_T, \mu^{\star}) 
			\bigg] - \varepsilon,
\\[0.7em]
		    \displaystyle  \mbox{where  {\color{black} the optimization is over the solutions}}\;\; 
		    \mathrm{d}X_t
		    = 
		    b \big(t, X_t,(\mu^\star_s)_{s \le t}, \mub^\star_t, \alpha_t \big) \mathrm{d}t
		    +
	        \sigma\big(t, X_t,(\mu^\star_s)_{s \le t}, \mub^\star_t, \alpha_t \big) \mathrm{d} W_t
		    + 
		    \sigma_0 \mathrm{d}B_t.
		    \end{cases}
		\end{align} 

This structure means that, when the process $(\mub^\star_t)_{t \in [0,T]}$ is fixed, a single representative player solves an
optimal control problem. The condition $\mub^\star_t=\Lc(X^\star_t,\alpha^\star_t|B),$ called $consistency\;condition$ or $fixed\;point\;problem$ in the literature, gives to $(X^\star,\alpha^\star)$, the $\varepsilon$--optimal control, a representation property of the entire population. {\color{black}The process} $  (\mub^\star_t)_{t \in [0,T]}$ can be seen as an equilibrium. This is exactly the classical MFG problem except for two aspects: first the solution is a (conditional) distribution of the state and the control $(\Lc(X^\star_t,\alpha^\star_t|B))_{t \in [0,T]}$ and not just a (conditional) distribution of the state $(\Lc(X^\star_t|B))_{t \in [0,T]},$ next, $(X^\star,\alpha^\star)$ is an $\varepsilon$--optimal control and not necessary an optimal control (or $0$--optimal control).

\medskip
For MFG of controls or $extended$ MFG, the literature on this topic focus primarily on the existence and uniqueness results of the limit problem (with $\varepsilon = 0$), usually without common noise i.e. $\sigma_0=0.$ 
\citeauthor*{DiogoVardan-ExtMFG} \cite{DiogoVardan-ExtMFG}, by using PDE methods, study these types of interactions
in the deterministic case i.e. $\sigma=\sigma_0=0.$  Strong assumptions of continuity and convexity make it possible to obtain the existence and the regularity of the solutions. In order to explore a problem of optimal liquidation in finance, \citeauthor*{Lehalle-card} \cite{Lehalle-card} apply similar PDE techniques for this problem in the case without common noise, while allowing $\sigma$ to be non--zero. With the same philosophy, \citeauthor*{Kobeissi} \cite{Kobeissi} provides some results and discusses properties of existence and uniqueness in examples. Let us also mention \citeauthor*{Kobeissi-Y_Achdou} \cite{Kobeissi-Y_Achdou} which gives numerical approximations via finite difference for the PDE system arising in the MFG of controls. {\color{black}See also \citeauthor*{GOMES201449} \cite{GOMES201449} and \citeauthor*{bonnans_pfeifer2019} \cite{bonnans_pfeifer2019} for other investigations of PDE techniques in the situation of MFG of controls.} 

\medskip
Probability techniques have also been used to give some results for the limit problem. Without common noise, using a weak formulation of the MFG of controls, \citeauthor*{carmona2015probabilistic} \cite{carmona2015probabilistic} obtain the existence and uniqueness of the MFG of controls, and from this solution, construct an approximate Nash equilibrium, all this by imposing an uncontrolled and non--degenerative volatility $\sigma$ ($\sigma>0).$ They illustrate their results on the price impact models (which share some similarities with those considered in \cite{Lehalle-card}) and the flocking model. Similarly, \citeauthor*{P-JamesonGraber} \cite{P-JamesonGraber}, for the studies of models of production of an exhaustible resource, solves similar existence and uniqueness problems.

\medskip
Except the recent work of \citeauthor*{M_Lauriere-Tangpi} \cite{M_Lauriere-Tangpi} which, using the FBSDEs, treats the convergence of Nash equilibria in the framework of MFG of controls, to the best of our knowledge, there are no other papers using probabilistic or PDE methods that answer the question of the convergence of $\varepsilon_N$--Nash equilibria to the MFG solution in this context. 
The assumptions of regularity of coefficients usually used in the literature are no longer verified in the presence of the distribution of control (see an example of this phenomena in \cite[Remark 2.4]{MFD-2020}).  Although using probabilistic point of view, the approach developed in this paper is very different from these previously mentioned, and considers very general assumptions. Despite many differences, this article is in the same spirit as \cite{lacker2016general}, which is, in the framework without law of control and with open loop controls, the most significant paper investigating the connection between large population differential games and MFG under very general assumptions. We want to emphasize that the interesting techniques developed in \cite{lacker2016general} do not work in the case of MFG of controls. As previously mentioned, {\color{black} because} of the presence of the law of control, the assumptions of regularity on the coefficients are no longer verified.

\medskip
In order to solve the difficulty generated by the empirical distribution of controls, we introduce the notion of $measure$--$valued$ MFG equilibrium. This notion is precisely defined in Section \ref{section:relaxed_formulation}. The idea of our notion comes from the (stochastic) Fokker--Planck equation verified by the pair $(\mu^\star,\mub^\star).$ This notion of MFG solution is very close to the classical notion. The main difference is that {\color{black}the optimization is taken over all solutions of specific controlled Fokker--Planck equations} and not to a solution of a controlled SDE. This notion {\color{black}has} already been considered in the literature by \citeauthor*{cardaliaguet2015master} \cite[Section 3.7.]{cardaliaguet2015master} and, in some way, by \citeauthor*{Lacker-closedloop} \cite{Lacker-closedloop}. 
Borrowing techniques from \cite{lacker2016general}, under suitable assumptions, we prove that the sequence of empirical measure flows $(\varphi^{N,\Xbb},\varphi^{N})_{N \in \N^*}$ is tight in a suitable space, and with the help of techniques introduced in our companion paper \cite{MFD-2020}, we show that every limit in distribution is a $measure$--$valued$ mean field equilibrium. And conversely, for each $measure$--$valued$ mean field equilibrium, we construct an approximate Nash equilibrium which has this $measure$--$valued$ mean field equilibrium as limit. 
In addition to these convergence results, this article provides an $\varepsilon$--strong existence and another approximation not taken into account until now. Similarly to the approximate Nash equilibrium, when $\varepsilon$ is positive and goes to zero, the sequence  $(\mu^\star,\mub^\star)$ is tight with any limit being a $measure$--$valued$ MFG equilibrium. Also, when there is common noise, any $measure$--$valued$ MFG equilibrium can be approached by a sequence of $\varepsilon$--strong MFG equilibrium $(\mu^\star,\mub^\star)$ verifying \eqref{def-epsilon-MFG}. 

\medskip
Consequently, there is a perfect symmetry between approximate Nash equilibria and $\varepsilon$--strong MFG equilibria. Besides, our notion of $measure$--$valued$ MFG equilibrium is the accumulation points of approximate Nash equilibria and $\varepsilon$--strong MFG equilibria. Therefore, if there exists a $measure$--$valued$ MFG equilibrium or an approximate Nash equilibrium, there is necessarily an $\varepsilon$--strong MFG equilibrium for any $\varepsilon>0.$ Without common noise, with similar arguments to \citeauthor*{LackerMFGcontrolledmartingale} \cite{LackerMFGcontrolledmartingale}, we show that there is a $measure$--$valued$ MFG equilibrium under general condition, as a result there is a $\varepsilon$--strong MFG equilibrium. 

\medskip
It is well known in the MFG theory that the existence of a strong MFG solution is very difficult to obtain and requires strong assumptions. Admitting a small error $\varepsilon>0,$ it is possible to get the existence of an $\varepsilon$--strong MFG equilibrium under general assumptions. It is worth emphasizing that our results allow to handle the case where $\sigma$ is controlled i.e. the control $\alpha$ appears in the function $\sigma.$ There are not many works that look at the situation where the volatility is controlled.
Let us also mention, in this paper, despite general assumptions considered, we are limited by some conditions that we must have for technical reasons, a separability condition on $(b,\sigma,L)$ (see \Cref{assum:main1}) and a non--degeneracy volatility condition of type $\sigma \sigma^{\top}> 0.$ 

\medskip
The rest of the paper is organized as follows. After introducing some notations, we provide respectively in \Cref{sec:finite-player-game} and \Cref{section:MFGcontrols}, the definition of the $N$--player game and the corresponding MFG of controls. The main limit Theorem \ref{thm:limitNashEquilibrium} and, its converse, Theorem \ref{thm:converselimitNash} are then stated in \Cref{subsec:main-results}. Later, in Section \ref{sec:no-commonnoise}, we present some existence and approximation results in the particular case without common noise. 
Most of the technical proofs are completed in Section \ref{sec:proofs} and Section \ref{sec:proofThm_existence}.



\vspace{0.5em}

{\bf \large Notations}.
	$(i)$
	Given a Polish space $(E,\Delta)$, $p \ge 1,$ we denote by $\Pc(E)$ the collection of all Borel probability measures on $E$,
	and by $\Pc_p(E)$ the subset of Borel probability measures $\mu$ 
	such that $\int_E \Delta(e, e_0)^p  \mu(de) < \infty$ for some $e_0 \in E$.
	We equip $\Pc_p(E)$ with the Wasserstein metric $\Wc_p$ defined by
	\[
		\Wc_p(\mu , \mu') 
		~:=~
		\bigg(
			\inf_{\lambda \in \Lambda(\mu, \mu')}  \int_{E \x E} \Delta(e, e')^p ~\lambda( \mathrm{d}e, \mathrm{d}e') 
		\bigg)^{1/p},
	\]
	where $\Lambda(\mu, \mu')$ denote the collection of all probability measures $\lambda$ on $E \x E$ 
	such that $\lambda( \mathrm{d}e, E) = \mu(\mathrm{d}e)$ and $\lambda(E,  \mathrm{d}e') = \mu'( \mathrm{d}e')$. Equipped with $\Wc_p,$ $\Pc_p(E)$ is a Polish space (see \cite[Theorem 6.18]{villani2008optimal}). For any  $\mu \in \Pc(E)$ and $\mu$--integrable function $\varphi: E \to \R,$ we define
	\begin{align*}
	    \langle \varphi, \mu \rangle
	    =
	    \langle \mu, \varphi \rangle
	    :=
	    \int_E \varphi(e) \mu(\mathrm{d}e),
	\end{align*}
	and for another metric space $(E^\prime,\Delta^\prime)$, we denote by $\mu \otimes \mu^\prime \in \Pc(E \x E')$ the product probability of any $(\mu,\mu^\prime) \in \Pc(E) \x \Pc(E^\prime)$.
	Given a probability space $(\Om, \Fc, \P)$ supporting a sub $\sigma$--algebra $\Gc \subset \Fc,$ then for a Polish space $E$ and any random variable $\xi: \Om \longrightarrow E$, both the notations $\Lc^{\P}( \xi | \Gc)(\om)$ and $\P^{\Gc}_{\om} \circ (\xi)^{-1}$ are used to denote the conditional distribution of $\xi$ knowing $\Gc$ under $\P$.
	
\medskip	
	
	
	
\medskip
	\noindent $(ii)$	
	For any $(E,\Delta)$ and $(E',\Delta')$ two Polish spaces, we use $C_b(E,E')$ to denote the set of continuous functions $f$ from $E$ into $E'$ such that $\sup_{e \in E} \Delta'(f(e),e'_0) < \infty$ for some $e'_0 \in E'$.
	Let $\N^*$ denote the set of positive integers. Given non--negative integers $m$ and $n$, we denote by $\S^{m \x n}$ the collection of all $m \x n$--dimensional matrices with real entries, equipped with the standard Euclidean norm, which we denote by $|\cdot|$ regardless of the dimensions. 
	We also denote $\S^n:=\S^{n \times n}$, and by $\mathrm{I}_n$ the identity matrix in $\S^n$. For any matrix $a \in \S^{n}$ which is symmetric positive semi--definite, we write $a^{1/2}$ the unique symmetric positive semi--definite square root of the matrix $a.$
	Let $k$ be a positive integer, we denote by $C^k_b(\R^n;\R)$ the set of bounded maps $f: \R^n \longrightarrow \R$, having bounded continuous derivatives of order up to and including $k$. 
	Let $f: \R^n \longrightarrow \R$ be twice differentiable, we denote by $\nabla f$ and $\nabla^2f$ the gradient and the Hessian of $f$ respectively.


\medskip
    \noindent $(iii)$
	Let $T > 0$, and $(\Sigma,\rho)$ be a Polish space, we denote by $C([0,T], \Sigma)$ the space of all continuous functions on $[0,T]$ taking values in $\Sigma$.
	Then $C([0,T], \Sigma)$ is a Polish space under the uniform convergence topology, and we denote by $\|\cdot\|$ the uniform norm. 
	When $\Sigma=\R^k$ for some $k\in\N$, we simply write $\Cc^k (:= C([0,T], \R^k)),$ also we shall denote by $\Cc^{k}_{\Wc}:=C([0,T], \Pc(\R^k)),$ and for $p \ge 1,$ $\Cc^{k,p}_{\Wc}:=C([0,T], \Pc_p(\R^k)).$ 

\medskip
	With a Polish space $E$, we denote by $\M(E)$ the space of all Borel measures $q( \mathrm{d}t,  \mathrm{d}e)$ on $[0,T] \x E$, 
	whose marginal distribution on $[0,T]$ is the Lebesgue measure $ \mathrm{d}t$, 
	that is to say $q( \mathrm{d}t, \mathrm{d}e)=q(t,  \mathrm{d}e) \mathrm{d}t$ for a family $(q(t,  \mathrm{d}e))_{t \in [0,T]}$ of Borel probability measures on $E$.
	{\color{black}We will denote by $\Lambda$ the canonical element of $\M(E)$ and we introduce}
	\begin{equation}\label{eq:lambda}
	\Lambda_{t \wedge \cdot}(\mathrm{d}s, \mathrm{d}e) :=  \Lambda(\mathrm{d}s, \mathrm{d}e) \big|_{ [0,t] \x E} + \delta_{e_0}(\mathrm{d}e) \mathrm{d}s \big|_{(t,T] \x E},\; \text{for some fixed $e_0 \in E$.}
	\end{equation}
	For $p \ge 1$, we use $\M_p(E)$ to designate the elements of $q \in \M(E)$ such that $q/T \in \Pc_p(E \x [0,T]).$


\section{Mean field games of controls (with common noise): Setup and main results} \label{sec:set-resutls}

In this section, we first introduce the $N$--player game, and the definition of $\epsilon_N$--Nash equilibria. Next, we formulate the notions of approximate strong and measure--valued MFG solutions which will be essential to describe the limit of the Nash equilibria. 

\medskip
The general assumptions used throughout this paper are now formulated. The dimensions $(n,\ell) \in \N^*\times\N,$ the nonempty Polish space $(U,\rho)$ and the horizon time $T>0$ are fixed and $\Pc^n_U$ denotes the space of all Borel probability measures on $\R^n \x U$ i.e. $\Pc^n_U:=\Pc(\R^n \x U).$ {\color{black} Also, we set $p \ge 2,$ $\nu \in \Pc_{p'}(\R^n)$ with $p'>p,$ and the probability space $(\Om,\H:=(\Hc_t)_{t \in [0,T]},\Hc,\P)$\footnote{The probability space $(\Om,\H,\P)$ contains as many random variables as we want in the sense that: each time we need a sequence of independent uniform random variables or Brownian motions, we can find them on $\Om$ without mentioning an enlarging of the space. }. } We are given the following Borel measurable functions
	\[
		\big[b, \sigma,L \big]:[0,T] \x \R^n \x \Cc^n_{\Wc}  \x \Pc^n_U \x U \longrightarrow \R^n \x \S^{n \x n} \x \R\;\mbox{and}\; 
		g: \R^n \x \Cc^n_{\Wc} \longrightarrow \R.
	\]
		

    \begin{assumption} \label{assum:main1} 
        $[b,\sigma,L]$ are Borel measurable in all their variables, and non--anticipative in the sense that, for all $(t, x, u, \pi, m) \in [0,T] \x \R^n \x U \x \Cc^n_{\Wc} \x \Pc^n_U$
	\[
		\big[b, \sigma,L \big] (t,x, \pi, m, u) 
		=
		\big[b, \sigma, L \big] (t,x, \pi_{t \wedge \cdot}, m, u).
	\]
        Moreover, there is positive constant $C$ such that
		
		
		$(i)$ $U$ is a compact nonempty polish set;
		
		$(ii)$ $b$ and $\sigma$ are bounded continuous functions, and $\sigma_0 \in \S^{n \x \ell}$ is a constant;
		
		\medskip
		$(iii)$ for all $(t, x,x',\pi,\pi',m,m',u) \in [0,T] \x \R^n \x \R^n \x \Cc^n_{\Wc} \x \Cc^n_{\Wc} \x \Pc^n_U \x \Pc^n_U \x U,$ one has
		\begin{align*}
			\big| [b,\sigma](t,x,\pi,m,u)
			-
			[b,\sigma](t,x',\pi',m',u)
			\big |
			~\le~
			C \Big( | x-x' | + \sup_{s \in [0,T]}\Wc_p(\pi_s,\pi'_s) + \Wc_p (m,m')\Big);
		\end{align*}

    \medskip
		$(iv)$ $\underline{Non\mbox{--}degeneracy\;condition}$: for some constant $\theta >0$, one has, for all $(t,x,\pi,m,u) \in [0,T] \x \R^n \x \Cc^n_{\Wc} \x \Pc^n_U \x U$,
		\begin{align*}
		    ~\theta \mathrm{I}_{n} \le ~\sigma \sigma^{\top}(t,x,\pi,m, u);
		\end{align*}
		
    \medskip
		$(v)$ the reward functions  $L$ and $g$ are continuous,
		and for all $(t,x,\pi,m,u) \in [0,T] \x \R^n \x \Cc^n_{\Wc} \x \Pc^n_U \x U$, one has
		\[  
			\big|L(t,x,\pi,m,u) \big| + |g(x,\pi)|
			\le  
			C \bigg [ 1+ | x |^{p} + \sup_{s \in [0,T]} \Wc_p(\pi_s,\delta_{0})^{p} + \int_{\R^n} | x'|^{p} m(\mathrm{d}x',U)   \bigg ];
		\]
		
		$(vi)$ $\underline{Separability\;condition}$: There exist continuous functions $(b^\circ, b^\star,a^\circ, a^\star,L^\circ,L^\star)$ satisfying
		\begin{align*}
		    [b,\sigma \sigma^\top ] ( t,x,\pi,m, u )
		    :=
		    [b^\star,a^\star] (t,\pi,m ) 
		    +
		    [b^\circ,a^\circ] ( t,x,\pi, u )\;\;\mbox{and}\;\;L ( t,x,\pi,m, u )
		    :=
		    L^\star(t,x,\pi,m ) 
		   +
		   L^\circ( t,x,\pi, u ),
		\end{align*}
		for all $(t,x,\pi,m, u) \in [0,T] \x \R^n \x \Cc^n_{\Wc} \x \Pc^n_U \x U.$
		
	\end{assumption}
	
	\begin{remark}
	{
	    Most of these assumptions are classical in the study of mean field games and control problems $($see {\rm\citeauthor*{lacker2016general} \cite{lacker2016general}, \citeauthor*{djete2019general} \cite{djete2019general}} and {\rm\citeauthor{MFD-2020} \cite{MFD-2020}} $)$. Only the ``separability condition'' and the ``non--degeneracy condition'' can be seen as non--standard. However, in the context of mean field games of controls, these conditions are used by many authors, for instance {\rm\citeauthor{Lehalle-card} \cite{Lehalle-card}}$($only separability condition$)$, {\rm\citeauthor*{carmona2015probabilistic} \cite{carmona2015probabilistic}} and {\rm\citeauthor*{M_Lauriere-Tangpi} \cite{M_Lauriere-Tangpi}}. These are essentially technical assumptions. 
	   }
	\end{remark}


\subsection{{\color{black}The $N$-player game}} \label{sec:finite-player-game}

{\color{black}
    
    On the probability space $(\Om,\H,\Hc,\P),$ let $(W^i)_{i \in \N^*}$ be a sequence of independent $\R^n$--valued $\H$--Brownian motions, $B$ be a $\R^\ell$--valued $\H$--Brownian motion and $(\xi^i)_{i \in \N^*}$ be a sequence of iid $\R^n$--valued $\Hc_0$--random variables of law $\nu.$ Besides $(W^i)_{i \in \N^*},$ $B$ and $(\xi^i)_{i \in \N^*}$ are independent.
	Let $\F^N = (\Fc^N_s)_{0 \le s \le T}$ be defined by the $\P$--completion of $\Ft^N:=(\Fct^N_t)_{t \in [0,T]}$ where
	$$
		\Fct^N_s 
		:=
		\sigma \big\{\xi^i, W^i_r, B_r, ~r \in [0,s],\;1 \le i \le N \big\}, 0 \le s \le T.
	$$
}	
	Let us denote by $\Ac^N$ the collection of all $U$--valued processes $\alpha = (\alpha_s)_{0 \le s \le T}$ which are $\F^N$-predictable.
	Then given a control rule/strategy $\overline{\alpha}:=(\alpha^1,\dots,\alpha^N) \in \Ac^N$, denote by $\Xbb_{\cdot}[\overline{\alpha}]:=(\Xbb^{1}_{\cdot}[\overline{\alpha}],\dots,\Xbb^{N}_{\cdot}[\overline{\alpha}])$ the unique strong solution of the following system of SDEs ({\color{black} the well--posedness is assured by \Cref{assum:main1}}): for each $i \in \{1,\dots,N\},$ $\E[\|\Xbb^i\|^p]< \infty,$
    \begin{align} \label{eq:N-agents_StrongMF_CommonNoise}
        \mathrm{d}\Xbb^{i}_t[\overline{\alpha}]
        =
        b\big(t,\Xbb^{i}_t[\overline{\alpha}],{\color{black}\varphi^{N,\Xbb,\overline{\alpha}}},\varphi^{N,\overline{\alpha}}_{t} ,\alpha^i_t \big) \mathrm{d}t 
        +
        \sigma \big(t,\Xbb^{i}_t[\overline{\alpha}],{\color{black}\varphi^{N,\Xbb,\overline{\alpha}}},\varphi^{N,\overline{\alpha}}_{t} ,\alpha^i_t \big) \mathrm{d}W^i_t
        +
        \sigma_0 \mathrm{d}B_t\;\mbox{with}\;\Xbb^i_0=\xi^i
    \end{align}
	where 
	\[
	    \varphi^{N,\overline{\alpha}}_{t}(\mathrm{d}x,\mathrm{d}u) := \frac{1}{N}\sum_{i=1}^N \delta_{\big(\Xbb^{i}_t[\overline{\alpha}],\;\alpha^i_t \big)}(\mathrm{d}x,\mathrm{d}u)
	    ~
	    \mbox{and}
	    ~
	    \varphi^{N,\Xbb,\overline{\alpha}}_{t}(\mathrm{d}x) := \frac{1}{N}\sum_{i=1}^N \delta_{\Xbb^{i}_t[\overline{\alpha}]} (\mathrm{d}x)~~
	    \mbox{for all}
	    ~~
	    t \in [0,T].
	\] 
	The reward value of player $i$ associated with control rule/strategy $\overline{\alpha}:=(\alpha^1,\dots,\alpha^N)$ is then defined by
	\[
	    J_i[\overline{\alpha}]
	    :=
	    \E \bigg[
        \int_0^T L\big(t,\Xbb^{i}_t[\overline{\alpha}],\varphi^{N,\Xbb,\overline{\alpha}},\varphi^{N,\overline{\alpha}}_{t} ,\alpha^i_t \big) \mathrm{d}t 
        + 
        g \big( \Xbb^{i}_T[\overline{\alpha}], \varphi^{N,\Xbb,\overline{\alpha}} \big)
        \bigg],
	\]
	and for $\beta \in \Ac^N$, one introduces the strategy $(\overline{\alpha}^{[-i]},\beta) \in \Ac^N$ by
	\[
	    (\overline{\alpha}^{[-i]},\beta)
	    :=
	    \big(\alpha^1,\dots,\alpha^{i-1},\beta,\alpha^{i+1},\dots,\alpha^N \big).
	\]
	\begin{definition}
	    For any $\varepsilon:=(\varepsilon_1,\dots,\varepsilon_N) \in (\R_{+})^N,$ $\overline{\alpha}$ is a $\varepsilon$--{\color{black}$($open loop$)$} Nash equilibrium if 
	\[
	    J_i[\overline{\alpha}] \ge \sup_{\beta \in \Ac^N} J_i\big((\overline{\alpha}^{[-i]},\beta) \big)-\varepsilon_i,\;\mbox{for each}\;i \in \{1,\dots,N\}.
	\]
	\end{definition}

\subsection{Mean field games of controls} \label{section:MFGcontrols}

\subsubsection{$\varepsilon$--Strong mean field game equilibrium}
\label{section:strong_formulation}
{\color{black}
We now formulate the classical MFG problem with common noise including the (conditional) law of control.
	On the probability space $(\Om,\H,\Hc,\P),$ let $(W,B)$ be $\R^n \x \R^\ell$--valued $\H$--Brownian motion and $\xi$ be a $\R^n$--valued $\Hc_0$--random variables of law $\nu.$
	Let $\F = (\Fc_s)_{0 \le s \le T}$ and $\G = (\Gc_s)_{0 \le s \le T}$ be defined by the $\P$--completion of $\Ft:=(\Fct_t)_{t \in [0,T]}$ and $\Gt:=(\Gct_t)_{t \in [0,T]}$ where
	$$
		\Fct_s 
		:=
		\sigma 
		\big\{\xi, W_r, B_r, ~r \in [0,s] \big\}
		~~~~~
		\mbox{and}
		~~~~~~~
		\Gct_s 
		:= 
		\sigma \big\{ B_r, ~r \in [0,s] \big\}.
	$$
}
	Let us denote by $\Ac$ the collection of all $U$--valued $\F$--predictable processes.
	Then, given $\alpha \in \Ac$, let $X^{\alpha}$ be the unique strong solution of the SDE (e.g. \cite[Theorem A.3]{djete2019mckean}): $\E[\|X^{\alpha} \|^p]< \infty,$
	and
	\begin{align} \label{eq:MKV_strong}
		\mathrm{d}X^{\alpha}_t
		= 
		b \big(t, X^{\alpha}_t,{\color{black}\mu^\alpha}, \mub^{\alpha}_t, \alpha_t \big) \mathrm{d}t
		+
		\sigma\big(t, X^{\alpha}_t, {\color{black}\mu^\alpha}, \mub^{\alpha}_t, \alpha_t \big) \mathrm{d} W_t
		+ 
		\sigma_0 \mathrm{d}B_t\;\mbox{with}\;X^\alpha_0=\xi
	\end{align}
	where $\mub^{\alpha}_t:=\Lc\big(X^{\alpha}_t, \alpha_t \big| \Gc_t \big)$ and $\mu^\alpha_t:=\Lc\big(X^{\alpha}_t \big| \Gc_r \big)$ for all $t \in [0,T].$
	
\medskip	
	Given $\alpha \in \Ac,$ and $X^{\alpha}$ solution of \eqref{eq:MKV_strong}, for every $\alpha' \in \Ac,$ let us introduce the unique strong  solution $X^{\alpha,\alpha'}$ of: $\E[\|X^{\alpha,\alpha'} \|^p]< \infty,$ and
	\begin{align} \label{eq:FixedMKV_strong}
		\mathrm{d}X^{\alpha,\alpha'}_t
		= 
		b \big(t, X^{\alpha,\alpha'}_t,\mu^\alpha, \mub^{\alpha}_t, \alpha'_t \big) \mathrm{d}t
		+
		\sigma\big(t, X^{\alpha,\alpha'}_t, \mu^\alpha, \mub^{\alpha}_t, \alpha'_t \big) \mathrm{d} W_t
		+ 
		\sigma_0 \mathrm{d}B_t\;\mbox{with}\;X^{\alpha,\alpha'}_0=\xi
	\end{align}
	and the {\color{black} reward function $\Psi$}
	\begin{align} \label{def:strong_value-function}
	    \Psi(\alpha,\alpha')
	    :=
	    \E\bigg[
				\int_0^T L(t, X^{\alpha,\alpha'}_t,\mu^\alpha, \mub^{\alpha}_t,\alpha'_t) \mathrm{d}t 
				+ 
				g(X^{\alpha,\alpha'}_T, \mu^{\alpha}) 
			\bigg].
	\end{align}
	
	\begin{definition}
	    For any $\varepsilon \in [0,\infty),$ we say that $\alpha$ is an $\varepsilon$--strong MFG equilibrium, if 
	\begin{align} \label{eq:optimality-strong}
	    \Psi(\alpha,\alpha)
	    \ge
	    \sup_{\alpha' \in \Ac} \Psi(\alpha,\alpha') - \varepsilon.
	\end{align}
	\end{definition}
	
	\medskip
	For all $(\alpha,\alpha') \in \Ac \x \Ac,$ let us define
	\begin{align*}
	    \mathrm{P}^{\alpha,\alpha'}
	    :=
	    \P \circ \Big(\mu^{\alpha,\alpha'}, \mu^{\alpha}, {\color{black}\delta_{\mub^{\alpha,\alpha'}_t}(\mathrm{d}m) \mathrm{d}t, \delta_{\mub^\alpha_t}(\mathrm{d}m) \mathrm{d}t}, B \Big)^{-1}
	\end{align*}
	where $\mub^{\alpha,\alpha'}_t:=\Lc\big(X^{\alpha,\alpha'}_t, \alpha'_t \big| \Gc_t \big)$ and  $\mu^{\alpha,\alpha'}_t:=\Lc\big(X^{\alpha,\alpha'}_t \big| \Gc_t \big)$ with $t \in [0,T],$ and
	\begin{align} \label{eq:McV-measure-v}
	    \mathrm{P}^\alpha
	    :=
	    \P \circ \Big(\mu^{\alpha}, \mu^{\alpha}, {\color{black}\delta_{\mub^\alpha_t}(\mathrm{d}m) \mathrm{d}t, \delta_{\mub^\alpha_t}(\mathrm{d}m) \mathrm{d}t}, B \Big)^{-1}.
	\end{align}
	
	$\Pcb_S$ and for each $\varepsilon \in [0,\infty),$ $\Pcb^\star_S[\varepsilon]$ denote the subsets of $\Pc\big(\Cc^n_{\Wc} \x \Cc^n_{\Wc} \x \M(\Pc^n_U) \x \M(\Pc^n_U) \x \Cc^\ell\big)$ defined as follows
	\begin{align*}
	    \Pcb_S
	    :=
	    \big\{
	        \mathrm{P}^{\alpha,\alpha'},\;\;\mbox{with}\;(\alpha,\alpha') \in \Ac \x \Ac
	    \big\}\;\;\;\mbox{and}\;\;\;
	    \Pcb^\star_S[\varepsilon]
	    :=
	    \big\{
	        \mathrm{P}^\alpha,\;\;\mbox{with}\;\alpha\;\mbox{is a}\;\varepsilon\mbox{--strong MFG equilibrium}
	    \big\}.
	\end{align*}
	{\color{black}In other words, $\Pcb_S$ is the subset of all {\color{black} distributions of} controlled McKean--Vlasov processes of type \eqref{eq:MKV_strong}, and $\Pcb^\star_S[\varepsilon]$ is all $\varepsilon$--strong MFG equilibrium. In what follows, the use of these forms of sets will become clearer.}
	

    \subsubsection{Measure--valued MFG equilibrium} \label{section:relaxed_formulation}
    
    {\color{black} Notice that, for each $(\alpha,\alpha') \in \Ac \x \Ac,$ the couple $(\mu^{\alpha,\alpha'},\mub^{\alpha,\alpha'})$ satisfies the Fokker--Planck equation: for each $f \in C_b^2(\R^n),$
    \begin{align*} 
        \mathrm{d}\langle f(\cdot-\sigma_0 B_t),\mu^{\alpha,\alpha'}_t \rangle
        =
	    &\langle \nabla f(\cdot-\sigma_0 B_t)^\top b(t,\cdot,\mu^\alpha,\mub^\alpha_t,\cdot), \mub^{\alpha,\alpha'}_t \rangle \mathrm{d}t + \frac{1}{2} \langle \mathrm{Tr}[\sigma \sigma^\top(t,\cdot,\mu^\alpha,\mub^\alpha_t,\cdot) \nabla^2f(\cdot-\sigma_0B_t)], \mub^{\alpha,\alpha'}_t \rangle \mathrm{d}t\;\;a.s.
    \end{align*}

    }
    {\color{black}Inspired by the Fokker--Planck equation satisfied by the couple $(\mu^{\alpha,\alpha'},\mub^{\alpha,\alpha'})$ (see also the discussion in \cite{MFD-2020}}), we carefully formulate the notion of measure--valued control rules which is essential for the notion of measure--valued MFG equilibrium that will be introduced just after.

\paragraph{Measure--valued control rules}

    Denote by $\M := \M \big(\Pc^n_U \big)$ the collection of all finite (Borel) measures $q(\mathrm{d}t, \mathrm{d}e)$ on $[0,T] \x \Pc^n_U$, 
	whose marginal distribution on $[0,T]$ is the Lebesgue measure $\mathrm{d}s$
	i.e. $q(\mathrm{d}s,\mathrm{d}e)=q(s,\mathrm{d}e)\mathrm{d}s$ for a measurable family $(q(s, \mathrm{d}e))_{s \in [0,T]}$ of Borel probability measures on $\Pc^n_U.$
	Let $\Lambda$ be the canonical element on $\M$.
	We then introduce the canonical filtration $\F^\Lambda = (\Fc^\Lambda_t)_{0 \le t \le T}$ on $\M$ by
	$$
		\Fc^\Lambda_t := \sigma \big\{ \Lambda(C \x [0,s]) ~: \forall s \le t, C \in \Bc(\Pc^n_U) \big\}.
	$$
	For each $q \in \M$, one has the disintegration property: $q(\mathrm{d}t, \mathrm{d}e) = q(t, \mathrm{d}e) \mathrm{d}t$, and there is a version of the disintegration
	such that $(t, q) \mapsto q(t, \mathrm{d}e)$ is $\F^\Lambda$-predictable.
	

    \vspace{4mm}
    \noindent
    The canonical element on $\Omb:=\Cc^n_{\Wc} \x \Cc^n_{\Wc} \x \M \x \M \x \Cc^\ell$ is denoted by $(\mup,\mu,\Lambdap,\Lambda,B).$ Then, the canonical filtration $\Fb = (\Fcb_t)_{t \in [0,T]}$ is defined by: for all $t \in [0,T]$  
    $$ 
        \Fcb_t
        :=
        \sigma 
        \big\{\mup_{t \wedge \cdot}, \mu_{t \wedge \cdot},\Lambdap_{t \wedge \cdot}, \Lambda_{t \wedge \cdot},B_{t \wedge \cdot}
        \big\}, 
    $$
    with $\Lambdap_{t \wedge \cdot}$ and $\Lambda_{t \wedge \cdot}$ denote the restriction of $\Lambdap$ and $\Lambda$ on $[0,t] \x \Pc^n_U$ (see definition \ref{eq:lambda}). Notice that we can choose a version of the disintegration $\Lambdap(\mathrm{d}m,\mathrm{d}t)=\Lambdap_t(\mathrm{d}m)\mathrm{d}t$ (resp $\Lambda(\mathrm{d}m,\mathrm{d}t)=\Lambda_t(\mathrm{d}m)\mathrm{d}t$) such that $(\Lambdap_t)_{t \in [0,T]}$ (resp $(\Lambda_t)_{t \in [0,T]}$) a $\Pc(\Pc^n_U)$--valued $\Fb$--predictable process.
    Let us also introduce the $``{\color{black}fixed}\;common\;noise"$ filtration $(\Gcb_t)_{t \in [0,T]}$ by
    $$ 
        \Gcb_t
        :=
        \sigma 
        \big\{ \mu_{t \wedge \cdot}, \Lambda_{t \wedge \cdot},B_{t \wedge \cdot}
        \big\}. 
    $$
    
    We consider $\Lc$ the following generator: for $(t,x,\pi,m,u) \in [0,T] \x \R^n \x \Cc^n_{\Wc} \x \Pc^n_U \x U $, and $ \varphi \in C^2(\R^n)$
    \begin{align}
    \label{eq:def_generator}
        \Lc_t\varphi(x,\pi,m,u) 
        &:= 
        \Lc^\circ_t\varphi(x,\pi,u)
        +
        \Lc^\star_t\varphi(x,\pi,m) 
    \end{align}
    where
    \begin{align}
    \label{eq:def_generatorC}
        \Lc^\circ_t\varphi(x,\pi,u) 
        &:= 
        \frac{1}{2}  \text{Tr}\big[a^\circ(t,x,\pi_{t \wedge \cdot},u) \nabla^2 \varphi(x) 
        \big] 
        + b^\circ(t,x,\pi_{t \wedge \cdot},u)^\top \nabla \varphi(x),
    \end{align}
    and
    \begin{align}
    \label{eq:def_generatorS}
        \Lc^\star_t\varphi(x,\pi,m) 
        &:= 
        \frac{1}{2}  \text{Tr}\big[a^\star(t,\pi_{t \wedge \cdot},m) \nabla^2 \varphi(x) 
        \big] 
        + b^\star(t,\pi_{t \wedge \cdot},m)^\top \nabla \varphi(x).
    \end{align}
    Also, for every $f \in C^{2}(\R^n),$ let us define $N_t(f):=N_t[\mup,\mu,\Lambdap,\Lambda](f)$ by
    \begin{align} \label{eq:FP-equation}
        N_t[\mup,\mu,\Lambdap,\Lambda](f)
        :=
        \langle f(\cdot-\sigma_0 B_t),\mup_t \rangle
	    -
	    \langle f,\mu_0 \rangle
	    &-\int_0^t \int_{\Pc^n_U} \int_{\R^n}\Lc^\star_r[ f(\cdot-\sigma_0 B_r)](x,\mu,m) \mup_r(\mathrm{d}x)\Lambda_r(\mathrm{d}m)\mathrm{d}r \nonumber
	    \\
	   &-\int_0^t  \int_{\Pc^n_U}  \langle \Lc^\circ_r[ f(\cdot-\sigma_0 B_r)](\cdot,\mu,\cdot), m \rangle \Lambdap_r(\mathrm{d}m)\mathrm{d}r,
    \end{align}
    and for each $\pi \in \Pc(\R^n),$ the Borel set $\Z_{\pi}$ by
    \begin{align*}
        \Z_{\pi}:=\Big\{ m \in \Pc^n_U: m(\mathrm{d}x,U)=\pi(\mathrm{d}x) \Big\}.
    \end{align*}
\begin{definition}[measure--valued control rule] \label{def:RelaxedCcontrol}
    We say that $\mathrm{P} \in \Pc(\Omb)$ is a measure--valued control rule if:
    \begin{enumerate}
        \item $\mathrm{P} \big(\mu'_0=\nu \big)=1$.
        
        \item $(B_t)_{t \in [0,T]}$ is a $(\mathrm{P},\Fb)$ Wiener process starting at zero and for $\mathrm{P}$--almost every $\om \in \Omb$, $N_t(f)=0$ for all $f \in C^{2}_b(\R^n)$ and every $t \in [0,T].$
        
        \item {\color{black}
        $(\Lambda'_t)_{t \in [0,T]}$ is a $\Gb$--predictable process.
        }
        
        \item For $\mathrm{d}\mathrm{P} \otimes \mathrm{d}t$ almost every $(t,\om) \in [0,T] \x \Omb$, $ \Lambda'_t\big(\Z_{\mu'_t} \big)=1.$
    \end{enumerate}

\end{definition}
    
    We shall denote $\Pcb_V$ the set of all measure--valued control rules.
    
\begin{remark}
    { $(i)$ To do an analogy with {\rm \Cref{section:strong_formulation}} $($the strong ``point of view''$)$, in order to give a better intuition of this definition, here, $\mu'$ plays the role of $(\Lc(X^{\alpha,\alpha'}_t | \Gc_t ))_{t \in [0,T]},$ $\Lambdap$ that of {\color{black}$\delta_{\Lc(X^{\alpha,\alpha'}_s,\;\alpha_s' | \Gc_s )}(\mathrm{d}m)\mathrm{d}s,$} $\mu$ and $\Lambda$ represent the fixed measures $\mu^\alpha$ and $\delta_{\mub^\alpha_s}(\mathrm{d}m)\mathrm{d}s,$ and $B$ is the common noise.  
    
    \medskip
    {\color{black} $(ii)$ Notice that the canonical space $\Omb:=\Cc^n_{\Wc} \x \Cc^n_{\Wc} \x \M \x \M \x \Cc^\ell$ is ``doubled$"$ because we need to consider the fixed processes $(\mu,\Lambda)$ which will be the optimum and the controlled processes $(\mup,\Lambdap)$. All these processes share the same spaces. Also, because of the condition 3 of {\rm\Cref{def:RelaxedCcontrol}}, the set $\Pcb_V$ cannot be closed in general. As $\Pcb_V$ is not closed, the proofs become much more delicate $($see for instance {\rm\Cref{prop:converse_general}} $)$ }
    }
\end{remark} 

    Now, using the measure--valued control rules, we introduce the notion of ($\varepsilon$--) measure--valued MFG solution.
    \paragraph{MFG solution} \label{para_MFG-solution}
    For all $(\pi',\pi,q',q) \in (\Cc^n_{\Wc})^2 \x \M(U)^2,$ one defines 
    \begin{align*}
        J\big(\pi',\pi,q',q \big)
        :=
        \int_0^T \bigg[\int_{\Pc^n_U}\langle L^\circ\big(t,\cdot,\pi,\cdot \big), m \rangle q'_t(\mathrm{d}m) +
        \int_{\Pc^n_U}\langle L^\star\big(t,\cdot,\pi,m \big), \pi'_t \rangle q_t(\mathrm{d}m)
        \bigg]
        \mathrm{d}t 
        +
        \langle g(\cdot,\pi),\pi'_T \rangle.
\end{align*}

\begin{definition}  \label{def:RelaxedMFG_control}
        For $\varepsilon \in [0,\infty),$ $\mathrm{P}^\star$ is an $\varepsilon$--measure--valued MFG solution if $\mathrm{P}^\star \in \Pcb_V$, and for every $\mathrm{P} \in \Pcb_V$ such that $\Lc^{\mathrm{P}^{\star}}\big(\mu,\Lambda, B\big)=\Lc^{\mathrm{P}}\big(\mu,\Lambda,B \big)$, one has
        \begin{align} \label{eq:optimality-relaxed}
            \E^{\mathrm{P}^{\star}}\big[J(\mu',\mu, \Lambda',\Lambda) \big] \ge \E^{\mathrm{P}} \big[J(\mu',\mu, \Lambda',\Lambda) \big] - \varepsilon,
        \end{align}
        and for $\mathrm{P}^{\star}$ almost every $\om \in \Omb,$
        \begin{align} \label{eq:consistency}
            {\color{black}\Lambda'
            =
            \Lambda}
            ~~
            \mbox{and}
            ~~
            \mu'=\mu.
        \end{align}
        When $\varepsilon=0,$ we just say that $\mathrm{P}^\star$ is a measure--valued MFG solution.
\end{definition}

\medskip
The space $\Pcb^{\star}_V[\varepsilon]$ is defined by
	\[
		\Pcb^{\star}_V [\varepsilon]
		~:=~
		\big\{ 
			\mbox{All}\;\Pr^\star\;\varepsilon\mbox{--measure--valued MFG solutions}
		\big\},
	\]
again when $\varepsilon=0,$ we shall denote $\Pcb_V^\star[0]$ by $\Pcb_V^\star.$

\begin{remark}
    {\color{black} Condition \eqref{eq:consistency} is the analog of the well--known consistency property in the MFG framework. Without taking into account the law of control, one of the main differences of this notion of MFG solutions is the optimality conditions \eqref{eq:optimality-relaxed} and \eqref{eq:optimality-strong}. Here, sometimes a small error $\varepsilon$ is authorized. With this condition, the MFG solutions turn out to be more flexible $($see the mains results in {\rm \Cref{subsec:main-results}$).$}
    }
\end{remark}

\paragraph*{Comparison \Cref{def:RelaxedMFG_control} and \cite[Definition 3.1]{carmona2014mean}}
{\color{black}
    { Looking at this kind of measure--valued solution is largely inspired by the notion considered in {\rm\cite{MFD-2020}} in the McKean--Vlasov setting. However, our notion of $(\varepsilon$--$)$ measure--valued MFG solution enters completely in the framework of MFG solutions considered in {\rm\citeauthor*{carmona2014mean} \cite{carmona2014mean}}. Indeed, in the situation where our coefficients satisfied the assumptions of the setting of {\rm\cite{carmona2014mean}} $($essentially without the law of the control and with no control in the volatility$)$, any weak MFG solution of {\rm\cite{carmona2014mean}} can be seen as a measure--valued MFG solution, and conversely any measure--valued MFG solution can be seen as a weak MFG solution. Let us be more precise about this equivalence. We begin by recalling the notion of weak MFG solution of {\rm\cite{carmona2014mean}}. \begin{definition}{\rm(\cite[Definition 3.1]{carmona2014mean})} \label{def_carmona}
        The tuple $(\Omt,\Fct,\Ft,\widetilde{\P},\Wt,\Bt,\mubb,\widetilde{\Lambda},\Xt)$ is a weak MFG solution if
    \begin{enumerate}
        \item $(\Omt,\Fct,\Ft,\widetilde{\P})$ is a complete probability space. Also, $(\Wt,\Bt)$ is a $\R^{d+\ell}$--valued $\Ft$--Brownian motion, $\Xt$ is a $\R^n$--valued $\Ft$--adapted continuous process with $\widetilde{\P} \circ (\Xt_0)^{-1}=\nu$, and $\widetilde{\Lambda}$ is a $\Pc(U)$--valued $\Ft$--predictable measurable process. Lastly, $\mubb$ is a $\Pc(\Cc^n \x \M(U) \x \Cc^n)$--valued random variable such that $\mubb(S)$ is $\Fct_t$--measurable whenever $S \in \Mct_t$ and $t \in [0, T]$ where $\Mct:=(\Mct_t)_{t \in [0,T]}$ is the filtration s.t. $\Mct_t$ is the $\sigma$--field generated by the maps $\Cc^n \x \M(U) \x \Cc^n \ni (w, q, x) \to (w_s, q(S), x_s) \in \R^n \x \R \x \R^n$, where $s \le t$ and $S$ is a Borel subset of $[0, t] \x A.$
        
        \item $\Xt_0,$ $\Wt$ and $(\Bt,\mubb)$ are independent
        
        \item $\widetilde{\Lambda}_{t \wedge \cdot}$ is conditionally independent of $\Fct^{\Xt_0,\Wt,\mubb,\Bt}_T$ given $\Fct^{\Xt_0,\Wt,\mubb,\Bt}_t,$ for each $t \in [0,T],$ where $\Fct^{\Xt_0,\Wt,\mubb,\Bt}_t:=\sigma\big\{ \Xt_0,\Wt_{t \wedge \cdot},\Bt_{t \wedge \cdot},\mubb(S): S \in \Mct_t  \big\}$
        
        \item The state equation holds
        \begin{align*}
            \mathrm{d}\Xt_t
            =
            \int_U b(t,\Xt_t,\mubb^x_t,u) \widetilde{\Lambda}_t(\mathrm{d}u)\mathrm{d}t
            +
            \sigma(t,\Xt_t,\mubb^x_t)\mathrm{d}\Wt_t
            +
            \sigma_0 \mathrm{d}\Bt_t
        \end{align*}
        where $\mubb^x_t:=\mubb \circ [(w,q,x) \mapsto x_t]^{-1}$
        
        \item The control $\widetilde{\Lambda}$ is optimal, in the sense that: if $(\Omt',\Fct',\Ft',\widetilde{\P}',\Wt',\Bt',\mubb',\widetilde{\Lambda}',\Xt')$ satisfies $(1$--$4)$ and $\widetilde{\P}'\circ \big(\Xt_0',\Wt',\mubb',\Bt' \big)^{-1}=\widetilde{\P}\circ \big(\Xt_0,\Wt,\mubb,\Bt \big)^{-1}$, then we have
        \begin{align*}
            \E^{\widetilde{\P}'} \bigg[
				\int_0^T \int_U L(t, \Xt_t',\mubb'^x_t,u) \widetilde{\Lambda}_t'(\mathrm{d}u) \mathrm{d}t 
				+ 
				g(\Xt_T', \mubb'^x_T) 
			\bigg]
			\le
			\E^{\widetilde{\P}} \bigg[
				\int_0^T \int_U L(t, \Xt_t,\mubb^x_t,u) \widetilde{\Lambda}_t(\mathrm{d}u) \mathrm{d}t 
				+ 
				g(\Xt_T, \mubb^x_T) 
			\bigg].
        \end{align*}
        
        \item The consistency condition holds: $\mubb= \widetilde{\P} \big( (\Wt,\widetilde{\Lambda},\Xt) \in  \big| \Bt,\mubb \big)$ a.s.
    \end{enumerate}
        
    \end{definition}
    With this previous definition in mind, for any weak solution $(\Omt,\Fct,\Ft,\widetilde{\P},\Wt,\Bt,\mubb,\widetilde{\Lambda},\Xt),$ we have that (see \Cref{appendix:equivalence1} for the details):
    \begin{align*}
        \Pr:=\widetilde{\P} \circ \big( (\mubb^x_t)_{t \in [0,T]}, \Pi, (\mubb^x_t)_{t \in [0,T]}, \Pi, B \big)^{-1}
    \end{align*}
    is a measure--valued solution where $\Pi:=\delta_{\mb_t}(\mathrm{d}m)\mathrm{d}t$ with $\mb_t:=\E^{\widetilde{\P}}[\delta_{\Xt_t}(\mathrm{d}x)\widetilde{\Lambda}_t(\mathrm{d}u)\big| \Bt,\mubb].$ Conversely, let $\Pr^*$ be a measure--valued MFG solution, we can construct an extension $(\widetilde{\Omb},\widetilde{\Fcb}_T,\widetilde{\Fb},\Prt^\star)$ of the probability space $(\Omb,\Fcb_T,\Fb,\Pr^\star)$ supporting processes $(\Wt,\mubb,\widetilde{\Lambda},\Xt)$ s.t. the tuple $(\widetilde{\Omb},\widetilde{\Fcb}_T,\widetilde{\Fb},\Prt^\star,\Wt,B,\mubb,\widetilde{\Lambda},\Xt)$ is a weak MFG solution (see \Cref{appendix:equivalence2} for details). 
    
}    

    \begin{remark}
        { Notice that the previous definitions of the strong MFG equilibrium and $N$--player game cover the case without common noise. Indeed, for the non common noise case, it is enough to take $\ell=0$ i.e. $B$ and $\sigma_0$ disappear $($see {\rm\cite{djete2019mckean}, \cite{djete2019general}} and {\rm\cite{MFD-2020}}$)$. When $\sigma_0=0$ and $\ell \neq 0$, $B$ can be seen as an additional noise.
        }
    \end{remark}

The next proposition ensures that our measure--valued MFG solution definition using Fokker--Planck equation indeed generalizes the classical notion. {\color{black} The proof is postponed in \Cref{proof-prop-strong-as-measure}}

\begin{proposition}\label{prop-strong-as-measure}
    $\Pcb_S^\star[\varepsilon] \subset \Pcb_V^\star[\varepsilon],$ for all $\varepsilon \in [0,\infty).$
\end{proposition}

\subsection{Main limit results} \label{subsec:main-results}

The main results of this paper are now given in the following two theorems.

\begin{theorem}[Limit Theorem]
\label{thm:limitNashEquilibrium}
    Let {\rm \Cref{assum:main1}} hold true, $\varepsilon \in [0,\infty),$ $(\varepsilon_i)_{i \in \N^*} \subset (0,\infty).$
    
    \medskip
    \hspace{1em}$(i)$ For each $N \in \N^*$, let $\overline{\alpha}^N$ be a $(\varepsilon_1,\dots,\varepsilon_N)$--Nash equilibrium, then the sequence $(\mathrm{P}^N)_{N \in \N^*}$ with $\mathrm{P}^N:=\mathrm{P}^N[\overline{\alpha}^N] \in \Pc \big( \Omb \big)$ is relatively compact in $\Wc_p(\Omb)$ where
    \[
        \mathrm{P}^N[\overline{\alpha}^N]
        :=
        \P \circ \Big( \varphi^{N,\Xbb,\overline{\alpha}^N},\varphi^{N,\Xbb,\overline{\alpha}^N}, \Lambda^N,\Lambda^N, B  \Big)^{-1}\;\mbox{with}\;\Lambda^N:=\delta_{\varphi^{N,\overline{\alpha}^N}_t}(\mathrm{d}m)\mathrm{d}t,
    \]
    and 
    \begin{align*}
        \mbox{if}\;\;\lim_{N \to \infty}\frac{1}{N}\sum_{i=1}^N \varepsilon_i=\varepsilon,\;\;\mbox{then each limit point}\;\;\mathrm{P}^\infty\;\;\mbox{is an}\;\;\varepsilon\mbox{--measure--valued MFG solution.}
    \end{align*}
    
    \medskip
    \hspace{1em}$(ii)$ Let $(\mathrm{P}^k)_{k \in \N^*} \subset \Pcb_S$ such that {\color{black}$\mathrm{P}^k \in \Pcb^\star_S[\varepsilon_k],$} for each $k \in \N^*.$ Then $(\mathrm{P}^k)_{k \in \N^*}$ is relatively compact in $\Wc_p(\Omb)$, and
    \begin{align*}
        \mbox{if}\;\;\lim_{k \to \infty} \varepsilon_k=\varepsilon,\;\;\;\mbox{then each limit point}\;\;\mathrm{P}^\infty\;\;\mbox{is an}\;\;\varepsilon\mbox{--measure--valued MFG solution.}
    \end{align*}
        
    \medskip
    In particular when $\varepsilon=0,$ $\mathrm{P}^\infty$ is a measure--valued MFG solution.

\end{theorem}

\begin{theorem}[Converse Limit Theorem]
\label{thm:converselimitNash}
    Let {\rm\Cref{assum:main1}} hold true, $\varepsilon \in [0,\infty),$ and $\mathrm{P}^\star \in \Pcb^\star_V[\varepsilon].$
    
    \medskip
    $(i)$ There exists a sequence $(\varepsilon_k)_{k \in \N^*} \subset[0,\infty)$ satisfying $\limsup_{k \to \infty} \varepsilon_k \in [0,\epsilon]$ such that:
    
    \medskip
    \hspace{2em}$(i.1)$ if $\ell \neq 0,$ one can find a sequence $(\mathrm{P}^k)_{k \in \N^*}$ with $\mathrm{P}^k \in \Pcb^\star_S[\varepsilon_k]$ for each $k \in \N^*,$ and $\mathrm{P}^\star=\Lim_{k \to \infty} \mathrm{P}^k,$ for the metric $\Wc_p.$
    
    \medskip
    \hspace{2em}$(i.2)$ if {\color{black}$\ell = 0$ i.e. there are no B and $\sigma_0$ $($no common noise$),$ } one can get a sequence $(\mathrm{P}^k_z)_{(k,z) \in \N^* \x [0,1]} \subset \Pcb_S$ with for each $k \in \N^*,$ $z \mapsto \mathrm{P}^k_z$ is Borel measurable,
    \begin{align*}
        \int_0^1 \mathrm{P}^k_z \mathrm{d}z \in \Pcb^\star_V[\varepsilon_k]\;\;\;\;\mbox{and}\;\;\;\;\Lim_{k \to \infty} \int_0^1 \mathrm{P}^k_z \mathrm{d}z = \mathrm{P}^\star\;\mbox{in}\;\Wc_p. 
    \end{align*}
    
    \medskip
    $(ii)$ There exists a sequence of positive numbers $(\varepsilon_i)_{i\in \N^*}$ such that $\limsup_{N \to \infty}\frac{1}{N}\sum_{i=1}^N \varepsilon_i \in [0,\varepsilon],$ and for each $N \in \N^*$, a $(\varepsilon_1,\dots,\varepsilon_N)$--Nash equilibrium $\overline{\alpha}^N=(\alpha^{1,N},\dots,\alpha^{N,N})$ such that 
    \[
        \mathrm{P}^\star
        =
        \lim_{N \to \infty}
        \P \circ \Big( \varphi^{N,\Xbb,\overline{\alpha}^N},\varphi^{N,\Xbb,\overline{\alpha}^N}, \Lambda^N,\Lambda^N, B  \Big)^{-1}\;\mbox{in}\;\Wc_p\;\mbox{with}\;\Lambda^N:=\delta_{\varphi^{N,\overline{\alpha}^N}_t}(\mathrm{d}m)\mathrm{d}t.
    \]

\end{theorem}

\begin{remark}
    { 
        {\rm Theorems \ref{thm:converselimitNash}} and {\rm \ref{thm:limitNashEquilibrium} } give a general characterization of solutions of MFG of controls by connecting measure--valued MFG solutions, approximate Nash equilibria and approximate strong MFG solutions. {\color{black} Consequently, the results of existence of measure--valued MFG solutions, of approximate Nash equilibria and of approximate strong MFG solutions are all related, the existence of one of the notions guarantees the existence of the others.} {\color{black}In the presence of the law of the control or the empirical distribution of the controls}, our limit theorem results seem to be the first which give this kind of characterizations under relative general assumptions.  Especially, approximate strong MFG solutions and their convergence result have never been considered in the literature. 
        Notice that, they also contain part of the most results of the case without the distribution of controls mentioned in {\rm \citeauthor*{lacker2016general} \cite{lacker2016general}}.
        Let us emphasize that, there is no existence result in these theorems, all results are given after assuming existence results. In {\rm\Cref{sec:no-commonnoise}} $($see below$)$, we discuss some existence results in the case without common noise.
    
    }
\end{remark}

\medskip
    The next corollaries are just a combination of Theorems \ref{thm:converselimitNash} and \ref{thm:limitNashEquilibrium}. The first mentions the closedness of $\Pcb_V^\star$ and the second a correspondence between approximate Nash equilibria and approximate--strong MFG solution.

\begin{corollary}
    Suppose that the conditions of   {\rm\Cref{thm:converselimitNash}} and {\rm \Cref{thm:limitNashEquilibrium}} hold.
    For each $\varepsilon \in [0,\infty),$ $\Pcb_V^\star[\varepsilon]$ is a closed set for the Wasserstein metric $\Wc_p.$
\end{corollary}

\begin{corollary}
\label{corollary:accumalating_points}
    Let us stay in the context of {\rm\Cref{thm:converselimitNash}} and {\rm \Cref{thm:limitNashEquilibrium}} with $\ell \neq 0.$ For any $\overline{\alpha}^N$ a $(\varepsilon_1,\dots,\varepsilon_N)$--Nash equilibrium, with $\lim_{N \to \infty}\frac{1}{N}\sum_{i=1}^N \varepsilon_i=0,$ there exists, for each convergent sub--sequence $(\Pr^{N_k}[\overline{\alpha}^{N_k}])_{k \in \N^*},$ a sequence $(\Pr^k)_{k \in \N^*}$ such that:
    \begin{align*}
        \mbox{for each}\;k \in \N^*,\;\Pr^k \in \Pcb^\star_S[\delta_k]\;\mbox{with}\;\Lim_{k \to \infty} \delta_k=0\;\;\mbox{and}\;\;
        \Lim_{k \to \infty} \Wc_p \big(\Pr^{N_k}[\overline{\alpha}^{N_k}], \Pr^k \big)=0.
    \end{align*}
\end{corollary}

\subsection{Particular case of no common noise} \label{sec:no-commonnoise}

This section discusses of the case without common noise. Here, we assume that $\sigma_0=0$ (or  $\ell=0).$ 
Let us introduce the notion of  non--random measure--valued MFG solution.

\begin{definition} 
        We say that $\mathrm{P}^\star$ is a  non--random measure--valued MFG solution if $\mathrm{P}^\star \in \Pcb^\star_V$ and there exists $(\nb,\qb) \in \Cc^{n}_{\Wc} \x \M(\Pc^n_U)$ such that
        \begin{align*}
            \Lambda_t(\mathrm{d}m)\mathrm{d}t=\qb_t(\mathrm{d}m)\mathrm{d}t\;\mbox{and}\;\mu=\nb\;\;\mathrm{P}^\star\mbox{--a.s.}
        \end{align*}
\end{definition}
In other words, this notion of solution in the absence of common noise i.e. $\sigma_0=0$ or $\ell=0$ focuses on the deterministic {\color{black}solution of the Fokker--Planck equation} mentioned in Definition \ref{def:RelaxedCcontrol}. Indeed, even without common noise, it is possible to get a ``random'' measure--valued MFG solution. With the help of this deterministic aspect, one has the next theorem (see proof in Section \ref{sec:proofThm_existence}).

\begin{theorem}
\label{thm:existenceMFG}
Let {\rm\Cref{assum:main1}} hold true.

\medskip
    $(i)$ There exists at least one non--random measure--valued no common noise MFG solution.
    
    \medskip
    $(ii)$ Moreover, for any non--random measure--valued MFG solution $\mathrm{P}^\star,$ there exists a sequence $(\mathrm{P}^k)_{k \in \N^*} \subset \Pcb_S$ and a sequence $(\varepsilon_k)_{k \in \N^*} \subset[0,\infty)$ satisfying $\lim_{k \to \infty} \varepsilon_k=0$ such that: for each $k \in \N^*,$ $\mathrm{P}^k \in \Pcb_S^\star[\varepsilon_k]$ and $\Lim_k \mathrm{P}^k=\mathrm{P}^\star$ in $\Wc_p.$
\end{theorem}

\begin{remark}
    {
    Under {\rm\Cref{assum:main1}}, item $(i)$ of {\rm\Cref{thm:existenceMFG}} is an existence result. {\color{black}Unlike the item $(i.2)$ of {\rm\Cref{thm:converselimitNash}} where the approximation is achieved through a sequence of $\varepsilon$--measure--valued MFG solution build with convex combination of distribution of strong McKean--Vlasov processes, the item $(ii)$ of {\rm\Cref{thm:existenceMFG}} shows that, when the solution is non--random, despite the fact that $\ell=0,$ it is possible to approximate a measure--valued solution through a sequence of $\varepsilon$--strong MFG solutions.}
    }
\end{remark}

\begin{remark}
    {
    {\color{black} Let us} mention that it is possible to prove the existence of solution when $\ell \neq 0,$ using for instance {\rm\citeauthor*{lacker2015translation} \cite{lacker2015translation}} for particular coefficients or the techniques used by {\rm \citeauthor*{carmona2014mean} \cite{carmona2014mean}} and {\rm\citeauthor*{Touzi-Barrasso-2020} \cite{Touzi-Barrasso-2020}} $($discretization of the common noise filtration$)$. But this requires another long technical proof and this is not the main purpose of this paper. 
    See also {\rm\citeauthor*{claisse-Zhenjie-Tan-2020} \cite{claisse-Zhenjie-Tan-2020}} for an existence and approximation of particular mean field games $($MFG with branching$)$.
    }
\end{remark}




\section{Proof of main results} \label{sec:proofs}

\subsection{Limit of Nash equilibria} \label{sec:proof-limitNAshEquilibria}

{\color{black} \paragraph*{Idea of the proof}
Before going into the details of the proof, in order to better understand our approach, we want to give the idea leading the proof in a simple situation. For simplicity, let us consider that $n=1,$ $U=[0,1],$ $b=0,$ $\ell=0$ and $\sigma(t,x,\pi,m,u):=\sqrt{ u^2 + \widetilde{\sigma}(m)^2}.$ In this setting, for a measure--valued control rule $\Pr \in \Pcb_V,$ the Fokker--Planck equation is rewritten: $\Pr$--a.s.
\begin{align} \label{eq:FP-example}
        \mathrm{d}\langle f,\mu'_t \rangle
	    =
	    &\int_{\Pc^n_U}  \int_{\R^n \x U} \frac{1}{2} \nabla^2f(x)u^2 m^x(\mathrm{d}u) \mu'_t(\mathrm{d}x) \Lambda'_t(\mathrm{d}m)\mathrm{d}t + \int_{\Pc^n_U}  \int_{\R^n} \frac{1}{2} \nabla^2f(x) \sigmat(m)^2 \mu'_t(\mathrm{d}x)\Lambda_t(\mathrm{d}m)\mathrm{d}t,
    \end{align}
where for each $m \in \Pc^n_U,$  $\R^n \ni x \to m^x \in \Pc(U)$ is a Borel function s.t. $m(\mathrm{d}x,\mathrm{d}u)=m^x(\mathrm{d}u) m(\mathrm{d}x,U).$   
The goal is to make a $connection$ between \Cref{eq:FP-example} and \Cref{eq:FixedMKV_strong} (the strong version). 


\medskip
To recover \Cref{eq:FixedMKV_strong}, we use \cite[Proposition 4.9]{MFD-2020}. Let us explain the idea. We consider an extension $(\Omt,\Ft:=(\Fct_t)_{t \in [0,T]},\widetilde{\Pr})$ of the probability space $(\Omb,\Fcb,\Pr)$ supporting a $\Ft$--Brownian motion $W,$ a $\Fct_0$--measurable random variable $X_0$ of law $\nu,$ and $F$ a uniform random variable. The variables $W,$ $X_0,$ $F$ and the $\sigma$--field $\Gcb_T$ are independent. Let us assume that \Cref{eq:FP-example} is uniquely solvable in the sense that: if there exists another $\Pc(\R^n)$--valued $\Gb$--adapted continuous process $(\vartheta'_t)_{t \in [0,T]}$ satisfying \Cref{eq:FP-example} (instead of $\mup$) then $\vartheta'=\mu'$ $\widetilde{\Pr}$--a.s.
By  \citeauthor*{blackwellDubins83} \cite{blackwellDubins83}, there exists a Borel application $\Phi: (q,v) \in \Pc(U) \x [0,1] \to \Phi(q)(v) \in U$ s.t. for all $q \in \Pc(U)$ and any $[0,1]$--valued uniform random variable $F,$
\begin{align} \label{eq:blackwell}
    \widetilde{\Pr} \circ \big(\Phi(q)(F) \big)^{-1}(\mathrm{d}u)=q(\mathrm{d}u).
\end{align}
We define $\alpha(t,x,F):=\Phi\big(\int_{\Pc^n_U} m^{x}(\mathrm{d}u) \Lambda'_t(\mathrm{d}m) \big)(F).$ Let $Z'$ be a $\Ft$--adapted solution of: $Z'_0=X_0,$ and 

\begin{align} \label{eq:sde-auxillary-example}
    \mathrm{d}Z'_t
    =
    \sqrt{\alpha(t,Z'_t,F)^2 + \int_{\Pc^n_U}\sigmat(m)^2 \Lambda_t(\mathrm{d}m)} \mathrm{d}W_t,
\end{align}
then by uniqueness of \Cref{eq:FP-example}, $\mu'_t=\Lc^{\widetilde{\Pr}}(Z'_t | \Gcb_t)$ $\widetilde{\Pr}$--a.s. If in addition \Cref{eq:sde-auxillary-example} is uniquely solvable, then $Z'$ is adapted to $\big(\Gcb_t \lor \sigma\{W_{t \wedge \cdot},F\} \big)_{t \in [0,T]}.$ Consequently $(\alpha(t,Z'_t,F))_{t \in [0,T]}$ is $\big(\Gcb_t \lor \sigma\{W_{t \wedge \cdot},F\} \big)_{t \in [0,T]}$--adapted. We see that an additional randomness is needed via a uniform random variable $F.$ \cite[Proposition 4.9]{MFD-2020} is essentially a result similar to the previous one without assuming the uniqueness of \Cref{eq:FP-example} and the uniqueness of \Cref{eq:sde-auxillary-example}. This is done by approximations and regulations. The next Proposition is a simplify version of \cite[Proposition 4.9]{MFD-2020}. We give a sketch of proof in \Cref{proof_prop_borrowed1}.

\begin{proposition}{\rm \cite[Proposition 4.9]{MFD-2020}} \label{Prop_borrowed1}
    For a sequence of $\Gb$--predictable processes $(\Lambda^k,\Lambda'^k)_{k \in \N^*}$ satisfying
    \begin{align*} 
    \Lim_{k \to \infty} \Prt \circ \big(  \Lambda^k, \Lambda'^k \big)^{-1}
        =
        \Pr \circ (\Lambda, \Lambdap)^{-1}\;\mbox{in}\;\Wc_p
    \end{align*}
Then, there exists a sequence of $U$--valued $\big(\sigma\{\Lambda^k_{t \wedge \cdot}, \Lambda'^k_{t \wedge \cdot},W_{t \wedge \cdot},F\} \big)_{t \in [0,T]}$--predictable processes  $(\alpha'^k)_{k \in \N^*}$  s.t. if we let $Z'^k$ be the solution of     

\begin{align*}
    \mathrm{d}Z'^k_t
    =
    \sqrt{(\alpha'^k_t)^2 + \int_{\Pc^n_U}\sigmat(m)^2 \Lambda^k_t(\mathrm{d}m)} \mathrm{d}W_t,
\end{align*}
we have, for a sub--sequence $(k_j)_{j \in \N^*},$
    \begin{align*}
        \Lim_{j \to \infty} \Lc^{\widetilde{\Pr}} \big(\mu'^j,\Lambda'^j, \Lambda^{k_j} \big)
        =
        \Lc^{\Pr} \big(\mu',\Lambda',\Lambda \big)\;\mbox{in}\;\Wc_p
    \end{align*}
    where $\mu'^j_t:=\Lc^{\widetilde{\Pr}}(Z'^{k_j}_t | \Lambda^{k_j}_{t \wedge \cdot}, \Lambda'^{k_j}_{t \wedge \cdot}),$ $\mb'^{j}_t:=\Lc^{\widetilde{\Pr}}(Z'^{k_j}_t, \alpha'^{k_j}_t | \Lambda^{k_j}_{t \wedge \cdot}, \Lambda'^{k_j}_{t \wedge \cdot})$ and $\Lambda'^j:=\delta_{\mb'^j_t}(\mathrm{d}m)\mathrm{d}t.$

\end{proposition}

In the goal of making a $connection$ between \Cref{eq:FP-example} and \Cref{eq:FixedMKV_strong}, we will apply the previous result (in its general form) when the sequence $(\Lambda^k,\Lambda'^k)_{k \in \N^*}$ is adapted to the filtration of $B$ and there exists $(\nb^k,\nb'^k)$ s.t. $\Lambda'^k=\delta_{\nb'^k_t}(\mathrm{d}m)\mathrm{d}t$ and $\Lambda^k=\delta_{\nb^k_t}(\mathrm{d}m)\mathrm{d}t$ (see \Cref{lemm:convergenceMFGstrong}).

\medskip
    Let us mentioning that if $\Prt(\Lambdap=\Lambda,\;\Lambda'^k=\Lambda^k\;\forall k)=1$ i.e. \Cref{eq:FP-example} is a totally non--linear Fokker--Planck equation, $Z'^k$ will be equal to $Z^k$ a solution of a McKean--Vlasov equation i.e. $\mathrm{d}Z^k_t
    =
    \sqrt{(\alpha_t^k)^2 + \sigmat(\Lc^{\widetilde{\Pr}}(Z^k_t,\alpha^k_t|\Lambda^k))^2} \mathrm{d}W_t.$
    
    \subsubsection{Reformulation of the measure--valued control rules by shifting the distributions}
}
Before proceeding, let us give a reformulation of the measure--valued control rules which will be necessary for our proof. To make an analogy with the strong point of view, we want here to get a Fokker--Planck equation involving $\Lc(X^{\alpha,\alpha'}-\sigma_0B|\Gc_T)$ instead of $\Lc(X^{\alpha,\alpha'}|\Gc_T).$ To do this, all {\color{black} the} coefficients must be shifted.  
Let us define, for all $(t,\bb,\pi,m) \in [0,T] \x \Cc^\ell \x \Cc^n_{\Wc} \x \Pc^n_U,$
\begin{align} \label{eq:shift-proba-initial}
    \pi_t[\bb](\mathrm{d}y):= \int_{\R^n} \delta_{\big(y'+\sigma_0 \bb_t \big)}(\mathrm{d}y) \pi_t(\mathrm{d}y'),\;\; m[\bb_t](\mathrm{d}u,\mathrm{d}y):=\int_{\R^n \x U} \delta_{(y'+\sigma_0 \bb_t)}(\mathrm{d}y)m(\mathrm{d}u,\mathrm{d}y')
\end{align}
and any $q \in \M,$
\begin{align} \label{eq:shift-proba-M}
        q_t[\bb](\mathrm{d}m)\mathrm{d}t:=\int_{\Pc^n_U} \delta_{\big(\tilde m[\bb_t] \big)}(\mathrm{d}m) q_t(\mathrm{d}\widetilde m)\mathrm{d}t.
    \end{align}
In the same way, let us consider the {\color{black}``shifted''} generator $\widehat{\Lc}^{\circ},$
\begin{align} \label{eq:shift-generator}
        \widehat{\Lc}^{\circ}_t[ \varphi](y,\bb,\pi,u):= \frac{1}{2}  \mathrm{Tr}\big[a^\circ(t,y+\sigma_0\bb_t,\pi,u) \nabla^2 \varphi(y) \big] +b^{\circ}(t,y+\sigma_0\bb_t,\pi,u)^{\top} \nabla \varphi(y),
    \end{align}
and also
    \begin{align*} 
        \;[\widehat b,\widehat \sigma](t,y,\bb,\pi,m,u):=[b,\sigma](t,y+\sigma_0\bb_t,\pi,m,u).
    \end{align*}
    Notice that the function $[\widehat b,\widehat \sigma]: [0,T] \x \R^n \x \Cc^\ell \x \Cc^n_{\Wc} \x U \to \R^n \x \S^{n}$ {\color{black}is} continuous and for each $\bb \in \Cc^\ell,$ $[\widehat b,\widehat \sigma](\cdot,\cdot,\bb,\cdot,\cdot)$ verify \Cref{assum:main1}.      

\medskip    
Next, on the canonical filtered space $(\Omb,\Fb),$ let us define the $\Pc(\R^n)$--valued $\Fb$--adapted continuous process $(\vartheta'_t)_{t \in [0,T]}$ and the $\Pc^n_U$--valued $\Fb$--predictable process  $(\Theta'_t)_{t \in [0,T]}$ by
    \begin{align} \label{eq:shift-proba}
        \vartheta'_t(\omb):=\mu'_t(\omb)[-B(\omb)]\;\;\mbox{and}\;\;\Theta'_t(\omb)(\mathrm{d}m):=\Lambda'_t(\omb)[-B(\omb)](\mathrm{d}m),\;\mbox{for all}\;(t,\omb)\in [0,T] \x \Omb.
    \end{align}    
    

\begin{lemma} \label{lemma:reformulation-approximation}
    Let $\Pr \in \Pcb_V.$ Then, $\Theta'_t(\Z_{\vartheta'_t})=1,$ $\mathrm{d}\Pr \otimes \mathrm{d}t,$ a.e. $(t,\omb) \in [0,T] \x \Omb,$ and $\Pr$--a.e. $\omb \in \Omb,$ for all $(f,t) \in C^{2}_b(\R^n) \x [0,T],$
    \begin{align*}
        N_t(f)
        =
        \langle f,\vartheta'_t \rangle
	    -
	    \langle f,\nu \rangle
	    &
	    -\int_0^t\int_{\Pc^n_U} \langle \widehat{\Lc}^{\circ}_r[ f](\cdot,B,\mu,\cdot), m \rangle \Theta'_r(\mathrm{d}m)\mathrm{d}r
	    - \int_0^t\int_{\Pc^n_U} \langle \Lc^\star_r[ f](\cdot,\mu,m), \vartheta'_r\rangle \Lambda_r(\mathrm{d}m)\mathrm{d}r.
    \end{align*}
    
    Moreover, there exists a sequence $(G^k)_{k \in \N^*},$ such that for each $k \in \N^*,$  $G^k:[0,T] \x \Cc^\ell \x \Cc^n_{\Wc} \x \M(\Pc^n_U) \to \Pc^n_U$ is a continuous function and
    
    \begin{align} \label{eq:density-control}
        \Lim_{k \to \infty}\Lc^{\mathrm{P}} \Big( \delta_{G^k \big(t,B_{t \wedge \cdot},\mu_{t \wedge \cdot}, \Lambda_{t \wedge \cdot} \big)}(\mathrm{d}m)\mathrm{d}t,B,\mu,\Lambda \Big)
        =
        \Lc^{\mathrm{P}} \Big( \Theta',B,\mu,\Lambda\Big)\;\mbox{in}\;\Wc_p.
    \end{align}

\end{lemma}

\begin{proof}
    The first point is just a reformulation of the process $N(f).$ For \Cref{eq:density-control}, as $\mathrm{P} \in \Pcb_V,$ by (an easy extension of) \cite[Lemma 5.2]{MFD-2020}, for $p'>p,$
    
        $$
            \sup_{t \in [0,T]} \int_{\R^n}|x|^{p'} \vartheta'_t(\om)(\mathrm{d}x)
            +
            \E^{\mathrm{P}}\bigg[
            \sup_{t \in [0,T]} \int_{\R^n}|x|^{p'} \mu_t(\mathrm{d}x)
            \bigg]
            \le 
            K 
            \bigg[ 1+\int_{\R^n}|y|^{p'}\nu(\mathrm{d}y) \bigg],\;\mathrm{P}\mbox{--a.e.}\;\om \in \Omb.
        $$
    We define $\Gamma:=
    \Big\{
        m \in \Pc^n_U:\; \int_{\R^n} |y|^{p'} m(\mathrm{d}y,U) \le \hat K
    \Big\},$ where $\hat K>0$ is such that $\hat K>K \bigg[ 1+\int_{\R^n}|y|^{p'}\nu(\mathrm{d}y) \bigg],$ with $K$ is the constant previously used. Then, as $U$ is a compact set, we can notice that $\Gamma$ is a compact set of $\Pc_p(\R^n \x U),$ and one has $\Theta'_t(\Gamma)=1,$ $\mathrm{d}\mathrm{P} \otimes \mathrm{d}t,$ a.e. $(t,\om) \in [0,T] \x \Omb.$ $\Theta'$ is $\Gb$--predictable because $\Lambda'$ is $\Gb$--predictable. Then we can write $\Theta'_t=G(t,B_{t \wedge \cdot},\mu_{t \wedge \cdot}, \Lambda_{t \wedge \cdot} )$ where $G:[0,T] \x \Cc^\ell \x \Cc^n_{\Wc} \x \M(\Pc^n_U) \to \Pc(\Pc^n_U)$ is a Borel function.  By combining all the previous observation, using \cite[Theorem 2.2.3]{CastaingCharles2004YMoT} (or \cite[Proposition C.1]{carmona2014mean}), there exists a sequence $(G^k)_{k \in \N^*}$ such that for each $k \in \N^*$  $G^k:[0,T] \x \Cc^\ell \x \Cc^n_{\Wc} \x \M(\Pc^n_U) \to \Pc^n_U$ is a Borel function and
    \begin{align*}
        \Lim_{k \to \infty} \Lc^{\mathrm{P}} \Big( \delta_{G^k \big(t,B_{t \wedge \cdot},\mu_{t \wedge \cdot}, \Lambda_{t \wedge \cdot} \big)}(\mathrm{d}m')\mathrm{d}t,B,\mu,\Lambda \Big)
        =
        \Lc^{\mathrm{P}} \big( \Theta',B,\mu,\Lambda\big).
    \end{align*}    
    By using another approximation, we can take $G^k$ as a continuous function.
\end{proof}

\subsubsection{Technical lemmas}
In this section, we will show some technical results needed to prove our first limit theorem result, namely Theorem \ref{thm:limitNashEquilibrium} using the more general form of  \cite[Proposition 4.9]{MFD-2020} with common noise.

\medskip
The following lemma establishes a result which implies that any measure--valued control rule satisfying some technical conditions can be approximated by processes of type $X^{\alpha,\alpha'}$ (see Definition \ref{eq:FixedMKV_strong}).

\begin{lemma}
\label{lemm:convergenceMFGstrong}
    Let {\rm \Cref{assum:main1}} hold true and $\mathrm{P} \in \Pcb_V.$ On the probability space $(\Om,\H,\P),$ for any sequence $(\alpha^k)_{k \in \N^*} \subset \Ac$ satisfying
    \begin{align*}
        \Lim_{k \to \infty} 
        \P \circ \Big( \mu^{\alpha^k}, \delta_{\mub^{\alpha^k}_t}(\mathrm{d}m) \mathrm{d}t, B \Big)^{-1}
        =
        \mathrm{P} \circ \big( \mu,\Lambda,B \big)^{-1},
    \end{align*}
    then there exists a family of probability $(\Pr^k_f)_{(k,f) \in \N^* \x [0,1]} \subset \Pcb_S$ such that for each $k \in \N^*$:
    \begin{align*}
        [0,1]\ni f  \mapsto \Pr^k_f \in \Pc(\Omb)\;\mbox{is Borel measurable},\;\;\Pr^k_f \circ (\mu,\Lambda,B)^{-1}=\P \circ \big( \mu^{\alpha^k}, \delta_{\mub^{\alpha^k}_t}(\mathrm{d}m) \mathrm{d}t, B \big)^{-1}\;\mbox{for each}\;f \in [0,1],
    \end{align*}
    and, for a sub--sequence $(k_j)_{j \in \N^*},$
    \begin{align*}
        &\Lim_{j \to \infty}
        \int_0^1 \E^{\Pr^{k_j}_f}\big[J(\mu',\mu,\Lambda',\Lambda) \big]  \mathrm{d}f
        =
        \E^{\Pr}\big[J(\mu',\mu,\Lambda',\Lambda) \big].
    \end{align*}
    
\end{lemma}

\begin{proof}
$\underline{Step\;1:Reformulation}$:
For $\Pr \in \Pcb_V,$ by definition, $\Pr$--a.s. $\om \in \Omb,$ $N_t(f)=0$ for all $f \in C^{2}_b(\R^n)$ and $t \in [0,T].$ By Lemma \ref{lemma:reformulation-approximation}, recall that $(\vartheta'_t)_{t \in [0,T]}$ and $(\Theta'_t)_{t \in [0,T]}$ {\color{black}are} defined in \eqref{eq:shift-proba}, one has $\Theta'_t(\Z_{\vartheta'_t})=1,$ $\mathrm{d}\Pr \otimes \mathrm{d}t$ a.s. $(t,\om) \in [0,T] \x \Omb,$ and $\Pr$--a.s. $\om \in \Omb,$ for all $(f,t) \in C^{2}_b(\R^n) \x [0,T],$
   \begin{align*}
        0
        =
        \langle f,\vartheta'_t \rangle
	    -
	    \langle f,\nu \rangle
	    &
	    -\int_0^t\int_{\Pc^n_U} \langle \widehat{\Lc}^{\circ}_r[ f](\cdot,B,\mu,\cdot), m \rangle \Theta'_r(\mathrm{d}m)\mathrm{d}r
	    - \int_0^t\int_{\Pc^n_U} \langle \Lc^\star_r[ f](\cdot,\mu,m), \vartheta'_r\rangle \Lambda_r(\mathrm{d}m)\mathrm{d}r.
    \end{align*}

    \medskip
    $\underline{Step\;2:Approximation}$: By Lemma \ref{lemma:reformulation-approximation}, there exists a sequence $(G^l)_{l \in \N^*}$ such that for each $l \in \N^*,$  $G^l:[0,T] \x \Cc^\ell \x \Cc^n_{\Wc} \x \M(\Pc^n_U) \to \Pc^n_U$ is a continuous function and
    \begin{align*}
        \Lim_{l \to \infty} \Lc^{\mathrm{P}} \Big( \delta_{G^l \big(t,B_{t \wedge \cdot},\zeta_{t \wedge \cdot}, \Lambda_{t \wedge \cdot} \big)}(\mathrm{d}m)\mathrm{d}t,B,\mu,\Lambda \Big)
        =
        \Lc^{\mathrm{P}} \big( \Theta',B,\mu,\Lambda\big).
    \end{align*}

\medskip
Now, we apply \cite[Proposition 4.9]{MFD-2020} (see also \cite[Proposition 4.7]{MFD-2020}). 
    First, there exists a sub--sequence $(l_k)_{k \in \N^*} \subset \N^*,$ such that if $\Lambda^k:=\delta_{\mub^{\alpha^{k}}_t}(\mathrm{d}m)\mathrm{d}t,$ and 
    \begin{align*}
        \mb'^{k}_t
        :=
        G^{l_k} \big(t,B_{t \wedge \cdot},\mu^{\alpha^k}_{t \wedge \cdot}, \Lambda^k_{t \wedge \cdot} \big)
        \;\mbox{and}\;
        \Theta'^{k}_t(\mathrm{d}m)\mathrm{d}t
        :=
        \delta_{\mb'^{k}_t}(\mathrm{d}m)\mathrm{d}t,
    \end{align*}
    one has
\begin{align*}
        &\Lim_{k \to \infty} 
        \P \circ \big( \Theta'^{k}, \mu^{\alpha^k}, \Lambda^k, B \big)^{-1}
        =
        \Lim_{l \to \infty} \Lc^{\mathrm{P}} \Big(\delta_{G^l \big(t,B_{t \wedge \cdot},\mu_{t \wedge \cdot}, \Lambda_{t \wedge \cdot} \big)}(\mathrm{d}m)\mathrm{d}t,\mu,\Lambda,B \Big)
        =
        \Lc^{\mathrm{P}} \big( \Theta',\mu,\Lambda,B \big).
\end{align*}
Next, under Assumption \ref{assum:main1}, by \cite[Proposition 4.9]{MFD-2020} (with separability condition see \cite[Remark 4.11]{MFD-2020}), as $(\xi,W)$ is $\P$ independent of $(B,\mu^{\alpha^k},\mub^{\alpha^k})_{k \in \N^*},$ there exist a $[0,1]$--valued uniform random variable $F$ independent of $(\xi,W,B),$ and a Borel function $R^{k}: [0,T] \x \R^n \x \Cc^n_{\Wc} \x \M \x \M \x \Cc^n \x \Cc^\ell \x [0,1] \to U$ such that if we let $\Xh'^{k}$ be the unique strong solution of: for all $t \in [0,T],$
\begin{align} \label{eq:general-weak_McK}
    \mathrm{d}\Xh'^{k}_t
    =
    \widehat{b}\big(t,\Xh'^{k}_t,B,\mu^{\alpha^k}, \mub^{\alpha^k}_t,\alpha'^{k}_{t}\big) \mathrm{d}t
    +
    \widehat{\sigma} \big(t,\Xh'^{k}_t,B,\mu^{\alpha^k}, \mub^{\alpha^k}_t,\alpha'^{k}_{t}\big)
    \mathrm{d}W_t\;\mbox{with}\;\Xh'^k_0=\xi,
\end{align}
    where $\Gc^{k}:=(\Gc^{k}_s)_{s \in [0,T]}:=(\sigma \{ \mu^{\alpha^k}_{s \wedge \cdot},\Theta'^{k}_{s \wedge \cdot},\Lambda^k_{s \wedge \cdot}, B_{s \wedge \cdot} \})_{s \in [0,T]},$
\begin{align*}
    \alpha'^{k}_t:= R^k\big(t,\xi,\mu^{\alpha^k}_{t \wedge \cdot},\Theta'^{k}_{t \wedge \cdot},\Lambda^k_{t \wedge \cdot},W_{t \wedge \cdot}, B_{t \wedge \cdot},F \big),\;\overline{\vartheta}'^{k}_t:=\Lc\big(\Xh'^{k}_t,\alpha'^{k}_t\big|\Gc^k_t\big)\;\mbox{and}\;\vartheta'^{k}_t:=\Lc\big(\Xh'^{k}_t\big|\Gc^{k}_t\big),
\end{align*}
then $\Lim_{k \to \infty} \E^{\P} \bigg[\int_0^T \Wc_p \big(\overline{\vartheta}'^{k}_t,\mb'^{k}_t \big)^p \mathrm{d}t \bigg]=0,$ and
    \begin{align*}
        \Lim_{j \to \infty} \Lc \big(\vartheta'^{k_j},V'^{k_j},\mu^{\alpha^{k_j}},\Lambda^{k_j},B \big)
        =
        \Lc^{\Pr} \big(\vartheta',\Theta',\mu,\Lambda,B \big),\;\mbox{in}\;\Wc_p,
    \end{align*}
    where $V'^{k}_t(\mathrm{d}m)\mathrm{d}t:=\delta_{\overline{\vartheta}'^{k}_t}(\mathrm{d}m)\mathrm{d}t$ and $(k_j)_{j \in \N^*} \subset \N^*$ is a sub--sequence.
    
\medskip
$\underline{Step\;3:Rewriting}$: {\color{black}Notice that, as $(\mu^{\alpha^k},\Theta'^{k},\Lambda^k)$ are adapted to the filtration of $B$ then $\G^k \subset \G.$} Besides, $(\xi,F,W)$ are independent of $\Gc_T,$ therefore $\Lc\big(\Xh'^k_t,\alpha'^{k}_t\big|\Gc^k_t\big)=\Lc\big(\Xh'^k_t,\alpha'^{k}_t\big|\Gc_t\big),$ $\P$--a.s. for all $t \in [0,T].$ Using definition of $[\widehat b, \widehat \sigma]$ (see the equations \eqref{eq:shift-generator}),
\begin{align*} 
    \mathrm{d}\Xh'^k_t
    =
    b\big(t,\Xh'^k_t+\sigma_0B_t,\mu^{\alpha^k},\mub^{\alpha^k}_t,\alpha'^{k}_{t}\big) \mathrm{d}t
    +
    \sigma \big(t,\Xh'^k_t+\sigma_0B_t,\mu^{\alpha^k},\mub^{\alpha^k}_t,\alpha'^{k}_{t}\big)
    \mathrm{d}W_t,\;\Xh'^k_0=\xi.
\end{align*}
Denote $X'^k:=\Xh'^k+\sigma_0B,$ one finds
\begin{align*} 
    \mathrm{d}X'^k_t
    =
    b\big(t,X'^k_t,\mu^{\alpha^k},\mub^{\alpha^k}_t,\alpha'^{k}_{t}\big) \mathrm{d}t
    +
    \sigma\big(t,X'^k_t,\mu^{\alpha^k},\mub^{\alpha^k}_t,\alpha'^{k}_{t}\big)
    \mathrm{d}W_t+\sigma_0 \mathrm{d}B_t
\end{align*}

It is straightforward to check that the function 
\begin{align*}
    (\pi,q,\bb) \in \Cc^n_{\Wc} \x \M \x \Cc^\ell \to \big(\pi[\bb],q_t[\bb](\mathrm{d}m)\mathrm{d}t,\bb \big) \in \Cc^n_{\Wc} \x \M \x \Cc^\ell
\end{align*}
is continuous. Consequently, one has
    \begin{align*}
        &\Lim_{j \to \infty} \Lc \Big( \big(\Lc(X'^{k_j}_s|\Gc_s) \big)_{s \in [0,T]}, \delta_{\Lc(X'^{k_j}_s,\alpha'^{k_j}_s|\Gc_s)} (\mathrm{d}m)\mathrm{d}s,\mu^{\alpha^{k_j}},\delta_{\mub^{\alpha^{k_j}}_s}(\mathrm{d}m)\mathrm{d}s, B  \Big)
        \\
        &=
        \Lim_{j \to \infty} \Lc^{\P} \big(\vartheta'^{k_j}[B],V'^{k_j}_t[B](\mathrm{d}m)\mathrm{d}t,\mu^{\alpha^{k_j}},\delta_{\mub^{\alpha^{k_j}}_s}(\mathrm{d}m)\mathrm{d}s, B \big)
        =
        \Lc^{\mathrm{P}} \big(\vartheta'[B],\Theta'_t[B](\mathrm{d}m)\mathrm{d}t,\mu,\Lambda,B \big)\;\;\;\mbox{in}\;\Wc_p.
    \end{align*}
{\color{black}Using the definition of $\vartheta'$ and $\Theta'$ in \eqref{eq:shift-proba}, we easily verify after calculations that $(\vartheta'[B],\Theta'_t[B](\mathrm{d}m)\mathrm{d}t,B)=(\mu',\Lambda',B),$ $\P$--a.e. }
Then   
 \begin{align} \label{eq:conv-after-app}
        &\Lim_{j \to \infty} \Lc \Big( \big(\Lc(X'^{k_j}_s|\Gc_s) \big)_{s \in [0,T]}, \delta_{(\Lc(X'^{k_j}_s,\alpha'^{k_j}_s|\Gc_s))} (\mathrm{d}m)\mathrm{d}s,\mu^{\alpha^{k_j}},\delta_{\mub^{\alpha^{k_j}}_s}(\mathrm{d}m)\mathrm{d}s, B  \Big)
        =
        \Lc^{\mathrm{P}} \big(\mu',\Lambda',\mu,\Lambda,B \big)\;\mbox{in}\;\Wc_p.
    \end{align}
    
    Now, to finish, we define 
    \begin{align*}
        \Pr^k_f
        :=
        \Lc \Big( \big(\Lc(X'^{k_j}_s|\Gc_s \lor F) \big)_{s \in [0,T]},\mu^{\alpha^{k_j}}, \delta_{(\Lc(X'^{k_j}_s,\alpha'^{k_j}_s|\Gc_s \lor F))} (\mathrm{d}m)\mathrm{d}s,\delta_{\mub^{\alpha^{k_j}}_s}(\mathrm{d}m)\mathrm{d}s, B \big| F = f \Big).
    \end{align*}
As $F$ is independent of $(\xi,W,B),$ it is straightforward to check that $\Pr^k_f \circ (\mu,\Lambda,B)^{-1}=\P \circ \big( \mu^{\alpha^k}, \delta_{\mub^{\alpha^k}_t}(\mathrm{d}m) \mathrm{d}t, B \big)^{-1}$ and $\Pr^k_f \in \Pcb_S$ for each $(k,f).$ Besides, with an easy manipulation of the conditional expectation and \Cref{eq:conv-after-app}, we have 

\begin{align*}
    \int_0^1 \E^{\Pr^{k_j}_f}\big[J(\mu',\mu, \Lambda',\Lambda) \big]  \mathrm{d}f
    =
    \E\bigg[
				\int_0^T L(t, X'^{k}_t,\mu^{\alpha^k}_{t \wedge \cdot}, \mub^{\alpha^k}_t,\gamma^k_t) \mathrm{d}t 
				+ 
				g(X'^{k}_T, \mu^{\alpha^k}) 
			\bigg]
\end{align*}
and 
\begin{align*}
        \Lim_{k \to \infty}\int_0^1 \E^{\Pr^{k}_f}\big[J(\mu',\mu, \Lambda',\Lambda) \big]  \mathrm{d}f
        =
        \E^{\Pr}\big[J(\mu',\mu, \Lambda',\Lambda) \big].
    \end{align*}
This is enough to conclude the proof.

\end{proof}

\medskip
    Now, we consider the case of $N$--player game. Loosely speaking, we will show that: given the controls $\overline{\alpha}^N:=\big(\alpha^1,\dots,\alpha^N \big),$ replace one control $\alpha^i$ by another $\kappa^N$ has no effect on the empirical distribution $(\varphi^{N,\Xbb,\overline{\alpha}},\varphi^{N,\overline{\alpha}})$ (see Definition \ref{eq:N-agents_StrongMF_CommonNoise}) when $N$ goes to infinity.
    
\medskip    
    Given $N \in \N^*,$ $(\alpha^i)_{1 \le i \le N} \subset \Ac^N$ and $\kappa^N \in \Ac^N.$ Let us introduce, for each $i \in \{1,\dots,N\},$ the unique strong solution $Z^i$ of:
    \begin{align*}
        \mathrm{d}Z^i_t
        =
        b\big(t,Z^i_t,\varphi^{N,\Xbb,\overline{\alpha}^N},\varphi^{N,\overline{\alpha}^N}_{t} ,\kappa^N_t \big)\mathrm{d}t 
        +
        \sigma \big(t,Z^i_t,\varphi^{N,\Xbb,\overline{\alpha}^N},\varphi^{N,\overline{\alpha}^N}_{t} ,\kappa^N_t \big) \mathrm{d}W^i_t
        +
        \sigma_0 \mathrm{d}B_t\;\mbox{with}\;Z^i_0=\xi^i
    \end{align*}
    where $(\varphi^{N,\Xbb,\overline{\alpha}},\varphi^{N,\overline{\alpha}})$ correspond to the empirical distributions associated with the controls $\overline{\alpha}^N:=\big(\alpha^1,\dots,\alpha^N \big)$ (see Definition \ref{eq:N-agents_StrongMF_CommonNoise})
\begin{lemma}
\label{lemma:estimation_convergence}
    There exists a constant $K>0$ (depending only on the $p$--moment of $\nu$) such that: if $\overline{\alpha}^{N,-i}:=\big(\overline{\alpha}^{[-i]},\kappa^N \big),$ for each $i \in \{1,\dots,N\},$ one has
    \begin{align*}
        \bigg(\E \bigg[ \sup_{t \in [0,T]} \Wc_p \big(\varphi^{N,\Xbb,\overline{\alpha}^N}_{t},\varphi^{N,\Xbb,\overline{\alpha}^{N,-i}}_t \big)^{\color{black}p} \bigg]
        +
        \E \bigg[ \sup_{t \in [0,T]} \big|Z^{i}_t-\Xbb^i_t[\overline{\alpha}^{N,-i}]\big|^p \bigg] \bigg)
	    \le
	    K \frac{1}{N}.
    \end{align*}
    Consequently, $\Limsup_{N \to \infty} \Wc_p \big(\Q^N, \widetilde{\Q}^N \big)=0,$ where
    \begin{align*}
        \Q^N:=\frac{1}{N} \sum_{i=1}^N\P \circ \Big( \Xbb^i[\overline{\alpha}^{N,-i}], \varphi^{N,\Xbb,\overline{\alpha}^{N,-i}}, \delta_{\big(\kappa^N_t,\varphi^{N,\overline{\alpha}^{N,-i}}_t \big)}(\mathrm{d}u,\mathrm{d}m)\mathrm{d}t \Big)^{-1},
    \end{align*}
    and
    \begin{align*}
        \widetilde{\Q}^N:=\frac{1}{N} \sum_{i=1}^N\P \circ \Big( Z^i, \varphi^{N,\Xbb,\overline{\alpha}^{N}}, \delta_{\big(\kappa^N_t,\varphi^{N,\overline{\alpha}^{N}}_t \big)}(\mathrm{d}u,\mathrm{d}m)\mathrm{d}t \Big)^{-1}.
    \end{align*}
\end{lemma}

\begin{proof}
    This proof is a successive application of the Gronwall's lemma.
    For $j \in \{1,\dots,N\}$ with $j \neq i,$ for all $t \in [0,T],$ {\color{black} using \Cref{assum:main1} especially the boundness and Lipschitz properties of $(b,\sigma),$} one finds
    \begin{align*}
        &\E \bigg[\sup_{s \in [0,t]} \big|\Xbb^j_s[\overline{\alpha}^{N,-i}]-\Xbb^j_s[\overline{\alpha}^{N}]\big|^p \bigg]
        \\
        &~~~~~~~~\le
        C \bigg( \E \bigg[\int_0^t \Big|\big[b,\sigma \big] \big(r,\Xbb^j_r[\overline{\alpha}^{N,-i}],\varphi^{N,\Xbb,\overline{\alpha}^{N,-i}},\varphi^{N,\overline{\alpha}^{N,-i}}_{r} ,\alpha^j_r \big) 
        -
        \big[b,\sigma \big] \big(r,\Xbb^j_r[\overline{\alpha}^{N}],\varphi^{N,\Xbb,\overline{\alpha}^{N}},\varphi^{N,\overline{\alpha}^{N}}_{r} ,\alpha^j_r \big) 
        \Big|^p \mathrm{d}r \bigg] \bigg)
        \\
        &~~~~~~~~\le C
        \bigg( \E \bigg[\int_0^t \sup_{s \in [0,r]} \big|\Xbb^j_s[\overline{\alpha}^{N,-i}]-\Xbb^j_s[\overline{\alpha}^{N}]\big|^p + \sup_{s \in [0,r]} \Wc_p \big(\varphi^{N,\Xbb,\overline{\alpha}^{N,-i}}_s, \varphi^{N,\Xbb,\overline{\alpha}^{N}}_s \big)^p +  \Wc_p \big(\varphi^{N,\overline{\alpha}^{N,-i}}_r, \varphi^{N,\overline{\alpha}^{N}}_r \big)^p \mathrm{d}r \bigg]\bigg),
    \end{align*}
    then by Gronwall's lemma, 
    \begin{align}
    \label{eq:first_estimation}
        &\E \bigg[\sup_{s \in [0,t]} \big|\Xbb^j_s[\overline{\alpha}^{N,-i}]-\Xbb^j_s[\overline{\alpha}^{N}]\big|^p \bigg]\le C
        \bigg( \E \bigg[\int_0^t \sup_{s \in [0,r]} \Wc_p \big(\varphi^{N,\Xbb,\overline{\alpha}^{N,-i}}_s, \varphi^{N,\Xbb,\overline{\alpha}^{N}}_s \big)^p +  \Wc_p \big(\varphi^{N,\overline{\alpha}^{N,-i}}_r, \varphi^{N,\overline{\alpha}^{N}}_r \big)^p \mathrm{d}r \bigg]\bigg).
    \end{align}
    Next, using result \eqref{eq:first_estimation},
    \begin{align*}
        &\E\bigg[\sup_{s \in [0,t]} \Wc_p \big(\varphi^{N,\Xbb,\overline{\alpha}^{N,-i}}_s, \varphi^{N,\Xbb,\overline{\alpha}^{N}}_s \big)^p
        +
        \Wc_p \big(\varphi^{N,\overline{\alpha}^{N,-i}}_t, \varphi^{N,\overline{\alpha}^{N}}_t \big)^p\bigg]
        \\
        &~~~~~~~~\le C \bigg(
        \frac{1}{N} \sum_{j=1}^N\E \bigg[\sup_{s \in [0,t]} |\Xbb^j_s[\overline{\alpha}^{N,-i}]-\Xbb^j_s[\overline{\alpha}^{N}]|^p \bigg]
        +
        \frac{\rho \big( \kappa^N_t,\alpha^i_t \big)^p}{N} \bigg)
        \\
         &~~~~~~~~\le C \bigg( \frac{1}{N} \sum_{j=1,j\neq i}^N \E \bigg[\int_0^t \sup_{s \in [0,r]} \Wc_p \big(\varphi^{N,\Xbb,\overline{\alpha}^{N,-i}}_s, \varphi^{N,\Xbb,\overline{\alpha}^{N}}_s \big)^p +  \Wc_p \big(\varphi^{N,\overline{\alpha}^{N,-i}}_r, \varphi^{N,\overline{\alpha}^{N}}_r \big)^p \mathrm{d}r \bigg]
         \\
        &~~~~~~~~~~~~~~~~~~+
         \frac{1}{N} \E \bigg[\sup_{s \in [0,t]} \big|\Xbb^i_s[\overline{\alpha}^{N,-i}]-\Xbb^i_s[\overline{\alpha}^{N}]\big|^p \bigg]
         +
        \frac{\rho \big( \kappa^N_t,\alpha^i_t \big)^p}{N}
         \bigg)
         \\
         &~~~~~~~~\le C \bigg( \frac{N-1}{N} \E \bigg[\int_0^t \sup_{s \in [0,r]} \Wc_p \big(\varphi^{N,\Xbb,\overline{\alpha}^{N,-i}}_s, \varphi^{N,\Xbb,\overline{\alpha}^{N}}_s \big)^p +  \Wc_p \big(\varphi^{N,\overline{\alpha}^{N,-i}}_r, \varphi^{N,\overline{\alpha}^{N}}_r \big)^p \mathrm{d}r \bigg]
        \\
        &~~~~~~~~~~~~~~~~~~+
         \frac{\int_{\R^n}|x|^p\nu(\mathrm{d}x)}{N} 
         +
        \frac{\sup_{(u,u') \in U \x U} \rho \big( u,u' \big)^p}{N}
         \bigg),
    \end{align*}
    by Gronwall's lemma again,
    \begin{align}
    \label{eq:second_estimation}
        &\E \bigg[\sup_{s \in [0,t]} \Wc_p \big(\varphi^{N,\Xbb,\overline{\alpha}^{N,-i}}_s, \varphi^{N,\Xbb,\overline{\alpha}^{N}}_s \big)^p
        +
        \Wc_p \big(\varphi^{N,\overline{\alpha}^{N,-i}}_t, \varphi^{N,\overline{\alpha}^{N}}_t \big)^p\bigg]
        \le C \bigg(
         \frac{\int_{\R^n}|x|^p\nu(\mathrm{d}x)}{N} 
         +
        \frac{\sup_{(u,u') \in U \x U} \rho \big( u,u' \big)^p}{N}
         \bigg).
    \end{align}
    To finish, 
    \begin{align*}
        &\E \bigg[\sup_{s \in [0,t]} \big|\Xbb^i_s[\overline{\alpha}^{N,-i}]-Z^i_s \big|^p \bigg]
        \\
        &~~~~~~~~\le
        C \bigg( \E \bigg[\int_0^t \big|\big[b,\sigma \big] \big(r,\Xbb^i_r[\overline{\alpha}^{N,-i}],\varphi^{N,\Xbb,\overline{\alpha}^{N,-i}},\varphi^{N,\overline{\alpha}^{N,-i}}_{r} ,\kappa^N_r \big) 
        -
        \big[b,\sigma \big] \big(r,Z^i_r,\varphi^{N,\Xbb,\overline{\alpha}^{N}},\varphi^{N,\overline{\alpha}^{N}}_{r} ,\kappa^N_r \big) 
        \big|^p \mathrm{d}r \bigg] \bigg)
        \\
        &~~~~~~~~\le C
        \bigg( \E \bigg[\int_0^t \sup_{s \in [0,r]} |\Xbb^j_s[\overline{\alpha}^{N,-i}]-\Xh^i_s|^p + \sup_{s \in [0,r]} \Wc_p \big(\varphi^{N,\Xbb,\overline{\alpha}^{N,-i}}_s, \varphi^{N,\Xbb,\overline{\alpha}^{N}}_s \big)^p +  \Wc_p \big(\varphi^{N,\overline{\alpha}^{N,-i}}_r, \varphi^{N,\overline{\alpha}^{N}}_r \big)^p \mathrm{d}r \bigg]\bigg),
    \end{align*}
    and thanks to Gronwall's lemma and result \eqref{eq:second_estimation}, one has
    \begin{align*}
        \E\bigg[\sup_{s \in [0,T]} |\Xbb^i_s[\overline{\alpha}^{N,-i}]-Z^i_s|^p \bigg]
        \le C T
        \bigg( \frac{\int_{\R^n}|x|^p\nu(\mathrm{d}x)}{N} 
         +
        \frac{\sup_{(u,u') \in U \x U} \rho \big( u,u' \big)^p}{N} \bigg).
    \end{align*}
    It is enough to conclude.
\end{proof}

\medskip

\medskip
The next result is the analog of Lemma \ref{lemm:convergenceMFGstrong} for the $N$--player game. To summarize, it states that any measure--valued control rule which verifies a particular constraint is the average limit of $N$--SDE processes of type \eqref{eq:N-agents_StrongMF_CommonNoise}.

\begin{lemma}
\label{lemm:convergenceNashEquilibrium}
    Let {\rm \Cref{assum:main1}} hold true, $\mathrm{P} \in \Pcb_V$ and a sequence $(\alpha^{i})_{i \in \N^* }$ s.t. for each $N \in \N^*,$ $(\alpha^{1},\cdots,\alpha^{N}) \subset \Ac^N$ and
    \begin{align*}
        \Lim_{N \to \infty} \P \circ \Big( \varphi^{N,\Xbb,\overline{\alpha}}, \delta_{ \varphi^{N,\overline{\alpha}}_t} (\mathrm{d}m)\mathrm{d}t, B \Big)^{-1}
        =
        \mathrm{P} \circ \big( \mu,\Lambda,B \big)^{-1}.
    \end{align*}
    There exists a sequence of Borel functions $(R^{i,N})_{(i,N) \in \{1,\dots,N\} \x \N^*}$ satisfying $R^{i,N}: [0,T] \x (\R^n)^N \x (\Cc^d)^N \x \Cc^\ell \x [0,1] \to U$ s.t. if for all $(t,f) \in [0,T] \x [0,1],$ $\kappa^{i,N}_t(f)$ is defined by {\color{blue}$\kappa^{i,N}_t(f):=R^{i,N} \big(t,\xi^1,\cdots,\xi^N,W^1_{t \wedge \cdot},\cdots,W^N_{t \wedge \cdot}, B_{t \wedge \cdot},f \big),$} then one has
\begin{align*}
    &\Lim_{N \to \infty}
    \frac{1}{N} \sum_{i=1}^N
    \int_0^1 J_i\big(\alpha^{1},\dots,\alpha^{i-1},\kappa^{i,N}(f),\alpha^{i+1,N},\dots,\alpha^{N} \big) \mathrm{d}f
        =
        \E^{\mathrm{P}}\big[J(\mu',\mu, \Lambda',\Lambda) \big].
\end{align*}
\end{lemma}

\begin{proof}
By Lemma \ref{lemma:reformulation-approximation}, there is a sequence $(G^l)_{l \in \N^*},$ such that for each $l \in \N^*,$  $G^l:[0,T] \x \Cc^\ell \x \Cc^n_{\Wc} \x \M(\Pc^n_U) \to \Pc(\R^n \x U)$ is a continuous function and
    
    \begin{align*}
        \Lim_{l \to \infty}\Lc^{\mathrm{P}} \Big( \delta_{G^l \big(t,B_{t \wedge \cdot},\mu_{t \wedge \cdot}, \Lambda_{t \wedge \cdot} \big)}(\mathrm{d}m)\mathrm{d}t,B,\mu,\Lambda \Big)
        =
        \Lc^{\mathrm{P}} \big( \Theta',B,\mu,\Lambda\big),
    \end{align*}    
recall that $\Theta'$ is defined in \eqref{eq:shift-proba}. 
Now, we apply \cite[Proposition 4.7]{MFD-2020}. We know that we can find a sub--sequence $(l_N)_{N \in \N^*} \subset \N^*$ such that if $\Lambda^N:=\delta_{\varphi^{N,\overline{\alpha}^N}_t}(\mathrm{d}m)\mathrm{d}t,$ 
    \begin{align*}
        \mb'^{N}_t
        :=
        G^{l_N} \big(t,B_{t \wedge \cdot},\varphi^{N,\Xbb,\overline{\alpha}^N}_{t \wedge \cdot}, \Lambda^N_{t \wedge \cdot} \big)
        \;\mbox{and}\;
        \Theta'^{N}
        :=
        \delta_{\mb'^{N}_t}(\mathrm{d}m)\mathrm{d}t,
    \end{align*}
    one has
\begin{align*}
        &\Lim_{N \to \infty} 
        \P \circ \big( \Theta'^{N}, \varphi^{N,\Xbb,\overline{\alpha}^N}, \Lambda^N, B \big)^{-1}
        =
        \Lim_{l \to \infty} \Lc^{\mathrm{P}} \Big(\delta_{G^l \big(t,B_{t \wedge \cdot},\mu_{t \wedge \cdot}, \Lambda_{t \wedge \cdot} \big)}(\mathrm{d}m)\mathrm{d}t,\mu,\Lambda,B \Big)
        =
        \mathrm{P} \circ \big( \Theta',\mu,\Lambda,B \big)^{-1}.
\end{align*}
Under \Cref{assum:main1}, by \cite[Proposition 4.7]{MFD-2020} (with separability condition see \cite[Remark 4.11]{MFD-2020}), there exist a $[0,1]$--valued uniform random variable $F$ independent of $(\xi^i,W^i,B)_{i \in \N^*}$ and a Borel function $R^{N}: [0,T] \x \R^n  \x \Cc^n_{\Wc} \x \M \x \M \x \Cc^n \x \Cc^\ell \x [0,1] \to U$ s.t. if  $(\Zh^{i})_{i \in \{1,\dots,N\}}$ is the unique strong solution of: for all $t \in [0,T]$
\begin{align} \label{eq:general-weak_McK}
    \mathrm{d}\Zh^{i}_t
    =
    \widehat{b}\big(t,\Zh^{i}_t,B,\varphi^{N,\Xbb,\overline{\alpha}^N}, \varphi^{N,\overline{\alpha}^N}_t,\gamma^{i,N}_{t}\big) \mathrm{d}t
    +
    \widehat{\sigma} \big(t,\Zh^{i}_t,B,\varphi^{N,\Xbb,\overline{\alpha}^N}, \varphi^{N,\overline{\alpha}^N}_t,\gamma^{i,N}_{t}\big)
    \mathrm{d}W^i_t\;\mbox{with}\;\Zh^i_0=\xi^i
\end{align}
    where 
\begin{align*}
    \gamma^{i,N}_t:= R^N\big(t,\xi^i,\varphi^{N,\Xbb,\overline{\alpha}^N}_{t \wedge \cdot},\Theta^{N}_{t \wedge \cdot},\Lambda^N_{t \wedge \cdot},W^i_{t \wedge \cdot}, B_{t \wedge \cdot},F \big),\;\;\overline{\vartheta}^{N}_t:=\frac{1}{N}\sum_{i=1}^N\delta_{(\Zh^i_t,\gamma^{i,N}_t)}\;\mbox{and}\;\vartheta^{N}_t:=\frac{1}{N}\sum_{i=1}^N\delta_{\Zh^i_t},
\end{align*}
then $\Lim_{N \to \infty} \E \bigg[\int_0^T \Wc_p \big(\overline{\vartheta}^{N}_t,\mb'^{N}_t \big)^p \mathrm{d}t \bigg]=0,$ and
    \begin{align*}
        \Lim_{j \to \infty} \Lc \big(\vartheta^{N_j},V^{N_j},\varphi^{N_j,\Xbb,\overline{\alpha}^{N_j}},\Lambda^{N_j},B \big)
        =
        \Lc^{\mathrm{P}} \big(\vartheta',\Theta',\mu,\Lambda,B \big)\;\mbox{in}\;\Wc_p,
    \end{align*}
    with $V^{N}:=\delta_{\overline{\vartheta}^{N}_t}(\mathrm{d}m)\mathrm{d}t$ and $(N_j)_{j \in \N^*} \subset \N^*$ is a sub--sequence.
    
    As in the proof Lemma \ref{lemm:convergenceMFGstrong}, we can rewrite $(Z^{i})_{i \in \{1,\dots,N\}}.$
Notice that
\begin{align*} 
    \mathrm{d}\Zh^i_t
    =
    b\big(t,\Zh^i_t+\sigma_0B_t,\varphi^{N,\Xbb,\overline{\alpha}^N},\varphi^{N,\overline{\alpha}^N}_t,\gamma^{i,N}_{t}\big) \mathrm{d}t
    +
    \sigma \big(t,\Zh^i_t+\sigma_0B_t,\varphi^{N,\Xbb,\overline{\alpha}^N},\varphi^{N,\overline{\alpha}^N}_t,\gamma^{i,N}_{t}\big)
    \mathrm{d}W^i_t.
\end{align*}
Denote $Z^i:=\Zh^i+\sigma_0B,$ then
\begin{align*} 
    \mathrm{d}Z^i_t
    =
    b\big(t,Z^i_t,\varphi^{N,\Xbb,\overline{\alpha}^N},\varphi^{N,\overline{\alpha}^N}_t,\gamma^{i,N}_{t}\big) \mathrm{d}t
    +
    \sigma \big(t,Z^i_t,\varphi^{N,\Xbb,\overline{\alpha}^N},\varphi^{N,\overline{\alpha}^N}_t,\gamma^{i,N}_{t}\big)
    \mathrm{d}W^i_t+\sigma_0 \mathrm{d}B_t.
\end{align*}

As the function $(\pi,q,\bb) \in \Cc^n_{\Wc} \x \M \x \Cc^\ell \to \big(\pi[\bb],q_t[\bb](\mathrm{d}m)\mathrm{d}t,\bb \big) \in \Cc^n_{\Wc} \x \M \x \Cc^\ell$ is continuous, if we define $\overline{\beta}^{N}_t:=\frac{1}{N}\sum_{i=1}^N\delta_{(Z^i_t,\gamma^{i,N}_t)},\;\mbox{and}\;\beta^{N}_t:=\frac{1}{N}\sum_{i=1}^N\delta_{Z^i_t},$ one has, in $\Wc_p,$
    \begin{align*}
        &\Lim_{j \to \infty} \Lc \big( \beta^{N_j}, \delta_{\overline{\beta}^{N_j}_t} (\mathrm{d}m)\mathrm{d}t,\varphi^{N_j,\Xbb,\overline{\alpha}^{N_j}},\Lambda^{N_j}_t, B  \big)
        \\
        &=
        \Lim_{j \to \infty} \Lc \big(\vartheta^{N_j}[B],V^{N_j}_t[B](\mathrm{d}m)\mathrm{d}t,\varphi^{N_j,\Xbb,\overline{\alpha}^{N_j}},\Lambda^{N_j}, B \big)
        =
        \Lc^{\mathrm{P}} \big(\vartheta'[B],\Theta'_t[B](\mathrm{d}m)\mathrm{d}t,\mu,\Lambda,B \big).
    \end{align*}
One knows that $(\vartheta'[B],\Theta'_t[B](\mathrm{d}m)\mathrm{d}t,B)=(\mu',\Lambda',B),$ $\mathrm{P}$--a.e. then   
 \begin{align} \label{eq:convergence_result}
        &\Lim_{j \to \infty} \Lc \Big( \beta^{N_j}, \delta_{\overline{\beta}^{N_j}_t} (\mathrm{d}m)\mathrm{d}t,\varphi^{N_j,\Xbb,\overline{\alpha}^{N_j}},\Lambda^{N_j}, B  \Big)
        =
        \Lc^{\mathrm{P}} \big(\mu',\Lambda',\mu,\Lambda,B \big)\;\mbox{in}\;\Wc_p.
    \end{align}

\medskip
Let us define 
    \begin{align*}
        \overline{\alpha}^{N,-i}:= (\overline{\alpha}^{[-i]},\gamma^{i,N})=\big(\alpha^1,\dots,\alpha^{i-1},\gamma^{i,N},\alpha^{i+1},\dots,\alpha^N \big),
    \end{align*}
thanks to Lemma \ref{lemma:estimation_convergence}, for each $i \in \{1,\dots,N\},$
\begin{align*}
    \bigg( \E \bigg[\int_0^T \Wc_p \big(\varphi^{N,\overline{\alpha}^N}_{r},\varphi^{N,\overline{\alpha}^{N,-i}}_r \big)^p \mathrm{d}r \bigg] + \E \bigg[ \sup_{t \in [0,T]} |Z^i_t-\Xbb^i_t[\overline{\alpha}^{N,-i}]|^p \bigg] \bigg)
	\le
	K \frac{1}{N},
\end{align*}
and $\Limsup_{N \to \infty} \Wc_p \big(\Q^N, \widetilde{\Q}^N \big)=0,$ where 
$$
    \Q^N:=\frac{1}{N} \sum_{i=1}^N\Lc \Big( \Xbb^i[\overline{\alpha}^{N,-i}], \varphi^{N,\Xbb,\overline{\alpha}^{N,-i}}, \delta_{\big(\gamma^{i,N}_t,\varphi^{N,\overline{\alpha}^{N,-i}}_t \big)}(\mathrm{d}u,\mathrm{d}m)\mathrm{d}t \Big)
$$
and 
$$
    \widetilde{\Q}^N:=\frac{1}{N} \sum_{i=1}^N\Lc \Big( Z^i, \varphi^{N,\Xbb,\overline{\alpha}^{N}}, \delta_{\big(\gamma^{i,N}_s,\varphi^{N,\overline{\alpha}^{N}}_t \big)}(\mathrm{d}u,\mathrm{d}m)\mathrm{d}t \Big).
$$

\medskip
Now, for each $f \in [0,1],$ we pose $\kappa^{i,N}_t(f):=R^N\big(t,\xi^i,\varphi^{N,\Xbb,\overline{\alpha}^N}_{t \wedge \cdot},\Theta^{N}_{t \wedge \cdot},\Lambda^N_{t \wedge \cdot},W^i_{t \wedge \cdot}, B_{t \wedge \cdot},f \big).$ Notice that $\kappa^{i,N}(f) \in \Ac^N$ and $\kappa^{i,N}_t(F)=\gamma^{i,N}.$ 

\medskip
Therefore, using \Cref{assum:main1} (especially the separability condition), the previous result combined with \eqref{eq:convergence_result} allow to get that
\begin{align*}
    &\Lim_{N \to \infty}
    \frac{1}{N} \sum_{i=1}^N \int_0^1 J_i \big( (\overline{\alpha}^{[-i]},\kappa^{i,N}(f)) \big) \mathrm{d}f
    \\
    &=
    \Lim_{N \to \infty}
    \frac{1}{N} \sum_{i=1}^N
    \E \bigg[
        \int_0^T L\big(t,\Xbb^{i}_t[\overline{\alpha}^{N,-i}],\varphi^{N,\Xbb,\overline{\alpha}^{N,-i}},\varphi^{N,\overline{\alpha}^{N,-i}}_{t} ,\gamma^{i,N}_t \big) \mathrm{d}t 
        + 
        g \big( \Xbb^{i}_T[\overline{\alpha}^{N,-i}], \varphi^{N,\Xbb,\overline{\alpha}^{N,-i}} \big)
        \bigg]
        \\
    &=
    \Lim_{N \to \infty}
    \frac{1}{N} \sum_{i=1}^N
    \E \bigg[
        \int_0^T L\big(t,Z^{i}_t,\varphi^{N,\Xbb,\overline{\alpha}^{N}},\varphi^{N,\overline{\alpha}^{N}}_{t} ,\gamma^{i,N}_t \big) \mathrm{d}t 
        + 
        g \big( Z^{i}_T, \varphi^{N,\Xbb,\overline{\alpha}^{N}} \big)
        \bigg]
        =
        \E^{\mathrm{P}}\big[J(\mu',\mu, \Lambda',\Lambda) \big].
\end{align*}
	
\end{proof}

\subsubsection{Proof of Theorem \ref{thm:limitNashEquilibrium} (Limit Theorem)}

\paragraph*{First point (i)} By using \cite[Proposition 4.15]{MFD-2020} (a slight extension\footnote{consisting in taking into account a canonical space of type $\Omb:=\Cc^n_{\Wc} \x \Cc^n_{\Wc} \x \M \x \M \x \Cc^\ell$ and not $\Omb:=\Cc^n_{\Wc} \x \M \x \Cc^\ell$ as in \cite{MFD-2020} } ), one finds $(\mathrm{P}^N)_{N \in \N^*}$ is relatively compact where
    \begin{align*}
        \mathrm{P}^N
        :=
        \P \circ \Big( (\varphi^{N,\Xbb,\overline{\alpha}^N}_t)_{t \in [0,T]},(\varphi^{N,\Xbb,\overline{\alpha}^N}_t)_{t \in [0,T]},\delta_{\varphi^{N,\overline{\alpha}^N}_t}(\mathrm{d}m)\mathrm{d}s, \delta_{\varphi^{N,\overline{\alpha}^N}_t}(\mathrm{d}mt)\mathrm{d}t, B  \Big)^{-1},
    \end{align*}
    and each limit point $\mathrm{P}^\infty$ of any sub--sequence belongs to $\Pcb_V.$ Next, let us show that $\mathrm{P}^\infty \in \Pcb^\star_V[\varepsilon].$ To simplify, the sequence $(\mathrm{P}^N)_{N \in \N^*}$ and its sub--sequence share the same notation. 

\medskip    
    Let $\mathrm{P} \in \Pcb_V$ such that $\Lc^{\mathrm{P}} \big(\mu, \Lambda, B \big)=\Lc^{\mathrm{P}^\infty} \big(\mu, \Lambda, B \big).$ By \Cref{lemm:convergenceNashEquilibrium}, there exist a $[0,1]$--valued uniform random variable $F$ and $(R^{i,N})_{(i,N) \in \{1,\dots,N\} \x \N^*}$ a sequence of Borel functions $R^{i,N}: [0,T] \x (\R^n)^N \x (\Cc^d)^N \x \Cc^\ell \x [0,1] \to U$ s.t. if we denote by 
    $$
        \kappa^{i,N}_t(F):=R^{i,N} \big(t,\xi^1,\cdots,\xi^N,W^1_{t \wedge \cdot},\cdots,W^N_{t \wedge \cdot}, B_{t \wedge \cdot}, F\big)
    $$
    then
\begin{align*}
    \Lim_{N \to \infty}
    \frac{1}{N} \sum_{i=1}^N
    \int_0^1 J_i\big(\alpha^1,\dots,\alpha^{i-1},\kappa^{i,N}(f),\alpha^{i+1},\dots,\alpha^N \big) \mathrm{d}f
        =
        \E^{\mathrm{P}}\big[J(\mu',\mu, \Lambda',\Lambda) \big].
\end{align*}
Therefore, as for each $N,$ $(\alpha^1,\cdots,\alpha^N)$ is an $(\varepsilon_1,\cdots,\varepsilon_N)$--Nash equilibrium,
\begin{align*}
    \E^{\mathrm{P}^\infty}\big[J(\mu',\mu, \Lambda',\Lambda) \big]
    &=
    \Lim_{N \to \infty} \frac{1}{N} \sum_{i=1}^N J_i [\overline{\alpha}^N]
    \\
    &\ge 
    \Lim_{N \to \infty} \bigg(\frac{1}{N} \sum_{i=1}^N
    \int_0^1 J_i\big(\alpha^1,\dots,\alpha^{i-1},\kappa^{i,N}(f),\alpha^{i+1},\dots,\alpha^N \big) \mathrm{d}f - \frac{1}{N} \sum_{i=1}^N \varepsilon_i \bigg)
    =
    \E^{\mathrm{P}}\big[J(\mu',\mu, \Lambda',\Lambda) \big]-\varepsilon,
\end{align*}
then $\E^{\mathrm{P}^\infty}\big[J(\mu',\mu, \Lambda',\Lambda)\big] \ge \E^{\mathrm{P}}\big[J(\mu',\mu, \Lambda',\Lambda)\big] -\varepsilon,$ for any $\mathrm{P} \in \Pcb_V$ such that $\Lc^{\mathrm{P}} \big(\mu, \Lambda, B \big)=\Lc^{\mathrm{P}^\infty} \big(\mu, \Lambda, B \big).$ It is straightforward to deduce that  for $\mathrm{P}^{\infty}$ almost every $\om \in \Omb,$ $\Lambda'_t(\om)(\mathrm{d}m)\mathrm{d}t=\Lambda_t(\om)(\mathrm{d}m )\mathrm{d}t$ and $\mu'(\om)=\mu(\om).$ We conclude that $\mathrm{P}^\infty \in \Pcb_V^\star[\varepsilon].$

    \paragraph*{Second point (ii)} The proof of this second part is similar to the previous proof. By using \cite[Proposition 4.15]{MFD-2020} (a slight extension), one gets $(\mathrm{P}^k)_{k \in \N^*}$ is relatively compact where $\mathrm{P}^k \in \Pcb^\star_S[\varepsilon_k]$ i.e. there exists $\alpha^k$ an $\varepsilon_k$--strong MFG equilibrium s.t.
    \begin{align*}
        \mathrm{P}^k
        :=
        \P\circ \Big( (\mu^{\alpha^k}_t)_{t \in [0,T]}, (\mu^{\alpha^k}_t)_{t \in [0,T]},\delta_{\mub^{\alpha^k}_s}(\mathrm{d}m) \mathrm{d}s, \delta_{\mub^{\alpha^k}_s}(\mathrm{d}m') \mathrm{d}s, B \Big)^{-1}.
    \end{align*}
    Each limit point $\mathrm{P}^\infty$ of any sub--sequence belongs to $\Pcb_V.$ Let us prove that $\mathrm{P}^\infty \in \Pcb^\star_V[\varepsilon].$ Again to simplify, $(\mathrm{P}^k)_{k \in \N^*}$ and its sub--sequence share the same notation. Let $\mathrm{P} \in \Pcb_V$ such that $\Lc^{\mathrm{P}} \big(\mu, \Lambda, B \big)=\Lc^{\mathrm{P}^\infty} \big(\mu, \Lambda, B \big).$ By Lemma \ref{lemm:convergenceMFGstrong}, there exists a family of probability $(\Pr^k_f)_{(k,f) \in \N^* \x [0,1]} \subset \Pcb_S$ such that for each $k \in \N^*$:
    \begin{align*}
        f \ni [0,1] \to \Pr^k_f \in \Pc(\Omb)\;\mbox{is Borel measurable},\;\;\Pr^k_f \circ (\mu,\Lambda,B)^{-1}=\P \circ \big( \mu^{\alpha^k}, \delta_{\mub^{\alpha^k}_t}(\mathrm{d}m) \mathrm{d}t, B \big)^{-1}\;\mbox{for each}\;f \in [0,1],
    \end{align*}
    and for a sub--sequence $(k_j)_{j \in \N^*},$
    \begin{align*}
        &\Lim_{j \to \infty}
        \int_0^1 \E^{\Pr^{k_j}_f}\big[J(\mu', \mu,\Lambda',\Lambda) \big]  \mathrm{d}f
        =
        \E^{\Pr}\big[J(\mu', \mu, \Lambda',\Lambda) \big].
    \end{align*}
    Then, using Assumption \ref{assum:main1} (especially separability condition), as $\alpha^k$ is an $\varepsilon_k$--strong MFG solution, one gets
    \begin{align*}
        \E^{\mathrm{P}^{\infty}}\big[J(\mu', \mu, \Lambda',\Lambda) \big]
        &=
        \Lim_{j \to \infty} \E^{\mathrm{P}^{k_j}}\big[J(\mu', \mu, \Lambda',\Lambda) \big]
        \\
        &\ge
        \Lim_{j \to \infty} \bigg(\int_0^1 \E^{\Pr^{k_j}_f}\big[J(\mu', \mu, \Lambda',\Lambda) \big]  \mathrm{d}f - \varepsilon_{k_j} \bigg)
		=
		\E^{\mathrm{P}}\big[J(\mu', \mu, \Lambda',\Lambda) \big]-\varepsilon.
    \end{align*}
    Obviously, for $\mathrm{P}^{\infty}$ almost every $\om \in \Omb,$ $\Lambda'_t(\om)(\mathrm{d}m)\mathrm{d}t=\Lambda_t(\om)\big(\mathrm{d}m \big)\mathrm{d}t$ and $\mu'(\om)=\mu(\om),$ we deduce that $\mathrm{P}^\infty \in \Pcb^\star_V[\varepsilon].$  
    
    \medskip
    
    {\color{black}

\begin{proof}[Proof of \Cref{prop-strong-as-measure}] \label{proof-prop-strong-as-measure}
    Let $\alpha$ be an $\varepsilon$--strong MFG equilibrium, and its corresponding probability $\mathrm{P}^\alpha \in \Pcb_S^\star[\varepsilon].$ It is straightforward to check that $\mathrm{P}^\alpha \in \Pcb_V.$ Let $\mathrm{P} \in \Pcb_V$ such that $\Lc^{\mathrm{P}^\alpha}\big(\mu,\Lambda, B\big)=\Lc^{\mathrm{P}}\big(\mu,\Lambda,B \big).$
    By \Cref{lemm:convergenceMFGstrong}, there exists a family of probability $(\Pr^k_f)_{(k,f) \in \N^* \x [0,1]} \subset \Pcb_S$ such that for each $k \in \N^*$:
    \begin{align*}
        f \ni [0,1] \to \Pr^k_f \in \Pc(\Omb)\;\mbox{is Borel measurable},\;\;\Pr^k_f \circ (\mu,\Lambda,B)^{-1}=\Lc^{\mathrm{P}^\alpha}\big(\mu,\Lambda, B\big)\;\mbox{for each}\;f \in [0,1],
    \end{align*}
    and for a sub--sequence $(k_j)_{j \in \N^*},$
    \begin{align*}
        &\Lim_{j \to \infty}
        \int_0^1 \E^{\Pr^{k_j}_f}\big[J(\mu',\mu, \Lambda',\Lambda) \big]  \mathrm{d}f
        =
        \E^{\Pr}\big[J(\mu',\mu, \Lambda',\Lambda) \big].
    \end{align*}    
    Consequently,
    \begin{align*}
        \E^{\mathrm{P}^\alpha} \big[J(\mu',\mu, \Lambda',\Lambda) \big]
        =
        \Psi(\alpha,\alpha )
        \ge
        \Lim_{k \to \infty} \int_0^1 \E^{\Pr^{k_j}_f}\big[J(\mu',\mu, \Lambda',\Lambda) \big]  \mathrm{d}f - \varepsilon
        =
        \E^{\mathrm{P}} \big[J(\mu',\mu, \Lambda',\Lambda) \big] - \varepsilon.
    \end{align*}
    As obviously $\Lambda_t(\mathrm{d}m)\mathrm{d}t=\Lambda'_t(\mathrm{d}m)\mathrm{d}t$ and $\mu=\mu',$ $\mathrm{P}^\alpha$--a.e., we can deduce that $\mathrm{P}^\alpha \in \Pcb_V^\star[\varepsilon]$ and conclude the proof
\end{proof}

    }

\subsection{The converse limit result} \label{sec:converse-limit}

This part is devoted to the proof of Theorem \ref{thm:converselimitNash}. We focus on the approximation of any measure--valued MFG solution by a sequence of approximate strong MFG solutions. The approximation by approximate Nash equilibria follows from this approximation. 

{\color{black}
\paragraph*{Problem and strategy of the proof}  Let $\Pr^\star$ be a measure--valued MFG solution. First, we find a sequence $(\alpha^k)_{k \in \N^*} \subset \Ac$ such that the sequence $(\Pr^{\alpha^k})_{k \in \N^*} \subset \Pcb_S$ (see \Cref{eq:McV-measure-v}) converges towards $\Pr^\star.$ Then, we find a sequence $(\varepsilon_k)_{k \in \N^*} \subset (0,\infty)$ satisfying $\Lim_{k \to \infty} \varepsilon_k=0$ and $\Pr^{\alpha^k} \in \Pcb_S[\varepsilon_k]$ for each $k \in \N^*.$ The sequence $(\varepsilon_k)_{k \in \N^*} \subset (0,\infty)$ is given by (recall that $\Psi(\alpha,\alpha')$ is given in \Cref{def:strong_value-function})
\begin{align*}
    \varepsilon_k
    :=
    \sup_{\alpha' \in \Ac} \Psi(\alpha^k,\alpha')-\Psi(\alpha^k,\alpha^k).
\end{align*}
It is obvious that $\varepsilon_k \ge 0$ for each $k.$ The difficulty is to show that $\Lim_{k \to \infty} \varepsilon_k=0.$ For proving this, it is enough to show that: for any $(\alpha'^k)_{k \in \N^*} \subset \Ac,$ we can find a sequence $(\mathrm{Q}^k)_{k \in \N^*} \subset \Pcb_V$ such that $\Lc^{\mathrm{Q}^k}(\mu,\Lambda,B)=\Lc^{\Pr^\star}(\mu,\Lambda,B)$ and 
\begin{align} \label{eq-conv-nash-property}
    \Lim_{k \to \infty}
    \big| \E^{\mathrm{Q}^k}[J(\mu',\mu,\Lambda',\Lambda)] - \E^{\Pr^{\alpha^k,\alpha'^k}}[J(\mu',\mu,\Lambda',\Lambda)] \big|
    =
    0.
\end{align}
For establishing property \eqref{eq-conv-nash-property}, how the sequence $(\alpha^k)_{k \in \N^*}$ is constructed is crucial. In the next parts, we will show the main points of the construction of $(\alpha^k)_{k \in \N^*}$ which will help us proving the property \eqref{eq-conv-nash-property}.
}

\medskip

\subsubsection{Strong controlled McKean--Vlasov processes as approximation of weak controlled McKean--Vlasov processes}

{\color{black}

The approximation of any measure--valued MFG solution $\Pr^\star$ by a sequence of distributions of McKean--Vlasov process $(\Pr^{\alpha^k})_{k \in \N^*} \subset \Pcb_S$ is achieved in two steps. First, we approximate $\Pr^\star$ by a sequence of ``weak'' controlled McKean--Vlasov processes $(Z^k)_{k \in \N^*}.$ Here,  ``weak'' means essentially that $Z^k$ satisfies \Cref{eq:MKV_strong} with a control $\alpha^k$ not adapted to $\F$ and $\mub^{\alpha^k}$ is not adapted to the filtration of $B$. And the second step is, from $(Z^k)_{k \in \N^*}$, to build another approximation which satisfies: $\alpha^k$ adapted to $\F$ and $\mub^{\alpha^k}$ is adapted to the filtration of $B$ i.e. strong controlled McKean--Vlasov processes. The second step is the most delicate. In this section, we focus on the approximation of weak controlled McKean--Vlasov process by strong controlled McKean--Vlasov processes. In the following, we give the definition of weak controlled McKean--Vlasov process that we use and the corresponding approximation by strong controlled McKean--Vlasov processes. This part is largely inspired/borrowed by Section 4.1.1. of \citeauthor*{djete2019general} \cite{djete2019general}.

\medskip
\begin{definition} \label{def:weak-solution}
    A weak controlled McKean--Vlasov process is a tuple $(\Omt,\Ft,\Pt,\Xt,\Wt,\Bt,\alephh,\alphat)$ satisfying:
\begin{enumerate}
    \item $(\Omt,\Ft,\Pt)$ is a complete filtered probability space. $(\Wt,\Bt)$ is a $\R^{d + \ell}$--valued $\Ft$--Brownian motion. $\Xt$ is a $\R^n$--valued $\Ft$--adapted continuous process with $\Pt \circ (\Xt_0)^{-1}=\nu.$ $\alphat$ is a $\Ft$--predictable process. Finally, $\alephh$ is a $\Pc(\Cc^n \x \M(U) \x \R^n)$--valued $\Ft$--adapted continuous process.
    
    \item $\Wt,\Xt_0$ and $(\alephh,\Bt)$ are independent.
    
    \item $\alephh$ verifies the condition: for all $t \in [0,T],$ $\Pt$--a.s.
    \begin{align} \label{eq:Hhypothesis-weak}
        \alephh_t
        =
        \Lc^{\Pt}(\Xt_{t \wedge \cdot}, \Deltat_{t \wedge \cdot}, \Wt| \Bt_{t \wedge \cdot}, \alephh_{t \wedge \cdot})
        =
        \Lc^{\Pt}(\Xt_{t \wedge \cdot}, \Deltat_{t \wedge \cdot}, \Wt| \Bt, \alephh)\;\mbox{where}\;\Deltat_t(\mathrm{d}u)\mathrm{d}t:=\delta_{\tilde \alpha_t}(\mathrm{d}u)\mathrm{d}t.
    \end{align}
    \item $\Xt$ satisfies
    \begin{align}
    \label{eq:weak_McK-SDE}
        \mathrm{d}\Xt_t= 
        b\big(t,\Xt_t,\aleph,\alephb_t,\alphat_t\big)  \mathrm{d}t 
        + 
        \sigma \big(t,\Xt_t,\aleph,\alephb_t,\alphat_t\big)  \mathrm{d}\Wt_t 
        +  
        \sigma_0 \mathrm{d}\Bt_t
    \end{align}
    where $\alephb_t:=\Lc^{\Pt}(\Xt_t,\alphat_t|\Bt,\alephh)$ and $\aleph_t:=\Lc^{\Pt}(\Xt_t|\Bt,\alephh).$
\end{enumerate}
\end{definition}

\paragraph*{Intuition/Idea of the proof} Given a weak controlled McKean--Vlasov process $(\Omt,\Ft,\Pt,\Xt,\Wt,\Bt,\alephh,\alphat),$ we want to build a sequence $(X^{\alpha^k},\alpha^k)_{k \in \N^*}$ converging in a weak sense towards the weak controlled McKean--Vlasov process and satisfying: $X^{\alpha^k}$ is solution of \Cref{eq:MKV_strong} and $\alpha^k \in \Ac$. In another words, we want to find a sequence of weak controlled McKean--Vlasov processes where $\alphat$ is a function of $(\Xt_0,\Wt,\Bt)$ and $\alephh$ is a function of $\Bt.$
The intuition is the following: first, under \Cref{assum:main1}, \Cref{eq:weak_McK-SDE} admits a unique solution $\Xt,$ then there exists a Borel function $G:\R^n \x \Cc^n \x \Cc^\ell \x \M(U) \x \Pc(\Cc^n \x \M(U) \x \Cc^n) \to \Cc^n$ s.t. $\Xt=G(\Xt_0,\Wt,\Bt,\Deltat,\alephh)$ $\Pt$--a.s. This observation leads us to deal only with $(\Xt_0,\Wt,\Bt,\Deltat,\alephh)$ because $\Xt$ is already a function of these variables. Second, it is well know that we can find a uniform random variable $F^{\tilde \alpha}$ independent of $(\Xt_0,\Wt,\Bt,\alephh)$ and a Borel function $G^{\tilde \alpha}:\R^n \x \Cc^n \x \Cc^\ell \x \Pc(\Cc^n \x \M(U) \x \Cc^n) \x [0,1] \to \M(U)$ s.t.
\begin{align*}
    \Lc^{\Pt}(\Xt_0,\Wt,\Bt,\alephh,\Deltat)=\Lc^{\Pt}(\Xt_0,\Wt,\Bt,\alephh,G^{\tilde \alpha}(\Xt_0,\Wt,\Bt,\alephh,F^{\tilde \alpha})).
\end{align*}
Now, we deal with $(\Xt_0,\Wt,\Bt,\alephh).$ As $\Wt,\Xt_0$ and $(\alephh,\Bt)$ are independent. We use the same previous procedure. We can find a uniform random variable $F^{\hat \aleph}$ independent of $(\Xt_0,\Wt,\Bt,F^{\tilde \alpha})$ and a Borel function $G^{\hat \mu}:\Cc^\ell \x [0,1] \to \Pc(\Cc^n \x \M(U) \x \Cc^n)$ s.t.
\begin{align*}
    \Lc^{\Pt}(\Xt_0,\Wt,\Bt,\alephh)=\Lc^{\Pt}(\Xt_0,\Wt,\Bt,G^{\hat \aleph}(\Bt,F^{\hat \aleph})).
\end{align*}
These equalities in distribution leads to 
\begin{align*}
    \Lc^{\Pt}(\Xt,\Wt,\Bt,\alephh,\Deltat)=\Lc^{\Pt}(\widetilde{S},\Wt,\Bt,\widetilde{G}^{\hat \aleph},\widetilde{G}^{\tilde \alpha})
\end{align*}
where $\widetilde{S}:=G(\Xt_0,\Wt,\Bt,\widetilde{G}^{\tilde \alpha},\widetilde{G}^{\hat \aleph})$, $\widetilde{G}^{\tilde \alpha}:=G^{\tilde \alpha}(\Xt_0,\Wt,\Bt,G^{\hat \aleph}(\Bt,F^{\hat \aleph}),F^{\tilde \alpha})$ and  $\widetilde{G}^{\hat \aleph}:=G^{\hat \aleph}(\Bt,F^{\hat \aleph}).$ The equality in distribution allows to check that $(\widetilde{S},\Wt,\Bt,\widetilde{G}^{\hat \aleph},\widetilde{G}^{\tilde \alpha})$ satisfies the same equation as $(\Xt,\Wt,\Bt,\alephh,\Deltat).$ $(\widetilde{S},\Wt,\Bt,\widetilde{G}^{\hat \aleph},\widetilde{G}^{\tilde \alpha})$ is almost what we want except two points: first, the presence of uniform random variables $(F^{\tilde \alpha},F^{\hat \aleph}),$ and second, we do not know if  $\widetilde{G}^{\hat \aleph}$ and $\widetilde{G}^{\tilde \alpha}$ are adapted to the filtrations of $(\Xt_0,\Wt,\Bt)$ and $\Bt$ respectively. The technical proof of \cite[Section 4.1.1.]{djete2019general} allows to ``remove'' $(F^{\tilde \alpha},F^{\hat \aleph})$ and to make $(\widetilde{G}^{\hat \aleph},\widetilde{G}^{\tilde \alpha})$ an adapted process by using approximations and the condition satisfies by $\alephh$ i.e. \Cref{eq:Hhypothesis-weak}.
The next Lemma is a combination of \cite[Lemma 4.3, Lemma 4.4]{djete2019general}. 

\begin{lemma} \label{lemma_weak-strong_general}
    There exit: a sequence of positive real $(\varepsilon_k)_{k \in \N^*}$ s.t. $\Lim_{k \to \infty} \varepsilon_k=0,$ and 
    \begin{itemize}
        
        \item a sequence $(\Kt^k,\alphat^k)_{k \in \N^*}$ where for each $k \in \N^*,$ $\alphat^k$ is a $U$--valued piecewise $\Ft$--predictable process and $\Kt^k$ satisfies 
        \begin{align*}
            \mathrm{d}\Kt^k_t= 
            b\big(t,\Kt^k_t,\pi^k,\pib^k_t,\alphat^k_t\big) \mathbf{1}_{t \in [\varepsilon_k,T]}  \mathrm{d}t 
            + 
            \sigma \big(t,\Kt^k_t,\pi^k,\pib^k_t,\alphat^k_t\big)  \mathrm{d}\Wt^k_t 
            +  
            \sigma_0 \mathrm{d}\Bt^k_t\;\mbox{with}\;\Kt^k_0=\Xt_0,
        \end{align*}
        where $\Bt^k_\cdot:=\Bt_{\cdot \vee \varepsilon_k}-\Bt_{\varepsilon_k},$ $\Wt^k_\cdot:=\Wt_{\cdot \vee \varepsilon_k}-\Wt_{\varepsilon_k},$ $\pib^k_t:=\Lc^{\Pt}(\Kt^k_t,\alphat^k_t|\Bt^k,\alephh)$ and $\pi^k_t:=\Lc^{\Pt}(\Kt^k_t|\Bt^k,\alephh);$
        
        \item a sequence of uniform random variable $(F^k)_{k \in \N^*}$ independent of $(\Xt_0,\Wt,\Bt^k)$ and a sequence $(\St^k,\gammat^k)_{k \in \N^*}$ where for each $k \in \N^*,$ $\gammat^k$ is a $U$--valued $(\sigma\{F^k, \Xt_0,\Wt_{t \wedge \cdot},\Bt^k_{t \wedge \cdot} \})_{t \in [0,T]}$--predictable process and $\Zt^k$ satisfies 
        \begin{align*}
            \mathrm{d}\St^k_t= 
            b\big(t,\St^k_t,\zeta^k,\zetab^k_t,\gammat^k_t\big) \mathbf{1}_{t \in [\varepsilon_k,T]}  \mathrm{d}t 
            + 
            \sigma \big(t,\St^k_t,\zeta^k,\zetab^k_t,\gammat^k_t\big)  \mathrm{d}\Wt^k_t 
            +  
            \sigma_0 \mathrm{d}\Bt^k_t
        \end{align*}
        where $\zetab^k_t:=\Lc^{\Pt}(\St_t,\gammat^k_t|\Bt^k,F^k)$ and $\zeta^k_t:=\Lc^{\Pt}(\St^k_t|\Bt^k,F^k);$

    \end{itemize}
    all these processes satisfy
    \begin{align} \label{conv-aux-weak-strong}
        \Lim_{k \to \infty}\E^{\Pt} \bigg[ \int_0^T \big[\rho(\alphat^k_t,\alphat_t)^2 + \Wc_{2}(\alephb_t,\pib^k_t) \big] \mathrm{d}t + \sup_{t \in [0,T]} \big[ |\Xt_t-\Kt^k_t|^2 + \Wc_{2}(\aleph_t,\pi^k_t) \big] \bigg] =0
    \end{align}
    and for each $k \in \N^*,$
    \begin{align*}
        \Lc^{\Pt} \big(\Kt^k,\Deltat^k,\Wt^k,\Bt^k,\Pih^k \big)
        =
        \Lc^{\Pt} \big(\St^k,\betat^k,\Wt^k,\Bt^k,\zetah^k \big),
    \end{align*}
    where $\Deltat^k:=\delta_{\tilde \alpha_t}(\mathrm{d}u)\mathrm{d}t,$ $\betat^k:=\delta_{\tilde \gamma_t}(\mathrm{d}u)\mathrm{d}t,$ $\Pih^k:=\Lc^{\Pt}(\Kt^k_{t \wedge \cdot}, \Deltat^k_{t \wedge \cdot}, \Wt^k| \Bt^k, \alephh)$ and $\zetah^k:=\Lc^{\Pt}(\St^k_{t \wedge \cdot}, \betat^k_{t \wedge \cdot}, \Wt^k| \Bt^k, F^k).$
\end{lemma}

Given the result of the previous \Cref{lemma_weak-strong_general} with the same notation, let us assume first that $\ell \neq 0.$ Then for each $k \in \N^*,$ we pose $F^k:=\phi(\Bt_{\varepsilon_k})$ where $\phi:\R^\ell \to [0,1]$ is a Borel function such that $\Lc^{\Pt}(\phi(\Bt_{\varepsilon_k}))=\Uc([0,1]).$ Therefore  $(\gammat^k)_{k \in \N^*}$ is a $U$--valued $(\sigma\{ \Xt_0,\Wt_{t \wedge \cdot},\Bt_{t \wedge \cdot} \})_{t \in [0,T]}$--predictable processes and $(\zetah^k)_{k \in \N^*}$ is $B$--measurable i.e. $\zetah^k_t:=\Lc^{\Pt}(\St^k_{t \wedge \cdot}, \betat^k_{t \wedge \cdot}, \Wt^k| \Bt).$  The next proposition shows the sequence of strong controlled McKean--Vlasov processes converges to the weak controlled McKean--Vlasov process. This proposition is easily obtained with the previous \Cref{lemma_weak-strong_general}.
\begin{proposition}{\rm($\ell \neq 0$)} \label{prop:weak-strong-common_noise}
    For each $k \in \N^*,$ let $\Ert^k$ be the solution of 
    \begin{align*}
            \mathrm{d}\Ert^k_t= 
            b\big(t,\Ert^k_t,\mu^k,\mub^k_t,\gammat^k_t\big)  \mathrm{d}t 
            + 
            \sigma \big(t,\Ert^k_t,\mu^k,\mub^k_t,\gammat^k_t\big)  \mathrm{d}\Wt_t 
            +  
            \sigma_0 \mathrm{d}\Bt_t\;\mbox{with}\;\Ert^k_0=\Xt_0,\;\mu^k_t:=\Lc^{\Pt}(\Ert^k_t|\Bt)\;\mbox{and}\;\mub^k_t:=\Lc^{\Pt}(\Ert^k_t,\gammat^k_t|\Bt)
        \end{align*}
        then $\big(\Lc^{\Pt} \big(\mu^k,\mu^k,\delta_{\mub^k_t}(\mathrm{d}m)\mathrm{d}t,\delta_{\mub^k_t}(\mathrm{d}m)\mathrm{d}t, B \big) \big)_{k \in \N^*} \subset \Pcb_S,$
        \begin{align*}
            \Lim_{k \to \infty} \E^{\Pt}\bigg[ \sup_{t \in [0,T]} |\Ert^k_t-\St^k_t|^p \bigg]=0\;\mbox{and}\;\Lim_{k \to \infty} \Lc^{\Pt} \big(\Ert^k,\betat^k,\Wt,\Bt,\muh^k \big)=\Lc^{\Pt} \big(\Xt,\Deltat,\Wt,\Bt,\alephh \big)\;\mbox{in}\;\Wc_p,
        \end{align*}
        where $\muh^k_t:=\Lc^{\Pt}(\Ert^k_{t \wedge \cdot}, \betat^k_{t \wedge \cdot}, \Wt| \Bt)$ a.s. for all $t \in [0,T].$
\end{proposition}

Now, when $\ell=0,$ then there is no common noise $B.$ In that case we obtain a convex combination of strong controlled McKean--Vlasov.

\begin{proposition}{\rm($\ell = 0$)} \label{prop:weak-strong-no_common_noise}
    let $\Ert^k$ be the solution of 
    \begin{align*}
            \mathrm{d}\Ert^k_t= 
            b\big(t,\Ert^k_t,\mu^k,\mub^k_t,\gammat^k_t\big)  \mathrm{d}t 
            + 
            \sigma \big(t,\Ert^k_t,\mu^k,\mub^k_t,\gammat^k_t\big)  \mathrm{d}\Wt_t 
            \;\mbox{with}\;\Ert^k_0=\Xt_0,\;\mu^k_t:=\Lc^{\Pt}(\Ert^k_t|F^k)\;\mbox{and}\;\mub^k_t:=\Lc^{\Pt}(\Ert^k_t,\gammat^k_t|F^k)
        \end{align*}
        then a.s. $\big(\Lc^{\Pt} \big(\mu^k,\mu^k,\delta_{\mub^k_t}(\mathrm{d}m)\mathrm{d}t,\delta_{\mub^k_t}(\mathrm{d}m)\mathrm{d}t \big| F^k \big) \big)_{k \in \N^*} \subset \Pcb_S,$
        \begin{align*}
            \Lim_{k \to \infty} \E^{\Pt}\bigg[ \sup_{t \in [0,T]} |\Ert^k_t-\St^k_t|^p \bigg]=0\;\mbox{and}\;\Lim_{k \to \infty} \Lc^{\Pt} \big(\Ert^k,\betat^k,\Wt,\Bt,\muh^k \big)=\Lc^{\Pt} \big(\Xt,\Deltat,\Wt,\Bt,\alephh \big)\;\mbox{in}\;\Wc_p,
        \end{align*}
        where $\muh^k_t:=\Lc^{\Pt}(\Ert^k_{t \wedge \cdot}, \betat^k_{t \wedge \cdot}, \Wt| F^k)$ a.s. for all $t \in [0,T].$
\end{proposition}

\medskip
$\mathbf{From\;now\;on\;and\;in\;the\;following,\;we\;will\;work\;with}$ $\ell \neq 0.$ $\mathbf{The\;case}$ $\ell=0$ $\mathbf{can\;be\;easily\;adapted\;from}$ \newline $\mathbf{\;the\;case}$ $\ell \neq 0.$

\medskip
Now, using \Cref{lemma_weak-strong_general}, we will give a first result which will help us to prove \Cref{eq-conv-nash-property} which is our main goal in this part.
Recall that the weak controlled McKean--Vlasov process $(\Omt,\Ft,\Pt,\Xt,\Wt,\Bt,\alephh,\alphat)$ is fixed. Let $(\gammat'^k)_{k\in \N^*}$ be a sequence such that for each $k \in \N^*,$ there exists a Borel function $\varphi^k:[0,T] \x \R^n \x \Cc^n \x \Cc^\ell \to U$ with $\gammat'^k_t=\varphi^k(t,\Xt_0,\Wt_{t \wedge \cdot},\Bt_{t \wedge \cdot})$ a.s. for each $t \in [0,T].$ For each $k \in \N^*$, we consider $\Ert'^k$ the strong solution of:
\begin{align*}
    \mathrm{d}\Ert'^k_t= 
            b\big(t,\Ert'^k_t,\mu^k,\mub^k_t,\gammat'^k_t\big)  \mathrm{d}t 
            + 
            \sigma \big(t,\Ert'^k_t,\mu^k,\mub^k_t,\gammat'^k_t\big)  \mathrm{d}\Wt_t 
            +  
            \sigma_0 \mathrm{d}\Bt_t\;\mbox{with}\;\Zt'_0=\Xt_0.
\end{align*}
$\mu^k,\mub^k$ and $\muh^k$ are defined in \Cref{prop:weak-strong-common_noise}.

\begin{lemma}
\label{lemma:weak_fixedAppr}
    There exists a sequence of $U$--valued $\Ft$--predictable processes $(\alphat'^k_t)_{t \in [0,T]}$ s.t if for each $k \in \N^*$, we let $\Xt'^k$ be the solution of:
    \begin{align*}
            \mathrm{d}\Xt'^k_t= 
            b\big(t,\Xt'^k_t,\aleph,\alephb_t,\alphat'^k_t\big)  \mathrm{d}t 
            + 
            \sigma \big(t,\Xt'^k_t,\aleph,\alephb_t,\alphat'^k_t\big)  \mathrm{d}\Wt_t 
            +  
            \sigma_0 \mathrm{d}\Bt_t\;\mbox{with}\;\Xt'^k_0=\Xt_0
        \end{align*}
       then
    \begin{align}  \label{eq:convergence-strong_weakMFG}
        \Lim_{k \to \infty} \Wc_p \bigg(\Pt \circ \Big(\Ert'^{k}, \muh^{k}, \delta_{(\mub^{k}_t,\tilde \gamma'^k_t)}(\mathrm{d}m,\mathrm{d}u)\mathrm{d}t \Big)^{-1}, \Pt \circ \Big( \Xt'^k,\alephh,\delta_{(\alephb_t,\tilde \gamma^k_t)}(\mathrm{d}m,\mathrm{d}u)\mathrm{d}t \Big)^{-1} \bigg)
        =
        0.
    \end{align}
    
\end{lemma}

\begin{proof}
    For each $k \in \N^*$, let us define $\St'^k$ the solution of 
    \begin{align*}
        \mathrm{d}\St'^k_t= 
            b\big(t,\St'^k_t,\zeta^k,\zetab^k_t,\gammat'^k_t\big) \mathbf{1}_{t \in [\varepsilon_k,T]}  \mathrm{d}t 
            + 
            \sigma \big(t,\St'^k_t,\zeta^k,\zetab^k_t,\gammat'^k_t\big)  \mathrm{d}\Wt^k_t 
            +  
            \sigma_0 \mathrm{d}\Bt^k_t\;\mbox{with}\;\St'_0=\Xt_0.
    \end{align*}
    Using \Cref{prop:weak-strong-common_noise}, as $\zetah^k_t=\Lc^{\Pt}(\St^k_{t \wedge \cdot}, \betat^k_{t \wedge \cdot}, \Wt^k| \Bt)$ and, $\Lim_{k \to \infty}\varepsilon_k=0$ and $\Lim_{k \to \infty}(\Wt^k,\Bt^k)=(\Wt,\Bt)$ a.s.,  we know that $\Lim_{k \to \infty} \Wc_p((\mu^k_t,\mub^k_t,\muh^k_t),(\zeta^k_t,\zetab^k_t,\zetah^k_t))=0.$  Then, it is easy to check that 
    \begin{align} \label{eq:strong-aux-cong}
        \Lim_{k \to \infty}\E^{\Pt} \bigg[ \sup_{t \in [0,T]} |\St'^k_t-\Ert'^k_t|^p \bigg] =0.
    \end{align}
    
    For each $k \in \N^*,$ in the function $\varphi^k,$ in order to emphasize the part of the Brownian motion $(W,B)$ before and after $\varepsilon_k,$ we rewrite
    \begin{align*}
        \varphi^k(t,\Xt_0,\Wt_{t \wedge \cdot},\Bt_{t \wedge \cdot})
        =
        \varphit^k(t, \Xt_0, \Wt^k_{t \wedge \cdot}, \Wt_{(t \wedge \varepsilon_k) \wedge \cdot}, \Bt^k_{t \wedge \cdot}, \Bt_{(t \wedge \varepsilon_k) \wedge \cdot}).
    \end{align*}
    
    Let $(\Wt^\circ,\Bt^\circ)$ be a $\R^{d + \ell}$--valued $\Ft$--Brownian motion $\Pt$--independent of $(\Wt,\Bt,\Xt_0,\alephh,\alphat, (F^k)_{k \in \N^*}).$ We define
    \begin{align*}
        \alphat'^k_t
        :=
        \varphit(t, \Xt_0, \Wt^k_{t \wedge \cdot}, \Wt^\circ_{(t \wedge \varepsilon_k) \wedge \cdot}, \Bt^k_{t \wedge \cdot}, \Bt^\circ_{(t \wedge \varepsilon_k) \wedge \cdot})\;\mbox{a.s. for all}\;t \in [0,T],\;\Deltat'^k:=\delta_{\tilde \alpha'^k_t}(\mathrm{d}u)\mathrm{d}t\;\mbox{and}\;\betat'^k:=\delta_{\tilde \gamma'^k_t}(\mathrm{d}u)\mathrm{d}t.
    \end{align*}
    It is straightforward that we can find a Borel function $\Phi:\R^n \x \M(U) \x \Cc^n \x \Cc^\ell \x \Pc(\Cc^n \x \M(U) \x \Cc^n) \to \Cc^n$ s.t. $\St'^k=\Phi(\Xt_0,\betat'^k,\Wt^k,\Bt^k,\zetah^k)$ a.s. Let us introduce the process $(\Kt'_t)_{t \in [0,T]}$
    $$
        \Kt':=\Phi(\Xt_0,\Deltat'^k,\Wt^k,\Bt^k,\Pih^k).
    $$
    Using \Cref{lemma_weak-strong_general} , especially $\Lc \big(\Xt^k,\Deltat^k,\Wt^k,\Bt^k,\Pih^k \big)=\Lc \big(\St^k,\betat^k,\Wt^k,\Bt^k,\zetah^k \big),$ we find that
    \begin{align} \label{eq:law-eq-aux}
        \Lc \big(\St'^k,\betat'^k,\Wt^k,\Bt^k,\Pih^k \big)=\Lc \big(\Kt'^k,\Deltat'^k,\Wt^k,\Bt^k,\zetah^k \big).
    \end{align}
    Notice that $\Kt'$ is the solution of 
    \begin{align*}
            \mathrm{d}\Kt'^k_t= 
            b\big(t,\Kt'^k_t,\pi^k,\pib^k_t,\alphat'^k_t\big) \mathbf{1}_{t \in [\varepsilon_k,T]}  \mathrm{d}t 
            + 
            \sigma \big(t,\Kt'^k_t,\pi^k,\pib^k_t,\alphat'^k_t\big)  \mathrm{d}\Wt^k_t 
            +  
            \sigma_0 \mathrm{d}\Bt^k_t\;\mbox{with}\;\Kt'^k_0=\Xt_0.
        \end{align*}
        
        To conclude the prove of the Lemma, we introduce $\Xt'^k$ solution of
        \begin{align*}
            \mathrm{d}\Xt'^k_t= 
            b\big(t,\Xt'^k_t,\aleph,\alephb_t,\alphat'^k_t\big)  \mathrm{d}t 
            + 
            \sigma \big(t,\Xt'^k_t,\aleph,\alephb_t,\alphat'^k_t\big)  \mathrm{d}\Wt_t 
            +  
            \sigma_0 \mathrm{d}\Bt_t\;\mbox{with}\;\Xt'^k_0=\Xt_0.
        \end{align*}
    Similarly to \eqref{eq:strong-aux-cong}, thanks to \eqref{conv-aux-weak-strong} of \Cref{lemma_weak-strong_general}, it is easy to verify that $\Lim_{k \to \infty}\E^{\Pt} \bigg[ \sup_{t \in [0,T]} |\Xt'^k_t-\Kt'^k_t|^p \bigg] =0.$ 
    By combining the previous convergence, \Cref{eq:law-eq-aux} and the convergence \eqref{eq:strong-aux-cong}, we deduce the convergence \eqref{eq:convergence-strong_weakMFG}. 
    
\end{proof}

\subsubsection{Approximation of measure--valued solution via weak controlled McKean--Vlasov processes}


\medskip
Let $\Pr \in \Pcb_V$ such that $\Pr(\Lambda'=\Lambda, \mu'=\mu)=1.$ Let $(\Omt,\Ft:=(\Fct_t)_{t \in [0,T]},\widetilde{\Pr})$ be an extension of $(\Omb,\Fb,\Pr)$ supporting a uniform random variable $F,$ a $\R^d$--valued $\Ft$--Brownian motion $\Wt$ and a $\R^n$--valued $\Fct_0$--random variable $\Xt_0$ with $\Lc(\Xt_0)=\nu$. Besides, $F,$ $\Wt,$ $\Xt_0$ and $\Gcb_T$ are independent. $(\mu,\Lambda,B)$ are naturally extended on $(\Omt,\Ft:=(\Fct_t)_{t \in [0,T]},\widetilde{\Pr}).$ The next Lemma is exactly \cite[Proposition 5.3]{MFD-2020}.

\begin{lemma} \label{lemma:app-measure-weak}
     There exists a sequence of $U$--valued $\big(\sigma\{F,\Xt_0,\Wt_{t \wedge \cdot}\} \lor \Gcb_t \big)_{t \in [0,T]}$--predictable processes $(\alphat^k)_{k \in \N^*}$ st. if we let $\Xt^k$ be the solution of
    \begin{align*}
            \mathrm{d}\Xt^k_t= 
            b\big(t,\Xt^k_t,\aleph^k,\alephb^k_t,\alphat^k_t\big)  \mathrm{d}t 
            + 
            \sigma \big(t,\Xt^k_t,\aleph^k,\alephb^k_t,\alphat^k_t\big)  \mathrm{d}\Wt_t 
            +  
            \sigma_0 \mathrm{d}\Bt_t\;\mbox{with}\;\Xt^k_0=\Xt_0
        \end{align*}
        where $\aleph^k_t=\Lc^{\Pt}(\Xt^k_t|\Gcb_t)$ and $\alephb^k_t=\Lc^{\Pt}(\Xt^k_t,\alphat^k_t|\Gcb_t).$ Then for a sub--sequence $(k_j)_{j \in \N^*}$
    \begin{align*}
        \Lim_{j \to \infty} \Bigg[ \Wc_p \Big(\delta_{\alephb^{k_j}_t}(\mathrm{d}m)\mathrm{d}t,\Lambda \Big)^p +
        \sup_{t \in [0,T]} \Wc_p(\aleph^{k_j}_t,\mu_t)
        \bigg]
        =
        0,\;\Prt\mbox{--a.s.}
    \end{align*}
\end{lemma}

\begin{remark}
{
   If we define $\alephh^k_t:=\Lc^{\Prt}(\Xt^k_{t \wedge \cdot}, \Art^k_{t \wedge \cdot}, \Wt| \Gcb_T)$ with $\Art^k:=\delta_{\tilde \alpha^k_t}(\mathrm{d}u)\mathrm{d}t.$ It is straightforward to check that the tuple $(\Omt,\Ft,\Prt,\Xt^k,\Wt^k,B,\alephh^k,\alphat^k)$ is a weak controlled McKean--Vlasov in the sense of {\rm \Cref{def:weak-solution}} $($see also {\rm\cite[Definition 2.3]{djete2019general}}$)$.
    }
\end{remark}

\medskip

Now, given the sequence $(\Xt^k,\alphat^k)_{k \in \N^*}$ of \Cref{lemma:app-measure-weak}, we consider a sequence of $U$--valued $\big(\sigma\{F,\Xt_0,\Wt_{t \wedge \cdot}\} \lor \Gcb_t \big)_{t \in [0,T]}$--predictable processes $(\alphat'^k)_{k \in \N^*}$ and let $\Xt'^k$ be the solution of
    \begin{align*}
            \mathrm{d}\Xt'^k_t= 
            b\big(t,\Xt'^k_t,\aleph^k,\alephb^k_t,\alphat'^k_t\big)  \mathrm{d}t 
            + 
            \sigma \big(t,\Xt'^k_t,\aleph^k,\alephb^k_t,\alphat'^k_t\big)  \mathrm{d}\Wt_t 
            +  
            \sigma_0 \mathrm{d}\Bt_t\;\mbox{with}\;\Xt'^k_0=\Xt_0.
        \end{align*}

\begin{lemma}
\label{lemma:converse_general}
   There exists a sequence $(\mathrm{Q}^k)_{k \in \N^*} \subset \Pcb_V$ satisfying $\mathrm{Q}^k \circ \big(\mu,\Lambda,B \big)^{-1}=\mathrm{P} \circ \big(\mu,\Lambda,B \big)^{-1}$ for each $k \in \N^*$ and
    \begin{align} \label{eq:convergence-relaxed_strongMFG}
        \Lim_{k \to \infty} \Wc_p \Big(\Prt \circ \big(\aleph'^{k},\aleph^k, \delta_{\alephb'^{k}_t}(\mathrm{d}m)\mathrm{d}t, \delta_{\alephb^{k}_t}(\mathrm{d}m)\mathrm{d}t, B \big)^{-1},\mathrm{Q}^k \Big)
        =
        0,
    \end{align}
    where $\aleph'^{k}_t:=\Lc^{\Prt}\big(\Xt'^k_t \big| \Gcb_t \big)$ and $\alephb'^{k}_t:=\Lc^{\Prt}\big(\Xt'^k_t,\alpha'^k_t \big| \Gcb_t \big)$ a.s. for all $t \in [0,T].$
\end{lemma}

\begin{proof}
    Let us take a convergent sub--sequence of $\big(\Prt \circ \big(\aleph'^{k},\aleph^k, \delta_{\alephb'^{k}_t}(\mathrm{d}m)\mathrm{d}t, \delta_{\alephb^{k}_t}(\mathrm{d}m)\mathrm{d}t, B \big)^{-1} \big)_{k \in \N^*}$ (possible because it is relatively compact see for instance \cite[Proposition 5.4]{MFD-2020}), denote by $\mathrm{P}^\infty$ its limit. One uses the same notation for the sub--sequence. The limit satisfies: $N_t(f)=0,$ $\mathrm{P}^\infty$--a.e., for all $t \in [0,T]$ and $f \in C^{2}_b(\R^n),$ where we recall that $(\mu',\mu,\Lambda',\Lambda,B)$ is the canonical variable on $\Omb:=(\Cc^n_{\Wc})^2 \x \M(\Pc^n_U)^2 \x \Cc^\ell,$ and $\Lambda'_t (\Z_{\mu'_t})=1,$ $\mathrm{d}\mathrm{P}^\infty \otimes \mathrm{d}t$--a.e. $(t,\om) \in [0,T] \x \Omb.$ Notice that, by previous \Cref{lemma:app-measure-weak},  $\Lim_{k \to \infty} (\delta_{\alephb^k_t}(\mathrm{d}m)\mathrm{d}t,\aleph^k)=(\Lambda,\mu)$ in $\Wc_p,$ $\Prt$--a.s. Then, one has
    \begin{align*}
        \Lim_{k \to \infty} \Wc_p \Big(\Prt \circ \big(\aleph^k, \delta_{\alephb'^{k}_t}(\mathrm{d}m)\mathrm{d}t, \delta_{\alephb^{k}_t}(\mathrm{d}m)\mathrm{d}t, B \big)^{-1}, \Prt \circ \big(\mu, \delta_{\alephb'^{k}_t}(\mathrm{d}m)\mathrm{d}t, \Lambda, B \big)^{-1} \Big)=0.
    \end{align*}
    Then, it is enough to apply \cite[Proposition 4.10]{MFD-2020} (see also \cite[Proposition 4.9]{MFD-2020}) and It\^{o}'s formula to conclude the proof.
    Indeed, there exists $(\beta'^{k})_{k \in \N^*}$ a sequence of $\Pc(U)$--valued $\big((\sigma\{ F,\Xt_0,\Wt_{t \wedge \cdot} \} \lor \Gcb_t) \otimes \Bc(\Pc^n_U) \big)_{t \in [0,T]}$--predictable processes such that if $(\St'^k_t)_{t \in [0,T]}$ is the unique strong solution of: $\St'_0=\Xt_0$ and
    \begin{align*}
        \mathrm{d}\St'^k_t
        =
        &
        \int_{\Pc^n_U} \int_U b \big(t,\St'^k_t,\mu,\nub,u \big) \beta'^{k}_{t}(\alephb'^k_t)(\mathrm{d}u) \Lambda_t(\mathrm{d}\nub) \mathrm{d}t
        +
        \bigg( \int_{\Pc^n_U} \int_U \sigma \sigma^\top \big(t,\St'^k_t,\mu,\nub,u \big) \beta'^{k}_{t}(\alephb'^k_t)(\mathrm{d}u) \Lambda_t(\mathrm{d}\nub)  \bigg)^{1/2} \mathrm{d}\Wt_t + \sigma_0 \mathrm{d}\Bt_t,
    \end{align*}
    then, one has, for a sub--sequence $(k_j)_{j \in \N^*} \subset \N^*,$ 
    \begin{align*}
        \Lim_{j \to \infty} \E^{\Prt} \bigg[ \int_0^T \Wc_{p} \big(\mbb'^{k_j}_t,\alephb'^k_t \big) \mathrm{d}t\bigg]=0\;\mbox{and}\; \Lim_{j \to \infty} \E^{\Prt} \bigg[\sup_{s \in [0,T]}  \Wc_p ( \vartheta'^{k_j}_t,\aleph'^{k_j}_t ) \bigg]=0,
    \end{align*}
    
    where
    $$
        \mbb'^{k}_t(\mathrm{d}x,\mathrm{d}u)
        :=
        \E^{\Prt}
        \Big[
        \beta'^{k}_t(\alephb^k_t)(\mathrm{d}u) \delta_{\St'^{k}_t}(\mathrm{d}x) \Big|\Gcb_t
        \Big]
        ~
        \mbox{and}
        ~
        \vartheta'^{k}_t:=\Lc^{\Prt}(\St'^{k}_t\big|\Gcb_t)
        ~~
        \mbox{for all}
        ~
        t \in [0,T].
    $$
    
    We define the sequence $(\mathrm{Q}^k)_{k \in \N^*} \subset \Pcb_V$ as follows
    \begin{align*}
        \mathrm{Q}^k
        :=
        \Prt \circ \big( \vartheta'^k, \mu, \delta_{\mbb'^k_t}(\mathrm{d}m) \mathrm{d}t, \Lambda, B \big)^{-1}.
    \end{align*}
    $(\mathrm{Q}^k)_{k \in \N^*}$ is the sequence we are looking for. 
    

\end{proof}

Now, the next proposition states our main objective which is convergence \eqref{eq-conv-nash-property}. The proposition is just a simple combination on the one hand \Cref{prop:weak-strong-common_noise} (see also \Cref{prop:weak-strong-no_common_noise} for $\ell=0$) and \Cref{lemma:app-measure-weak}, and on the other hand \Cref{lemma:converse_general} and \Cref{lemma:weak_fixedAppr}. The proof is therefore omitted.

\begin{proposition} \label{prop:converse_general}
    Let $\Pr \in \Pcb_V$ s.t. $\Pr(\mu'=\mu,\Lambdap=\Lambda)=1.$ There exists $(\alpha^k)_{k \in \N^*} \subset \Ac$ s.t.
    \begin{align*}
        \Lim_{k \to \infty}\Wc_p \big(\Pr^{\alpha^k},\Pr \big)\;\mbox{with for each}\;k\in \N^*\;\Pr^{\alpha^k}\;\mbox{is defined in {\rm\Cref{eq:McV-measure-v}}}.
    \end{align*}
    In addition, for any $(\alpha'^k)_{k \in \N^*} \subset \Ac,$ we can find a sequence $(\mathrm{Q}^k)_{k \in \N^*} \subset \Pcb_V$ such that $\Lc^{\mathrm{Q}^k}(\mu,\Lambda,B)=\Lc^{\Pr}(\mu,\Lambda,B)$ and 
\begin{align} \label{eq:app-strong-meassure}
    \Lim_{k \to \infty}
    \big| \E^{\mathrm{Q}^k}[J(\mu',\mu,\Lambdap,\Lambda)] - \E^{\Pr^{\alpha^k,\beta^k}}[J(\mu',\mu,\Lambdap,\Lambda)] \big|
    =
    0.
\end{align}
\end{proposition}

\begin{remark}
{
    We stress again that it is not easy to find a sequence of measure--valued control rules verifying {\rm\Cref{eq:app-strong-meassure}}. Indeed, notice that the set $\Pcb_V$ is not a closed set in general. Therefore a classical compactness argument does not work here.
    }
\end{remark}

}


\subsubsection{Approximation of strong controlled McKean--Vlasov processes via $N$--interacting controlled processes}

In this part, we provide the analog of \Cref{prop:converse_general} for the $N$--player game.

\begin{proposition}
\label{lemma:converse_convergence}
    Under {\rm\Cref{assum:main1}}, for any $\alpha \in \Ac,$ there exists a sequence $(\alpha^{i,N})_{(i,N) \in \{1,\dots,N\} \x \N^*}$ satisfying for each $N \in \N^*,$ $(\alpha^{i,N})_{i \in \{1,\dots,N\}} \subset \Ac^N$ s.t.
    \begin{align*}
        \Lim_{N \to \infty} \P \circ \big( \varphi^{N,\Xbb,\overline{\alpha}^N}, \delta_{\varphi^{N,\overline{\alpha}^N}_t}(\mathrm{d}m)\mathrm{d}t,B \big)^{-1}
        =
        \P \circ \big( \mu^\alpha, \delta_{\mub^\alpha_t}(\mathrm{d}m)\mathrm{d}t, B \big)^{-1},
    \end{align*}
    where $\overline{\alpha}^N=(\alpha^{i,N},\dots,\alpha^{i,N}).$
    
    \medskip
    In addition, for any sequence $(\kappa^{i,N})_{(i,N) \in \{1,\dots,N\} \x \N^*}$ satisfying for each $N \in \N^*,$ $(\kappa^{i,N})_{i \in \{1,\dots,N\}} \subset \Ac^N,$ there exists a sequence $(\Pr^{i,N})_{(i,N) \in \{1,\dots,N\} \x \N^*} \subset \Pcb_S$ such that $\Pr^{i,N} \circ \big( \mu,\Lambda,B \big)^{-1}=\P \circ \Big( \mu^\alpha, \delta_{\mub^\alpha_t}(\mathrm{d}m)\mathrm{d}t, B \Big)^{-1}$ for each $(i,N)$ and
    \begin{align*}
        \Lim_{N \to \infty} \bigg|\frac{1}{N} \sum_{i=1}^{N} J_i\big( (\overline{\alpha}^{[-i]},\kappa^{i,N})\big)- \frac{1}{N} \sum_{i=1}^{N} \E^{\mathrm{P}^{i,N}}\big[J(\mu',\mu, \Lambda', \Lambda) \big] \bigg|
        =
        0.
    \end{align*}
   
\end{proposition}

\begin{proof}
    
    Using (an extention of) \cite[Proposition 4.15]{djete2019general}, there exists $(\psi^{i,N})_{(i,N)}$ a sequence of Borel functions $\psi^{i,N}:[0,T] \x \R^n \x \Cc^n \x \Cc^\ell \to U$ such that if $\alpha^{i,N}_t:=\psi^{i,N}(t,\xi^i,W^i_{t \wedge \cdot},B_{t \wedge \cdot})$ for all $t \in [0,T],$ and $\overline{\alpha}^N=(\alpha^{i,N},\dots,\alpha^{i,N}),$ then one has
    \begin{align*}
        \lim_{N \to \infty} \E^{\P} \bigg[ \sup_{t \in [0,T]} \Wc_p \big(\varphi^{N,\Xbb,\overline{\alpha}^N}_t,\mu^\alpha_t \big) + \int_0^T \Wc_p \big( \varphi^{N,\overline{\alpha}^N}_t,\mub^\alpha_t\big) \mathrm{d}t \bigg]=0. 
    \end{align*}
    
    Next, by easy adaptation of \Cref{lemma:estimation_convergence} (successive application of Gronwall's lemma), there exists a sequence $(C_N)_{N \in \N^*}$ converging to zero when $N$ goes to infinity satisfying: for each $i \in \{1,\dots,N\},$ if $\overline{\alpha}^{N,-i}:=(\overline{\alpha}^{-i},\kappa^{i,N}),$ 
    \begin{align*}
        \E^{\P} \bigg[~ \int_0^T \Wc_p \big( \varphi^{N,\overline{\alpha}^{N,-i}}_t,\mub^\alpha_t\big) \mathrm{d}t +  \sup_{t \in [0,T]} \Wc_p \big(\varphi^{N,\Xbb,\overline{\alpha}^{N,-i}}_t,\mu^\alpha_t \big)\bigg] + \E^{\P}\bigg[\sup_{t \in [0,T]} \big|\Xbb^i_t[(\overline{\alpha}^{-i},\kappa^{i,N})]-Z^{i,N}_t \big|^p \bigg] 
        \le C_N,
    \end{align*}
    where $Z^{i,N}$ denotes the unique strong solution of
    \begin{align*}
		\mathrm{d}Z^{i,N}_t
		= 
		b \big(t, Z^{i,N}_t,\mu^\alpha_{t \wedge \cdot}, \mub^{\alpha}_t, \kappa^{i,N}_t \big) \mathrm{d}t
		+
		\sigma \big(t, Z^{i,N}_t,\mu^\alpha_{t \wedge \cdot}, \mub^{\alpha}_t, \kappa^{i,N}_t \big) \mathrm{d} W^i_t
		+ 
		\sigma_0 \mathrm{d}B_t\;\mbox{with}\;Z^{i,N}_0=\xi^i.
	\end{align*}
	Therefore, $\Limsup_{N \to \infty} \Wc_p \big(Q^N, \widehat{Q}^N \big)=0,$ where $Q^N:=\frac{1}{N} \sum_{i=1}^N\Lc^{\P} \Big( \Xbb^i[\overline{\alpha}^{N,-i}], \varphi^{N,\Xbb,\overline{\alpha}^{N,-i}}, \delta_{\big(\kappa^{i,N}_t,\varphi^{N,\overline{\alpha}^{N,-i}}_t \big)}(\mathrm{d}u,\mathrm{d}m)\mathrm{d}t \Big),$
    and $\widehat{Q}^N:=\frac{1}{N} \sum_{i=1}^N\Lc^{\P} \Big( Z^{i,N}, \mu^\alpha, \delta_{\big(\kappa^{i,N}_t,\mub^\alpha_t \big)}(\mathrm{d}u,\mathrm{d}m)\mathrm{d}t \Big).$
	
\medskip	
    Thanks to this result and some techniques used in proof of \Cref{lemm:convergenceNashEquilibrium} (with the separability condition), one has
    \begin{align*}
        \Lim_{N \to \infty} \Bigg|\frac{1}{N} \sum_{i=1}^N
        &\E^{\P} \bigg[
        \int_0^T L\big(t,\Xbb^{i}_t[\overline{\alpha}^{N,-i}],\varphi^{N,\Xbb,\overline{\alpha}^{N,-i}},\varphi^{N,\overline{\alpha}^{N,-i}}_{t} ,\kappa^{i,N}_t \big) dt 
        + 
        g \big( \Xbb^{i}_T[\overline{\alpha}^{N,-i}], \varphi^{N,\Xbb,\overline{\alpha}^{N,-i}} \big)
        \bigg]
        \\
        &~~~~~~~~~~~~~~~~~~-\frac{1}{N} \sum_{i=1}^N
        \E^{\P} \bigg[
				\int_0^T L(t, Z^{i,N}_t,\mu^\alpha_{t \wedge \cdot}, \mub^{\alpha}_t,\kappa^{i,N}_t) \mathrm{d}t 
				+ 
				g(Z^{i,N}_T, \mu^{\alpha}_T)
			\bigg]
		\Bigg|
		=
		0.
    \end{align*}
    We define the sequence $(\Pr^{i,N})_{(i,N) \in \{1,\dots,N\} \x \N^*} \subset \Pcb_S$ by 
    \begin{align*}
        \Pr^{i,N}
        :=
        \P \circ \big(\mu^{i,N},\mu^\alpha, \delta_{\mub^{i,N}_t}(\mathrm{d}m)\mathrm{d}t, \delta_{\mub^{\alpha}_t}(\mathrm{d}m)\mathrm{d}t, B \big)^{-1}
    \end{align*}
    where $\mu^{i,N}_t:=\Lc^{\P}\big(Z^{i,N}_t \big| \Gc_t \big)$ and $\mub^{i,N}_t:=\Lc^{\P}\big(Z^{i,N}_t,\kappa^{i,N}_t \big| \Gc_t \big)$ a.s. for all $t \in [0,T].$ $(\Pr^{i,N})_{(i,N) \in \{1,\dots,N\} \x \N^*} \subset \Pcb_S$ is the sequence we are looking for.
    
\end{proof}

\subsubsection{Proof of Theorem \ref{thm:converselimitNash} (Converse Limit Theorem)}

\paragraph*{First point (i)} 
    Let $\mathrm{P} \in \Pcb_V^\star[\varepsilon]$ be an $\varepsilon$--measure--valued solution. By \Cref{prop:converse_general}, first, there exists a sequence $(\alpha^k)_{k \in \N^*}\subset \Ac$ s.t.
    \begin{align*}
        \Lim_{k \to \infty} \Pr^{\alpha^k}
        =
        \mathrm{P}\;\mbox{in}\; \Wc_p.
    \end{align*}
    Let us introduce, for each $\alpha' \in \Ac,$ 
    \begin{align*}
        \varepsilon^{k}
        :=
        \sup_{\alpha' \in \Ac}
			\Psi(\alpha^k,\alpha')
			-
			\E^{\mathrm{P}^{\alpha^k}} \big[J(\mu',\mu, \Lambda', \Lambda) \big].
    \end{align*}
        
    \medskip
    Remark that $\varepsilon^{k} \ge 0$ for all $k.$ There exists a sequence $(\gamma^k)_{k \in \N}\subset \Ac$ verifying $$
        \Psi(\alpha^k,\gamma^k)-\E^{\mathrm{P}^{\alpha^k}} \big[J(\mu',\mu, \Lambda', \Lambda) \big] \ge \varepsilon^{k}-2^{-k}.
    $$
    By the second part of \Cref{prop:converse_general}, there exists a sequence $(\mathrm{Q}^k)_{k \in \N^*}\subset \Pcb_V$ satisfying $\mathrm{Q}^k \circ \big(\mu,\Lambda,B \big)^{-1}=\mathrm{P} \circ \big(\mu,\Lambda,B \big)^{-1}$ for each $k \in \N^*$ and $\Limsup_{k \to \infty} \big|\Psi(\alpha^k,\gamma^k) - \E^{\mathrm{Q}^k} \big[J(\mu',\mu, \Lambda', \Lambda) \big]\big|=0.$
    Then, as $\mathrm{P}$ is an {\color{black}$\varepsilon$--measure--valued} MFG solution, 
    \begin{align*}
        \varepsilon \ge \Limsup_{k \to \infty} \E^{\mathrm{Q}^k} \big[J(\mu',\mu, \Lambda', \Lambda) \big] - \E^{\mathrm{P}} \big[J(\mu',\mu, \Lambda', \Lambda)\big]
        &\ge
        \Limsup_{k \to \infty} \Psi(\alpha^k,\gamma^k)-\E^{\mathrm{P}} \big[J(\mu',\mu, \Lambda', \Lambda) \big]
        \ge
        \Limsup_{k \to \infty}\varepsilon^{k}.
    \end{align*}
    Then $\Limsup_{k \to \infty} \varepsilon^{k} \in [0,\varepsilon],$ and
    \begin{align*}
        \E^{\mathrm{P}^{\alpha^k}} \big[J(\mu',\mu, \Lambda', \Lambda) \big]
        \ge
        \sup_{\alpha' \in \Ac} \Phi(\alpha^{k},\alpha')- \varepsilon^{k},\;\;\mbox{for each}\;k \in \N^*.
    \end{align*}
    We can conclude.
    
    \paragraph*{Second point (ii)} Let $\varepsilon \in [0,\infty)$ and $\mathrm{P}^{\alpha} \in \Pcb_S[\varepsilon]$ with $\alpha \in \Ac.$ By \Cref{lemma:converse_convergence}, there exists a sequence $(\alpha^{i,N})_{(i,N) \in \{1,\cdots,N\} \x \N^*}$ such that $\alpha^{i,N} \in \Ac^N,$  and
    \begin{align*}
        \Lim_{N \to \infty} \P \circ \Big( \varphi^{N,\Xbb,\overline{\alpha}^N}, \delta_{\varphi^{N,\overline{\alpha}^N}_t}(\mathrm{d}m)\mathrm{d}t, B \Big)^{-1}
        =
       \P \Big( \mu^{\alpha}, \delta_{\mub^\alpha_t}(\mathrm{d}m)\mathrm{d}t, B \Big)^{-1}
       =
       \Pr ( \mu, \Lambda, B )^{-1},
    \end{align*}
    where $\overline{\alpha}^N=(\alpha^{i,N},\dots,\alpha^{i,N}).$ 
    
\medskip    
    In addition, for any sequence $(\kappa^{i,N})_{(i,N) \in \{1,\dots,N\} \x \N^*}$ satisfying for each $N \in \N^*,$ $(\kappa^{i,N})_{i \in \{1,\dots,N\}} \subset \Ac^N,$ there exists a sequence $(\Pr^{i,N})_{(i,N) \in \{1,\dots,N\} \x \N^*} \subset \Pcb_S$ such that $\Pr^{i,N} \circ \big( \mu,\Lambda,B \big)^{-1}=\P \circ \Big( \mu^\alpha, \delta_{\mub^\alpha_t}(\mathrm{d}m)\mathrm{d}t, B \Big)^{-1}$ for each $(i,N)$ and
    \begin{align*}
        \Lim_{N \to \infty} \bigg|\frac{1}{N} \sum_{i=1}^{N} J_i\big( (\overline{\alpha}^{[-i]},\kappa^{i,N})\big)- \frac{1}{N} \sum_{i=1}^{N} \E^{\mathrm{P}^{i,N}}\big[J(\mu',\mu, \Lambda', \Lambda) \big] \bigg|
        =
        0.
    \end{align*}
    Define
    \begin{align*}
        c^{i,N}
        :=
        \sup_{\alpha' \in \Ac^N}
			J_i\big[( (\overline{\alpha}^{N})^{-i},\alpha')\big]
			-
			J_i[\overline{\alpha}^{N}].
    \end{align*}
    There exists a sequence of controls $(\kappa^{i,N})_{(i,N) \in \{1,\dots,N\} \x \N^*}$ satisfying $J_i\big[( (\overline{\alpha}^{N})^{-i},\kappa^{i,N})\big]-J_i[\overline{\alpha}^{N}]\ge c^{i,N}-2^{-N},$ for each $i \in \{1,...,N\}.$ Therefore, as $\mathrm{P} \in \Pcb^\star_S[\varepsilon]$ i.e. a $\varepsilon$--strong MFG solution,
    \begin{align*}
        \varepsilon
        &\ge
        \Big(\Limsup_{N \to \infty}
        \frac{1}{N} \sum_{i=1}^{N} \E^{\mathrm{P}^{i,N}}\big[J(\mu',\mu, \Lambda', \Lambda) \big]
        -
        \E^{\mathrm{P}}\big[J(\mu',\mu, \Lambda', \Lambda) \big] \Big)
        \\
        &\ge 
        \Limsup_{N \to \infty} 
        \Big(\frac{1}{N} \sum_{i=1}^{N} J_i\big[( (\overline{\alpha}^{N})^{-i},\kappa^{i,N})\big]
        -
        \frac{1}{N} \sum_{i=1}^{N} J_i[\overline{\alpha}^{N}] \Big)
        \ge
        \Limsup_{N \to \infty}
        \frac{1}{N} \sum_{i=1}^{N} c^{i,N}.
    \end{align*}
    By combining this result with the first point (see proof above), we can conclude.

\section{Proof of Theorem \ref{thm:existenceMFG} (Existence)}
\label{sec:proofThm_existence}

\subsection{Measure--valued no common noise MFG equilibrium} 

In this part, we discuss of the case without common noise. Let  $\sigma_0=0$ (or $\ell=0$) and $p'>p$ is fixed. 
In order to {\color{black}prove} our theorem, a more adequate framework and other definitions are necessary. Let us introduce the notion of  deterministic measure--valued no common noise control rule 

\begin{definition}
        Given $(\nb,\qb) \in \Cc^{n,p}_{\Wc} \x \M(\Pc_p(\R^n \x U)),$ $(\nb', \qb') \in \Cc^{n}_\Wc \x \M(\Pc^n_U)$ is a  deterministic measure--valued no common noise control rule if $($recall that $N_t$ is defined in {\rm\Cref{eq:FP-equation}}$)$,
    \begin{itemize}
        
        \item 
        $\nb'_0=\nu,$ and
        $N_t[\nb',\nb,\qb',\qb](f)=0$ for all $f \in C^{2}_b(\R^n)$ and every $t \in [0,T]$ .
        
        \item For $\mathrm{d}t$ almost every $t \in [0,T]$, $ \qb'_t\big(\Z_{\nb'_t} \big)=1.$
    \end{itemize}
\end{definition}
$\Rc(\nb,\qb)$ will denote the set of all deterministic measure--valued no common noise control rules defined as previously. We also consider 
\begin{align*}
    \Rc^\star(\nb,\qb)
    :=
    \arg\max_{(\nb',\qb') \in \Rc(\nb,\qb)} J\big(\nb',\nb, \qb',\qb \big),
\end{align*}
where we recall that
\begin{align*}
    J\big(\nb',\nb, \qb',\qb \big)
        :=
        \int_0^T \bigg[\int_{\Pc^n_U}\langle L^\circ\big(t,\cdot,\nb,\cdot \big), m \rangle \qb'_t(\mathrm{d}m) +
        \int_{\Pc^n_U}\langle L^\star\big(t,\cdot,\pi,m \big), \nb'_t \rangle \qb_t(\mathrm{d}m)
        \bigg]
        \mathrm{d}t 
        +
        \langle g(\cdot,\nb),\nb'_T \rangle.
\end{align*}
Notice that by \cite[Lemma 5.2]{MFD-2020}, $\Rc(\nb,\qb) \subset \Cc^{n,p}_{\Wc} \x \M(\Pc_p(\R^n \x U)),$ and the set $\Rc^\star(\nb,\qb)$ is nonempty.

\begin{definition} 
    $(\nb^\star,\qb^\star) \in \Cc^{n,p}_\Wc \x \M(\Pc_p(\R^n \x U))$ is a  deterministic measure--valued no common noise MFG solution if $(\nb^\star,\qb^\star) \in \Rc^\star(\nb^\star,\qb^\star).$  We shall denote $\Sc^\star$ all deterministic measure--valued no common noise MFG solutions.
\end{definition}

Mention that in the following, it will be more convenient to look at $\Rc$ as a set valued function: 
\[
    \Rc: (\nb,\qb) \in \Cc^{n,p}_{\Wc} \x \M(\Pc_p(\R^n \x U)) \to \Rc(\nb,\qb) \subset \Cc^{n,p}_{\Wc} \x \M(\Pc_p(\R^n \x U)).
\]

\paragraph*{Continuity of $\Rc$} In the next propositions, it is shown that $\Rc$ is both upper and lower hemicontinuous, and this is enough to conclude that $\Rc$ is continuous. We refer to \cite[chapter 17]{aliprantis2006infinite} for an overview on set valued functions.

\begin{lemma}{\rm(\cite[Lemma 5.2]{MFD-2020})}
\label{lemma:estimation_deterministic}
    There exists a constant $C>0$ (depend only of coefficients $[\sigma,b]$ and $\nu$) such that for any $(\nb,\qb) \in \Cc^{n,p}_\Wc \x \M(\Pc_p(\R^n \x U)),$ and $(\nb',\qb') \in \Rc(\nb,\qb),$ one has
    \begin{align*}
        \sup_{t \in [0,T]} \int_{\R^n} |x|^{p'} \nb'_t(\mathrm{d}x) \le C \Big(1 + \int_{\R^n} |y|^{p'} \nu(\mathrm{d}y) \Big).
    \end{align*}
    Furthermore, for any $(t,s) \in [0,T] \x [0,T],$ one gets $\Wc_p \big(\nb'_t, \nb'_s \big)^p \le C|t-s|.$
\end{lemma}

\begin{proposition}{\rm(Upper Hemicontinuity)}
    Let $(\nb,\qb) \in \Cc^{n,p}_\Wc \x \M(\Pc_p(\R^n \x U)).$ $\Rc(\nb,\qb)$ is a compact set of $\Cc^{n,p}_\Wc \x \M(\Pc_p(\R^n \x U))$. In addition for any sequence $(\nb^k,\qb^k)_{k \in \N^*} \subset \Cc^{n,p}_\Wc \x \M(\Pc_p(\R^n \x U))$ such that $\Lim_{k \to \infty} (\nb^k,\qb^k)=(\nb,\qb),$ and $(\nb'^k,\qb'^k) \in \Rc(\nb^k,\qb^k)$ for each $k \in \N^*,$ then $(\nb'^k,\qb'^k)_{k \in \N^*}$ is relatively compact and each limit point belongs to $\Rc(\nb,\qb).$
\end{proposition}

\begin{proof}
    By Lemma \ref{lemma:estimation_deterministic}, one finds 
    \begin{align*}
        \sup_{(\nb',\qb') \in \Rc(\nb,\qb)}\sup_{t \in [0,T]} \int_{\R^n} |x|^{p'} \nb'_t(\mathrm{d}x) < \infty\;\mbox{and}\;\Lim_{\delta \to 0}\sup_{(\nb',\qb') \in \Rc(\nb,\qb)} \sup_{t \in [0,T]}\Wc_p \big(\nb'_t, \nb'_{(t+\delta) \wedge T} \big)=0,
    \end{align*}
    as $U$ is a compact set and that for $\mathrm{d}t$ almost every $t \in [0,T]$, $ \qb'_t\big(\Z_{\nb'_t} \big)=1,$ one has 
    \[
        \sup_{(\nb',\qb') \in \Rc(\nb,\qb)}\int_0^T \int_{\Pc_p(\R^n \x U)} \Wc_p\big(m,m_0 \big)^{p'} \qb'_t(\mathrm{d}m)\mathrm{d}t < \infty,\;\mbox{for any}\;m_0 \in \Pc_p(\R^n \x U).
    \]
    Then by {\color{black}Aldous} criterion \cite[Lemma 16.12]{kallenberg2002foundations}, $\Rc(\nb,\qb)$ is a compact set of $\Cc^{n,p}_\Wc \x \M(\Pc_p(\R^n \x U)).$

\medskip    
    By similar way, the sequence $(\nb'^k,\qb'^k)_{k \in \N^*}$ is relatively compact. By passing to the limit in the equation verified by $(\nb'^k, \qb'^k,\nb^k,\qb^k)$ i.e.  $N_t[\nb'^k,\nb^k , \qb'^k,\qb^k](f)=0,$ for each $(t,f) \in [0,T] \x C^{2}_b(\R^n)$ (see for instance \cite[Lemma 4.1]{MFD-2020}),  it is straightforward to check that each limit belongs to $\Rc(\nb,\qb)$ (see \cite[Proposition 5.4]{MFD-2020}).
\end{proof}

\begin{proposition}{\rm (Lower Hemicontinuity)}
    Let $(\nb,\qb) \in \Cc^{n,p}_\Wc \x \M(\Pc_p(\R^n \x U))$ and  $(\nb^k,\qb^k)_{k \in \N^*}$ be a sequence of elements of $\Cc^{n,p}_\Wc \x \M(\Pc_p(\R^n \x U))$ such that $\Lim_{k \to \infty} (\nb^k,\qb^k)=(\nb,\qb),$ and $(\nb',\qb') \in \Rc(\nb,\qb).$ There exists $(\nb'^j, \qb'^j) \in \Rc(\nb^{k_j},\qb^{k_j}),$ for each $j \in \N^*$ where $(k_j)_{j \in \N^*} \subset \N^*$ is a sub--sequence with $\Lim_{j \to \infty} (\nb'^j,\qb'^j)=(\nb', \qb').$
\end{proposition}

\begin{proof}
    As $\Lim_{k \to \infty} \big(\nb^{k},\qb^{k},\qb' \big)=(\nb,\qb,\qb'),$ by \Cref{lemma:converse_general} and/or \cite[Proposition 4.10]{MFD-2020},
    there exists $(\nb'^j, \qb'^j) \in \Rc(\nb^{k_j},\qb^{k_j})$ for each $j \in \N^*$ where $(k_j)_{j \in \N^*}$ is a sub--sequence with $\Lim_{j \to \infty} (\nb'^j, \qb'^j)=(\nb', \qb').$
\end{proof}

\begin{theorem}
\label{thm:existenceMFG-nocommon}
    The set $\Sc^\star$ is nonempty and compact.
\end{theorem}

\begin{proof}    
Under \Cref{assum:main1}, it straightforward to verify that $ J: \Cc^{n,p}_{\Wc} \x \Cc^{n,p}_{\Wc} \x \M(\Pc_p(\R^n \x U)) \x \M(\Pc_p(\R^n \x U)) \to \R$ is continuous. The set valued function $\Rc$ is continuous because it is upper and lower hemicontinuous. Besides $\Rc$ has nonempty compact convex values. By Berge Maximum theorem \cite[Theorem 17.31]{aliprantis2006infinite}, $\Rc^\star$ has nonempty compact convex values and is upper hemicontinuous. Consequently its graph $\mathrm{Gr}(\Rc^\star):=\big\{\big(\nb,\qb, \widetilde{\nb},\widetilde{\qb} \big): (\widetilde{\nb},\widetilde{\qb}) \in \Rc^\star(\nb,\qb) \big\}$ is closed. Let $(\nb,\qb) \in \Cc^{n,p}_{\Wc} \x \M(\Pc_p(\R^n \x U)),$ we say that $(\nb,\qb) \in K$ if $(\nb,\qb) \in \Cc^{n,p}_{\Wc} \x \M(\Pc_p(\R^n \x U))$ and 
    \begin{align*}
        \sup_{t \in [0,T]} \int_{\R^n} |x|^{p'} \nb_t(\mathrm{d}x) + \int_0^T \int_{\Pc_p(\R^n \x U)} \Wc_{p'}\big(m,m_0 \big)^{p'} \qb_t(\mathrm{d}m)\mathrm{d}t\;\; \le M,
    \end{align*}
    where $m_0$ is an element of $\Pc_{p'}(\R^n \x U)$ and $M < \infty$ is defined by
    \[
        M:=\sup \bigg\{ \int_0^T \int_{\Pc_p(\R^n \x U)} \Wc_p\big(m,m_0 \big)^{p'} \qb'_t(\mathrm{d}m)\mathrm{d}t + C \Big(1 + \int_{\R^n} |x|^{p'} \nu(\mathrm{d}x) \Big),\;(\nb',\qb') \in \Rc^\star\big(\Cc^{n,p}_{\Wc} \x \M(\Pc_p(\R^n \x U)) \big)  \bigg\},
    \]
    and in addition
    \begin{align*}
        \Wc_p \big(\nb_t, \nb_s \big)^p \le C|t-s|,\;\mbox{for all}\;(t,s) \in [0,T] \x [0,T].
    \end{align*}
    Thanks to the above techniques, it is obvious that $K$ is a compact set of $\Cc^{n,p}_{\Wc} \x \M(\Pc_p(\R^n \x U)),$ and $\Rc^\star$ is a set valued function of $K$ into himself i.e. $\Rc^\star:(\nb,\qb) \in K \to \Rc^\star(\nb,\qb) \subset K.$
    
\medskip    
    Let $E$ be a Polish space, {\color{black}denote by $\Mc(E)$ the set of signed measure on $E.$} Equipped of the weak convergence topology $\tau_\om:=\sigma\big(\Mc(E), C_b(E) \big)$ generated by the bounded continuous function, $\Mc(E)$ is a locally convex Hausdorff space. Accordingly, $C([0,T];\Mc(\R^n))$ is a locally convex Hausdorff space. Likewise, $\Mc(\Pc^n_U \x [0,T])$ is a locally convex Hausdorff space equipped of $\tau_\om:=\sigma\big(\Mc(\Pc^n_U \x [0,T]), C_b(\Pc^n_U \x [0,T]) \big).$ Then $C([0,T];\Mc(\R^n)) \x \Mc(\Pc^n_U\x [0,T])$ is a locally convex Hausdorff space. One can see $K$ as a subset of $C([0,T];\Mc(\R^n)) \x \Mc(\Pc^n_U \x [0,T]).$ As the topology of $C([0,T];\Mc(\R^n)) \x \Mc(\Pc^n_U) \x [0,T])$ induced on $K$ is equivalent to the topology of $\Cc^n_\Wc \x \M(\Pc(\R^n \x U)),$ we deduce that $K$, which is a compact set of $\Cc^{n,p}_\Wc \x \M(\Pc_p(\R^n \x U)) (\subset \Cc^n_\Wc \x \M(\Pc^n_U))$, is also a compact set of $C([0,T];\Mc(\R^n)) \x \Mc(\Pc(\R^n \x U) \x [0,T]).$ To conclude, we apply the fixed point theorem of Kakutani-–Fan-–Glicksberg (see \cite[Corollary 17.55]{aliprantis2006infinite}) to deduce $\Sc^\star$ is nonempty and compact. Therefore we can find $(\nb^\star,\qb^\star) \in \Rc^\star(\nb^\star,\qb^\star).$
    \end{proof}
    
    \subsection{Proof of existence of non--random measure--valued MFG solution} Now, let us prove the main result of this part.
    If $\mathrm{P}^\star(\mathrm{d}\pi',\mathrm{d}\pi,\mathrm{d}q',\mathrm{d}q,\mathrm{d}\bb):=\delta_{(\nb^\star,\nb^\star,\qb^\star,\qb^\star)}(\mathrm{d}\pi',\mathrm{d}\pi,\mathrm{d}q',\mathrm{d}q,\mathrm{d}\bb) P_B(\mathrm{d}\bb) \in \Pc(\Omb),$ it is straightforward to check that $\mathrm{P}^\star$ is a non--random measure--valued MFG solution where $P_B$ is the $\R^\ell$ Wiener measure. This result proves the point $(i)$ of Theorem \ref{thm:existenceMFG}. The point $(ii)$ is just an easy combination of \Cref{prop:converse_general} and techniques used in proof in theorem \ref{thm:converselimitNash} (by using the fact that $(\nb^\star,\qb^\star)$ are deterministic or equivalently $(\mu,\Lambda)$ are not random).

\bibliographystyle{plain}
\bibliography{Extended_mean-field_games_revision_V1}

\begin{appendix}
    \section{Proof of equivalence between measure--valued MFG solution and weak MFG solution of \cite[Definition 3.1]{carmona2014mean}}
    
    \subsection{Weak MFG solution as a particular case of measure--valued MFG solution} \label{appendix:equivalence1}
    
    \begin{proof}
        Let $(\Omt,\Fct,\Ft,\widetilde{\P},\Wt,\Bt,\mubb,\widetilde{\Lambda},\Xt)$ be a weak MFG solution (see \Cref{def_carmona} or \cite[Definition 3.1]{carmona2014mean}). Let $f \in C^2_b(\R^n),$ by applying It\^{o}'s formula and taking the conditional expectation w.r.t $\sigma\{\Bt,\mubb\},$ one gets
    \begin{align*} 
        \mathrm{d}\langle f(\cdot-\sigma_0 \Bt_t),\mubb^x_t \rangle
        =
	    &\E^{\Pt} \bigg[ \nabla f(\Xt_t-\sigma_0 \Bt_t)^\top \int_U b(t,\Xt_t,\mubb^x_t,u) \widetilde{\Lambda}_t(\mathrm{d}u) \Big| \Bt,\mubb \bigg] \mathrm{d}t + \frac{1}{2} \langle \mathrm{Tr}[\sigma \sigma^\top(t,\cdot,\mubb^x_t) \nabla^2f(\cdot-\sigma_0\Bt_t)], \mubb^x_t \rangle \mathrm{d}t.
    \end{align*}
    We define $\overline{\mubb}_t(\mathrm{d}x,\mathrm{d}u):=\E^{\Pt}\big[\delta_{\Xt_t}(\mathrm{d}x)\widetilde{\Lambda}_t(\mathrm{d}u) | \Bt,\mubb \big].$ Therefore, one can check that $\Pr \in \Pcb_V$ where
    \begin{align*}
        \Pr
        :=
        \Pt \circ \big( \mubb^x, \mubb^x, \delta_{\overline{\mubb}_t}(\mathrm{d}m)\mathrm{d}t, \delta_{\overline{\mubb}_t}(\mathrm{d}m)\mathrm{d}t, B \big)^{-1}.
    \end{align*}
    Now, we show that $\Pr \in \Pcb^\star_V$ i.e. the optimality condition is satisfied. Let $\Qr \in \Pcb_V$ s.t. $\Lc^{\Qr}(\mu,\Lambda,B)=\Lc^{\Pr}(\mu,\Lambda,B).$ Let $(\Omb',\Fb',\Qr')$ be an extension  of $(\Omb,\Fb,\Qr)$ supporting a $\R^n$--valued $\widetilde{\Fb}$--Brownian motion $W',$ a $\R^n$--valued $\Fcb'_0$--measurable variable $X'_0$ s.t. $\Lc^{\Qr'}(X'_0)=\nu,$ and $F'$ a uniform random variable. Besides $W',$ $X'_0,$ $F'$ and $\Gcb_T$ are independent. By \cite[Proposition 4.10]{MFD-2020} (see also \cite[Proposition 4.9]{MFD-2020}), there exists $(\beta'^{k})_{k\in \N^*}$ a sequence of $\Pc(U)$--valued $\big((\sigma\{ F',X'_0,W'_{t \wedge \cdot} \} \lor \Gcb_t) \otimes \Bc(\Pc^n_U) \big)_{t \in [0,T]}$--predictable processes such that if $(X'^k_t)_{t \in [0,T]}$ is the unique strong solution of: 
    \begin{align*}
        \mathrm{d}X'^k_t
        =
        &
        \int_U b \big(t,X'^k_t,\mu_t,u \big) \int_{\Pc^n_U} \beta'^{k}_{t}(m)(\mathrm{d}u) \Lambda'_t(\mathrm{d}m) \mathrm{d}t
        +
        \sigma \big(t,X'^k_t,\mu_t \big)\mathrm{d}W'_t + \sigma_0 \mathrm{d}B_t,
    \end{align*}
    then, one has, for a sub--sequence $(k_j)_{j \in \N^*} \subset \N^*,$ 
    \begin{align*}
        \Lim_{j \to \infty} \E^{\Qr'} \bigg[ \int_0^T \Wc_{p} \big(\mbb'^{k_j}_t[m],m \big) \Lambda'_t(\mathrm{d}m) \mathrm{d}t\bigg]=0\;\mbox{and}\; \Lim_{j \to \infty} \E^{\Qr'} \bigg[\sup_{s \in [0,T]}  \Wc_p ( \mu'^{k_j}_t,\mu'_t ) \bigg]=0,
    \end{align*}
    where
    $$
        \mbb'^{k}_t[m](\mathrm{d}x,\mathrm{d}u)
        :=
        \E^{\Qr'}
        \Big[
        \beta'^{k}_t(m)(\mathrm{d}u) \delta_{X'^{k}_t}(\mathrm{d}x) \Big|\Gcb_t
        \Big]
        ~
        \mbox{and}
        ~
        \mu'^{k}_t:=\Lc^{\Qr'}(X'^{k}_t\big|\Gcb_t)
        ~~
        \mbox{for all}
        ~
        t \in [0,T].
    $$
    Therefore
    \begin{align*}
    \Lim_{k \to \infty}
        \E^{\Qr'}\bigg[
				\int_0^T \int_U L(t, X'^k_t,\mu_{t}, u) \int_{\Pc^n_U}\beta'^k_t(m)(\mathrm{d}u) \Lambda'_t(\mathrm{dm}) \mathrm{d}t 
				+ 
				g(X'^k_T, \mu_T) 
			\bigg]
			=
			\E^{\Qr}[J(\mu',\mu,\Lambda',\Lambda)].
    \end{align*}
    
    Notice that, we can find a Borel function $\phi'^k: [0,T] \x [0,1] \x \R^n \x \Cc^n \x \Cc_{\Wc} \x \M(\Pc^n_U) \x \Cc^\ell \to \Pc(U)$ satisfying $\phi'^k(t,F',X'_0,W'_{t \wedge \cdot},\mu_{t \wedge \cdot},\Lambda_{t \wedge \cdot},B_{t \wedge \cdot})=\int_{\Pc^n_U}\beta'^k_t(m) \Lambda'_t(\mathrm{d}m)$ $\Qrt$--a.s. On the probability space $(\Omt,\Fct,\Ft,\Pt),$ we define $\widetilde{h}'^k_t:=\phi'^k(t,\widetilde{F},\Xt_0,\Wt_{t \wedge \cdot},\mubb^x_{t \wedge \cdot},\Rrt_{t \wedge \cdot},\Bt_{t \wedge \cdot}),$ where $\Rrt:=\delta_{\overline{\mubb}_t}(\mathrm{d}m)\mathrm{d}t$ and $\widetilde{F}$ is a uniform variable independent of $\Xt_0,$ $\Wt,$ $\Bt$ and $\mubb.$ Let $\Xt'^k$ be the solution of 
    \begin{align*}
        \mathrm{d}\Xt'^k_t
        =
        &
        \int_U b \big(t,\Xt'^k_t,\mubb^x_t,u \big)  \widetilde{h}'^k_t(\mathrm{d}u) \mathrm{d}t
        +
        \sigma \big(t,\Xt'^k_t,\mubb^x_t \big)\mathrm{d}\Wt_t + \sigma_0 \mathrm{d}B_t\;\mbox{with}\;\Xt'^k_0=\Xt_0.
    \end{align*}
    
    It is straightforward to check that $(\Omt,\Fct,\Fb,\widetilde{\P},\Wt,\Bt,\mubb,\widetilde{h}'^k,\Xt'^k)$ satisfies the point $(1$--$4)$ of \Cref{def_carmona} for all $k \in \N^*$. As $(\Omt,\Fct,\Ft,\widetilde{\P},\Wt,\Bt,\mubb,\widetilde{\Lambda},\Xt)$ is a weak MFG solution and
    \begin{align*}
        \Qrt \circ \Big(X'^k, \int_{\Pc^n_U}\beta'^k_t(m)(\mathrm{d}u) \Lambda'_t(\mathrm{d}m) \mathrm{d}t, W', \mu, B \Big)^{-1}
        =
        \Pt \circ \Big( \Xt'^k, \widetilde{h}'^k_t(\mathrm{d}u) \mathrm{d}t, \Wt, \mubb^x, \Bt   \Big)^{-1},
    \end{align*}
    then 
    \begin{align*}
        \E^{\Pr}[J(\mu',\mu,\Lambda',\Lambda)]
        &\ge
        \Lim_{k \to \infty}
        \E^{\Pt}\bigg[
				\int_0^T \int_U L(t, \Xt'^k_t,\mubb^x_{t}, u) \widetilde{h}'^k_t(\mathrm{d}u) \mathrm{d}t 
				+ 
				g(\Xt'^k_T, \mubb^x_T) 
			\bigg]
			\\
        &=
        \Lim_{k \to \infty}
        \E^{\Qr'}\bigg[
				\int_0^T \int_U L(t, X'^k_t,\mu_{t}, u) \int_{\Pc^n_U}\beta'^k_t(m)(\mathrm{d}u) \Lambda'_t(\mathrm{dm}) \mathrm{d}t 
				+ 
				g(X'^k_T, \mu_T) 
			\bigg]
			=
			\E^{\Qr}[J(\mu',\mu,\Lambda',\Lambda)].
    \end{align*}
    We can conclude that $\Pr$ belongs to $\Pcb_V^\star.$
    \end{proof}
     \subsection{Measure--valued as a weak MFG solution} \label{appendix:equivalence2}

    \begin{proof}
    
    Let $\Pr \in \Pcb_V^\star.$ By \cite[Theorem 1.3]{Lacker-Shkolnikov-Zhang_2020}, there exists $(\widetilde{\Omb},\widetilde{\Fb},\Prt)$ an extension of $(\Omb,\Fb,\Pr)$ supporting a Brownian motion $\Wt,$ a $\R^n$--valued $\widetilde{\Fb}$--adapted continuous process $\Xt$ s.t. $\Lc^{\Prt}(\Xt_0)=\nu.$ Besides $\Xt_0,$ $\Wt,$ and $\Gcb_T$ are independent, and
    
    \begin{enumerate}
        \item $\Xt$ satisfies
        \begin{align*}
            \mathrm{d}\Xt_t
            =
            \int_U b(t,\Xt_t,\mu_t,u) \int_{\Pc^n_U} m^{\Xt_t}(\mathrm{d}u) \Lambda_t(\mathrm{d}m)\mathrm{d}t
            +
            \sigma(t,\Xt_t,\mu_t)\mathrm{d}\Wt_t
            +
            \sigma_0 \mathrm{d}B_t\;\widetilde{\Pr}\mbox{--a.s.}
        \end{align*}
        where for each $m \in \Pc^n_U,$ the Borel function $x \in \R^n \to \Pc(U) \ni m^x$ satisfies $m^x(\mathrm{d}u)m(\mathrm{d}x,U)=m(\mathrm{d}x,\mathrm{d}u).$
        
        \item $\mu_t=\Lc^{\widetilde{\Pr}}(\Xt_t|\Gcb_t)=\Lc^{\widetilde{\Pr}}(\Xt_t|\Gcb_T)$ a.s. for each $t \in [0,T].$
        
        \item $X_{t \wedge \cdot}$ is conditionally independent of $\sigma\{\Wt_{T \wedge \cdot} \} \lor \Gcb_T$ given $\sigma\{\Wt_{t \wedge \cdot} \} \lor \Gcb_t$.
    \end{enumerate}
    We pose $\widetilde{\Lambda}_t(\mathrm{d}u)\mathrm{d}t:=\int_{\Pc^n_U}m^{\Xt_t}(\mathrm{d}u) \Lambda_t(\mathrm{d}m)\mathrm{d}t$ and $\mubb:=\Lc^{\Prt}(W,\widetilde{\Lambda},\Xt|\Gcb_T).$
    It follows that $(\widetilde{\Omb},\widetilde{\Fcb},\widetilde{\Fb},\Prt,\Wt,B,\mubb,\widetilde{\Lambda},\Xt)$ satisfies the point $(1$--$4)$ of \Cref{def_carmona}. Now, let $(\Omt',\Fct',\Ft',\widetilde{\P}',\Wt',\Bt',\mubb',\widetilde{\Lambda}',\Xt')$ be a tuple satisfying $(1$--$4)$ of \Cref{def_carmona} and $\widetilde{\P}'\circ \big(\Xt_0',\Wt',\mubb',\Bt' \big)^{-1}=\widetilde{\Pr}\circ \big(\Xt_0,\Wt,\mubb,B \big)^{-1}$. Let us define $\widetilde{\mubb}'^x_t:=\Lc^{\Pt'}(\Xt'_t|\Bt',\mubb'^x),$ $\widetilde{\overline{\mubb}}'_t:=\E^{\Pt'}[\delta_{\Xt'_t}(\mathrm{d}x)\widetilde{\Lambda}'_t(\mathrm{d}u)|\Bt',\mubb'^x]$ and $\overline{\mubb}'_t:=\E^{\mubb'}[\delta_{X_t}(\mathrm{d}x)\theta_t(\mathrm{d}u)] $ where $(X,\theta)$ is the canonical process on $(\Cc^n \x \M(U)).$ Then, we easily verify that 
    \begin{align*}
        \Qrt:=\Pt' \circ \big( \widetilde{\mubb}'^x, \mubb'^x, \delta_{\widetilde{\overline{\mubb}}'_t}(\mathrm{d}m)\mathrm{d}t, \delta_{\overline{\mubb}'_t}(\mathrm{d}m)\mathrm{d}t, \Bt' \big)^{-1} \in \Pcb_V\;\mbox{and}\;\Lc^{\Pr}(\mu,\Lambda,B)=\Lc^{\Qrt}(\mu,\Lambda,B).
    \end{align*}
    As $\Pr \in \Pcb_V^\star,$ we deduce that $(\widetilde{\Omb},\widetilde{\Fcb},\widetilde{\Fb},\Prt,\Wt,B,\mubb,\widetilde{\Lambda},\Xt)$ satisfies the optimality condition, then he is a weak MFG solution.

     \end{proof}

    \section{Sketch of proof of technical results borrowed from \cite{MFD-2020}}
    
    \begin{proof}[Sketch of proof of \Cref{Prop_borrowed1}] \label{proof_prop_borrowed1}
        Recall that in the simplified case of \Cref{Prop_borrowed1}, for a measure--valued control rule $\Pr \in \Pcb_V,$ the Fokker--Planck equation is rewritten: $\Pr$--a.s.
\begin{align*}
        \mathrm{d}\langle f,\mu'_t \rangle
	    =
	    &\int_{\Pc^n_U}  \int_{\R^n \x U} \frac{1}{2} \nabla^2f(x)u^2 m^x(\mathrm{d}u) \mu'_t(\mathrm{d}x) \Lambda'_t(\mathrm{d}m)\mathrm{d}t + \int_{\Pc^n_U}  \int_{\R^n} \frac{1}{2} \nabla^2f(x) \sigmat(m)^2 \mu'_t(\mathrm{d}x)\Lambda_t(\mathrm{d}m)\mathrm{d}t.
    \end{align*}
    
    For each $\delta >0,$ we define  $G_\delta(x):= \delta^{-1}G(\delta^{-1}x),$ $G(x):=(1+|x|^2)^{-1} \big(\int_{\R} (1+|x'|^2)^{-1} \mathrm{d}x' \big)^{-1}$ and (recall that $\Lambdap_t(\Z_{\mu'_t})=1$)
    \begin{align*}
        \overline{\sigma}_{\delta}(x,\Lambda'_t)^2
        :=
        \int_{\Pc^n_U} \int_{\R \x U} u^2 \frac{G_{\delta}(x-y) m(\mathrm{d}u,\mathrm{d}y)}{\int_{\R} G_{\delta}(x-y')m(\mathrm{d}y',U)} \Lambda'_t(\mathrm{d}m)
        =
        \int_{\Pc^n_U} \int_{\R \x U} u^2 \frac{G_{\delta}(x-y) m^y(\mathrm{d}u)\mu'_t(\mathrm{d}y)}{\int_{\R} G_{\delta}(x-y')\mu'_t(\mathrm{d}y')} \Lambda'_t(\mathrm{d}m) 
        .
    \end{align*}
    Notice that $ \R \ni x \to \overline{\sigma}_\delta(x,\Lambda'_t) \in \R$ is smooth (infinitely differentiable).  
    Also, given $(\Lambda',\mu',\Lambda),$ let us introduce $\mu'^\delta,$ the unique solution of 
    \begin{align} \label{eq:firstapp}
        \mathrm{d}\langle f,\mu'^\delta_t \rangle
	    =
	    & \int_{\R^n} \frac{1}{2} \nabla^2f(x)\overline{\sigma}_{\delta}(x,\Lambda'_t)^2 \mu'^\delta_t(\mathrm{d}x)\mathrm{d}t + \int_{\Pc^n_U}  \int_{\R^n} \frac{1}{2} \nabla^2f(x) \sigmat(m)^2 \mu'^\delta_t(\mathrm{d}x)\Lambda_t(\mathrm{d}m)\mathrm{d}t.
    \end{align}
    Under \Cref{assum:main1}, one has that (see  \cite[Lemma 4.2, Lemma 4.5]{MFD-2020} or \cite[Lemma 2.1, Proposition 4.3]{gyongy1986mimicking})
    \begin{align*}
        \Lim_{\delta \to 0} \sup_{t \in [0,T]} \Wc_p(\mu'_t,\mu'^\delta_t)=0\;\mbox{and}\;\Lim_{\delta \to 0}  \overline{\sigma}_{\delta}(x,\Lambda'_t)^2= \int_{\Pc^n_U} \int_{U} u^2 m^x(\mathrm{d}u) \Lambda'_t(\mathrm{d}m),\;\mu'_t(\mathrm{d}x)\mathrm{d}t\mbox{--a.e.}.
    \end{align*}
    For now, let us set $\delta >0.$ Let $(\Lambda'^k,\Lambda^k)_{k \in \N^*}$ be a sequence s.t. $\Lim_{k \to \infty} (\Lambda'^k,\Lambda^k) = (\Lambda',\Lambda)$ in $\Wc_p.$ Then, if we define $\mu'^{\delta,k}$ the unique solution of 
    \begin{align} \label{eq:secondapp}
        \mathrm{d}\langle f,\mu'^{\delta,k}_t \rangle
	    =
	    & \int_{\R^n} \frac{1}{2} \nabla^2f(x)\overline{\sigma}_{\delta}(x,\Lambda'^k_t)^2 \mu'^{\delta,k}_t(\mathrm{d}x)\mathrm{d}t + \int_{\Pc^n_U}  \int_{\R^n} \frac{1}{2} \nabla^2f(x) \sigmat(m)^2 \mu'^{\delta,k}_t(\mathrm{d}x)\Lambda^k_t(\mathrm{d}m)\mathrm{d}t,
    \end{align}
    using the uniqueness of \Cref{eq:firstapp} and the regularity of the coefficients, it is straightforward that: for each $\delta > 0,$ 
    \begin{align*}
        \Lim_{k \to \infty} \sup_{t \in [0,T]} \Wc_p(\mu'^\delta_t,\mu'^{\delta,k}_t)=0.
    \end{align*}
    We set now $k \in \N^*$ and $\delta >0.$ We use the argument mentioned in \eqref{eq:blackwell}. Recall that we place ourselves on an extension $(\Omt,\Ft:=(\Fct_t)_{t \in [0,T]},\widetilde{\Pr})$ of the probability space $(\Omb,\Fcb,\Pr)$ where we have a $\Ft$--Brownian motion $W,$ a $\Fct_0$--measurable random variable $X_0$ of law $\nu$ and a $F$ uniform variable. The variables $W,$ $X_0,$ $F$ and the $\sigma$-field $\Gcb_T$ are independent. 
    Let $Z'$ be a $\Ft$--adapted solution of: $Z'^{\delta,k}_0=X_0,$  

\begin{align}
    \mathrm{d}Z'^{\delta,k}_t
    =
    \sqrt{\alpha^{\delta,k}(t,Z'^{\delta,k}_t,F)^2 + \int_{\Pc^n_U}\sigmat(m)^2 \Lambda^k_t(\mathrm{d}m)} \mathrm{d}W_t\;\;\mbox{where}\;\;\alpha^{\delta,k}(t,x,F):=\Psi\bigg(\int_{\Pc^n_U} \frac{G_{\delta}(x-y) m(\mathrm{d}u,\mathrm{d}y)}{\int_{\R} G_{\delta}(x-y')m(\mathrm{d}y',U)} \Lambda'^k_t(\mathrm{d}m) \bigg)(F).
\end{align}

Using the uniqueness of \Cref{eq:secondapp}, we check that $\mu'^{\delta,k}_t=\Lc^{\Prt}(Z'^{\delta,k}_t | \Lambda'^k, \Lambda^k)=\Lc^{\Prt}(Z'^{\delta,k}_t | \Lambda'^k_{t \wedge \cdot}, \Lambda^k_{t \wedge \cdot})$ a.s. for all $t \in [0,T].$ Consequently, one has (see for instance \cite[Proposition 4.9, Proposition 4.10]{MFD-2020})
\begin{align*}
    \Lim_{\delta \to 0} \Lim_{k \to \infty} \sup_{t \in [0,T]} \Wc_p \big( \mu'_t, \Lc^{\Prt}(Z'^{\delta,k}_t | \Lambda'^k, \Lambda^k) \big)\;\mbox{and}\;\Lim_{\delta \to 0} \Lim_{k \to \infty} \delta_{\Lc^{\tilde \Pr}(Z'^{\delta,k}_t,\; \alpha^{\delta,k}(t,Z'^{\delta,k}_t,F) | \Lambda'^k, \Lambda^k)} (\mathrm{d}m) \mathrm{d}t= \Lambda'.
\end{align*}
    
    This is enough to conclude.
    
    \end{proof}
    
\end{appendix}

\end{document}